\documentclass[11pt,leqno]{article}
\usepackage{amssymb,amsmath,amsthm}

\newcommand{\n}{\noindent}

\newcommand{\vp}{\varepsilon}
\newcommand{\bb}[1]{\mathbb{#1}}
\newcommand{\cl}[1]{\mathcal{#1}}

\newcommand{\ovl}{\overline}

\theoremstyle{plain}
\newtheorem{lem}{Lemma}[section]

\newtheorem{pro}[lem]{Proposition}

\newtheorem{thm}[lem]{Theorem}
\newtheorem{cor}[lem]{Corollary}

\theoremstyle{definition}
\newtheorem{defn}[lem]{Definition}

\theoremstyle{remark}
\newtheorem{rem}[lem]{Remark}

\numberwithin{equation}{section}

\def\ie{{\it i.e. \ }}
\def\EE{\bb E}
\def\N{\bb N}
\def\F{\bb F}
\def\C{\bb C}
\begin{document}

\title{Martingale inequalities and Operator space structures on $L_p$}

\author{by\\
Gilles Pisier\\
Texas A\&M University\\
College Station, TX 77843, U. S. A.\\
and\\
Universit\'e Paris VI\\
Equipe d'Analyse, Case 186, 75252\\
Paris Cedex 05, France}

\maketitle

\tableofcontents

\section*{Introduction}

\indent 
In probability theory and harmonic analysis, the classical  inequalities for martingales in $L_p$
due to Don Burkholder, and also to Richard Gundy and Burgess Davis have had
an invaluable impact, with multiple interaction with other fields. See \cite{Bur} for a recent review.\\ 
In \cite{PX1}, a non-commutative version of Burkholder's martingale inequalities is given. This is valid in any non-commutative $L_p$-space say $L_p(\tau)$ (associated to a finite trace $\tau$ on a von Neumann algebra) for any $1<p<\infty$. In particular this applies to martingales of the form $f_n = \sum^n_1 \vp_k\otimes x_k$ where $x_k\in L_p(\tau)$ and $(\vp_n)$ is a standard random choice of signs $\vp_n = \pm 1$ (equivalently we can think of $(\vp_n)$ as the Rademacher functions on [0,1]).
In that case, Burkholder's inequality reduces to Khintchine's inequality, for which the non-commutative case is due to Lust--Piquard (\cite{LP}).

In the classical setting, Khintchine's inequality expresses the fact that the closed span in $L_p$ of $\{\vp_n\}$ is isomorphic to $\ell_2$ (as a Banach space). If $\{\vp_n\}$ is replaced 
by a sequence $(g_n)$ of independent standard Gaussian
random variables, the span in $L_p$ becomes isometric to $\ell_2$.  In the recently developed theory of operator spaces, Lust--Piquard's non-commutative Khintchine inequalities can be interpreted (see \cite[p. 108]{P4}) as saying that the span in $L_p$ of $[\vp_n]$ (or $(g_n)$)  is completely isomorphic to a Hilbertian operator space that we will denote here by ${\cl K}{\cl H}_p$. The precise description of ${\cl K}{\cl H}_p$ is not important for this paper, but for reference let us say merely that, for $2<p<\infty$ (resp. $1\le p<2$),
 ${\cl K}{\cl H}_p$ has the structure of intersection (resp. sum) of row and column spaces in the Schatten class $S_p$. For the non-commutative Burkholder inequalities from \cite{PX1}, the situation is analogous:\ the relevant square function combines the two cases of ``row'' and ``column'' in analogy with
  the definition of ${\cl K}{\cl H}_p$. 

This result was a bit of a disappointment because there is a canonical notion of ``operator Hilbert space,'' namely the space $OH$ from \cite{P3} and one would have expected in closer analogy to the classical case, that the span in $L_p$ of $[\vp_n]$ (or $(g_n)$) should be completely isomorphic to $OH$. In the preceding references, the spaces $L_p$ (commutative or not) are always equipped with what we call their ``natural'' operator space structure defined using complex interpolation. Then the space ${\cl K}{\cl H}_p$ is
completely isomorphic to $OH$ only when $p=2$.

In the present paper, we take a different route. We will equip $L_p$ with another o.s.s., hopefully still rather natural, but limited to $p$ equal to an even integer (for some of our results we even assume $p=2^k$). Roughly we imitate the classical idea that $f\in L_4$ iff $|f|^2\in L_2$ in order  to define our new o.s.s.\ on $L_4$ and then we iterate the process to define the same for $L_6,L_{8}$ and so on. Thus   given  the  space $L_p(\Omega,\mu)$ we can associate to it (assuming $p\in 2{\bb N}$ and $p\ge 2$) an operator space that we denote by $\Lambda_p(\Omega,\mu)$ that is isometric to the original $L_p(\Omega,\mu)$ as a Banach space.

It turns out that with this new structure a quite different picture emerges for  Khintchine's (or more generally Burkholder's) inequalities. Indeed, we will prove that the span of $\{\vp_n\}$ in $\Lambda_p(\Omega,\mu)$ is completely isomorphic to the space $OH$, (i.e.\ to $\ell_2$ equipped with the o.s.s.\ of $OH$). Similarly we will prove martingale inequalities involving a square function that is simply defined as $S = \sum d_n\otimes \bar d_n$ when $(d_n)$ is a martingale difference sequence.

We limit our treatment to $L_p$ for $p$ an even integer. Thus we stopped shy of making the obvious extensions:\ we can use duality to define $\Lambda_q$ for $1<q<2$ of the form $q = \frac{2n}{2n-1}$ for some integer $n>1$ and then use complex interpolation to define $\Lambda_p$ in the remaining intermediate values of $p$'s or $q$'s. While this procedure makes perfectly good sense it is rather ``unnatural'' given that if $p(0), p(1)$ and $p(\theta)$ are even integers such that $p_0<p_\theta <p_1$ and $p(\theta)^{-1} = (1-\theta) p(0)^{-1} + \theta p(1)^{-1}$, the space $\Lambda_{p(\theta)}$ does \emph{not} coincide (in general) with $(\Lambda_{p(0)}, \Lambda_{p(1)})_\theta$. This happens for instance 
when $p(0)=2,p(1)=\infty$ and $\theta=1/2$, since the spaces $L_4$ and
$\Lambda_4$ differ as operator spaces.

We will now review the contents of this paper. After some general background in \S \ref{sec1}, we explain in \S \ref{sec2} some basic facts that will be used throughout the paper. The main point is that we use an ordering on $B(H)\otimes \ovl{B(H)}$ denoted by $T\prec S$ that is such that
\[
0 \prec T \prec S \Rightarrow \|T\|\le \|S\|
\]
where the norm is the minimal (or spatial) tensor product on $B(H) \otimes_{\min} \ovl{B(H)}$, i.e.\ the norm induced by $B(H\otimes_2 \ovl H)$. As we explain in \S 2, it is convenient to abuse the notation and to extend the notation $T\prec S$ to pairs, $T,S$ in $B(H_1) \otimes\cdots\otimes B(H_{2n})$ when the collection $\{H_1,\ldots, H_{2n}\}$ can be permuted to be of the form $\{K_1,\ldots, K_n, \ovl K_1,\ldots, \ovl K_n\}$ so that $T,S$ can be identified with elements of $B({\cl H})\otimes \ovl{B({\cl H})}$ with ${\cl H} = K_1 \otimes\cdots\otimes K_n$.

In \S \ref{sec3}, we use the properties of this ordering to prove a version of H\"older's inequality that allows us to introduce, for each even integer $p$, a new operator space structure on the space $L_p(\Omega,{\cl A},\mu)$ associated to a general measure space $(\Omega,{\cl A},\mu)$. We follow the same route that we used for $p=2$ in \cite{P3} to define the space $OH$ starting from Haagerup's Cauchy--Schwarz inequality. Suitable iterations of the latter leads to versions of H\"older's inequality in $L_4,L_6,L_8,\ldots$ from which a specific norm can be introduced on $B(H)\otimes  L_p(\mu)$ that endows $L_p(\mu)$ with an operator space structure. We denote by $\Lambda_p(\Omega,\Sigma,\mu)$ the resulting operator space. It is natural
to identify  $\Lambda_1(\Omega,\Sigma,\mu)$ with $L_1(\mu)$ equipped with its maximal operator space structure in the Blecher-Paulsen sense (see e.g. \cite{ER,P6} for details on this).

Some of the calculations involving $\Lambda_p(\mu)$ are rather satisfactory, e.g.\ for any $f_j\in B(H_j)\otimes L_p(\mu)$, $j=1,\ldots, p$ the pointwise product $L_p\times\cdots\times L_p\to L_1$ applied to $(f_1,\ldots, f_n)$ leads to an element denoted by $F=f_1 \dot\otimes \cdots \dot\otimes f_n$ in $B(H_1)\otimes\cdots\otimes B(H_n)\otimes L_1$, and if ${\cl H} = H_1 \otimes\cdots\otimes H_n$ we have
\[
\|F\|_{B({\cl H})\otimes_{\min} \Lambda_1(\mu)} \le \Pi\|f_j\|_{B(H_j)\otimes_{\min} \Lambda_p}.
\]
In particular, if $q\le p$ are even integers and if $\mu(\Omega)=1$, the inclusion $\Lambda_p(\mu)\to \Lambda_q(\mu)$ is completely contractive. We also show that all conditional expectations are completely contractive on $\Lambda_p(\mu)$. When $p\to\infty$, we recover the usual  operator space structure of $L_\infty(\mu)$ as the limit of those of $\Lambda_p(\mu)$.

In \S \ref{sec4} we prove a version of Burkholder's square function inequalities for martingales in $\Lambda_p(\mu)$. If $(d_n)$ is a sequence of martingale differences in $B(H)\otimes L_p$ the relevant square function is $S = \Sigma d_n\dot\otimes \bar d_n$. Here we restrict to $p=2^k$ for some $k$, but
at least one side of the inequality is established for any even $p$  by a different argument in \S \ref{sec11}. We also prove an analogue for $\Lambda_p$ of  the inequality due to Stein expressing that 
for any sequence $(f_n)$ in $L_p$ the $L_p$-norm
of $ ( \sum |f_n|^2)^{1/2}$ dominates that of $ ( \sum |\EE_n(f_n)|^2)^{1/2}$ for any $1<p< \infty$.

In \S \ref{sec5}, we consider the conditioned square function $\sigma = \sum \EE_{n-1}(d_n\dot\otimes \bar d_n)$ and we prove a version of the Burkholder--Rosenthal inequality adapted to $\Lambda_p(\mu)$. As can be expected, the preceding inequalities imply the complete boundedness of the multipliers called ``martingale transforms''. Not surprisingly, in \S \ref{sec6} we can also prove similar results for the Hilbert transform, say on ${\bb T}$ or ${\bb R}$, using the well known Riesz-Cotlar trick. 

In \S \ref{sec7}, we compare the ``old'' and the ``new'' o.s.s.\ on $L_p(\mu)$. We show that the (isometric) inclusion $L_p(\mu)\to \Lambda_p(\mu)$ is completely contractive, but its inverse is not completely bounded and we show that its c.b.\ norm in the $n$-dimensional case grows at least like $n^{\frac12(\frac12-\frac1p)}$.

In \S \ref{sec8}, we turn to the non-commutative case. We introduce the space $\Lambda_p(\tau)$ associated to a non-commutative measure space $(\cl M,\tau)$. By this we mean a von Neumann algebra $\cl M$ equipped
with a semi-finite faithful normal trace $\tau$. 

In \S \ref{compa}, we repeat the comparison made in \S \ref{sec7}. It turns out that the non-commutative
case is significantly more intricate, mainly because the (joint) complete boundedness
of the product map $L_p \times L_q\to L_r $ ($p^{-1}+ q^{-1}=r^{-1}$) no longer holds in general
(I am grateful to Quanhua Xu for drawing my attention to this).
This leads us to consider yet another operator space structure on $L_p(\tau)$ that we denote
by ${\cl L}_p(\tau)$ for which  it still holds (see Proposition  \ref{propro}).
 When $p\in 2\N$, we then extend the main result of \S \ref{sec7}
by showing (see Corollary \ref{corcor}) that the identity defines a completely contractive map 
 ${\cl L}_p(\tau)\cap  { {\cl L}_p(\tau)}^{op}\to {\Lambda}_p(\tau)  $.
In the commutative case ${ L}_p(\tau)$ and ${\cl L}_p(\tau)$ are identical. 
We give estimates of the growth in $n$  of the c.b. norms of the maps 
$L_p(M_n,tr) \to \Lambda_p(M_n,tr)$ 
and  $\Lambda_p(M_n,tr) \to L_p(M_n,tr)$  induced by the identity.

In \S \ref{sec9bis}, assuming $\tau(1)=1$,  we study the ``limit"
when $p\to \infty$ of the spaces $\Lambda_p(\cl M,\tau)$, that we denote by
$\Lambda_\infty(\cl M,\tau)$. Surprisingly, we are able to identify the resulting
operator space: Indeed, when  $(\cl M,\tau)$ is $M_n$ equipped with its normalized trace
then $\Lambda_\infty(\cl M,\tau)$ can be identified completely isometrically
with $CB(OH_n,OH_n)$, i.e. the space of c.b. maps on the 
$n$-dimensional operator Hilbert space. More generally (see Theorem \ref{mn}), when $\cl M\subset B(H)$, the o.s.s. of
$\Lambda_\infty(\cl M,\tau)$  can be identified  with the one induced
by $CB(OH,OH)$, where by $OH$ we mean $H$ equipped with its
unique self-dual structure in the sense of \cite{P6}. The verification of these facts
leads us to several observations on the space $CB(OH,OH)$ that may be of independent interest.
In particular, the latter space satisfies a curious identity (see \eqref{g}) that appears like an
operator space analogue of the Gelfand axiom for $C^*$-algebras. Furthermore, to any operator space
$E\subset B(H)$, we associate the operator space $\underline{E}\subset CB(OH)$
(equipped with the  operator space structure induced by $CB(OH)$), and we show
that if $F$ is another operator space, for any cb-map $u:\ E\to F$ we have
$$\|u:\ \underline{E}\to \underline{F}\|_{cb}\le \|u:\  {E}\to  {F}\|_{cb}.$$

In \S \ref{sec10} we extend the Burkholder inequalities except that---for the moment---we can only prove the two sides of the martingale inequality for $p=4$. Note however that the right hand side
is established for all even $p$ in \S \ref{sec11} by a different method based on the notion of $p$-orthogonality.

Nevertheless, in \S \ref{sec9}, using Buchholz's ideas in \cite{Buch2} we can prove versions of the non-commutative Khintchine inequalities for $\Lambda_p(\tau)$ with optimal constants for any even integer $p$. We may consider spin systems, free semi-circular (or ``free-Gaussian'') families, or the free generators of the free group in the associated free group factor. Returning to the commutative case this yields the Rademacher function case with optimal constants.
The outcome is that  the span of each of these sequences in $\Lambda_p$ is completely isomorphic to $OH$ and completely complemented.

In \S \ref{sec12} we transplant the results of \cite{Har} (see also \cite{Har2}) 
on non-commutative  lacunary series to the setting of $\Lambda_p(\tau)$.
We use the view point of \cite{P5}  to abbreviate the presentation.
Let $\Gamma$ be a  discrete group. We will consider $\Lambda(p)$-sets in Rudin's sense inside $\Gamma$.
The main point is that 
a certain class of $\Lambda(p)$-subsets of $\Gamma$  again spans a copy of $OH$ in the
operator spaces   $\Lambda_p(M,\tau)$ when $M$ is the von Neumann algebra of    $\Gamma$. 
For our new o.s.s. the relevant notion
of $\Lambda(p)$-set is slightly more general than the one needed in \cite{Har}. 

Lastly in the appendix \S \ref{sec13} we include a discussion of elements that have $p$-th moments
defined by pairings as in Buchholz's paper \cite{Buch2}. We simply translate in the abstract language of tensor products
some very well known classical  ideas on Wick products for Gaussian random variables.
Our goal is to emphasize the similarity between the Gaussian case
and the free or $q$-Gaussian analogues. We feel this appendix fits well with the extensive
use of tensor products throughout the sections preceding it.

The 
proof of the  initial non-commutative martingale inequalities of \cite{PX1} 
is the main source of inspiration for the present results. We also make crucial use
of our version for the space $\Lambda_p$ of Junge's ``dual Doob" inequality from \cite{Ju}. 
Although we take a divergent route, we should point out to the reader that
the methods of \cite{PX1} have been considerably improved 
in a series of important  later works, such as \cite{JX0,JX1,JX2}, 
by M. Junge and Q. Xu, or \cite{Ra1,Ra2,PRa} by N. Randrianantoanina
and J. Parcet. See also \cite{JPX, Ra3,Ra4,RX,X3} for progress related
to Khintchine's inequalities. The reader is referred to these papers
to get an idea of what the ``main stream" on 
non-commutative martingale and Khintchine inequalities is about.

\section{Background on operator spaces}\label{sec1}
In this section we summarize the Theory of Operator Spaces. We refer either to  \cite{ER} or \cite{P6}
for full details.

We recall that an operator space is just a closed subspace of the algebra $B(H)$ of all bounded operators on a Hilbert space $H$.

Given an operator space $E\subset H$, we denote by $M_n(E)$ the space of $n\times n$ matrices with entries in $E$ and we equip it with the norm induced by that of $M_n(B(H))$, i.e.\ by the operator norm on $H\oplus\cdots\oplus H$ ($n$ times). We denote by $E\otimes F$ the \emph{algebraic} tensor product of two vector spaces $E,F$.

If $E\subset B(H)$ and $F\subset B(K)$ are operator spaces, we denote by $E\otimes_{\min} F$ the closure of $E\otimes F$ viewed as a subspace of $B(H\otimes_2 K)$. We denote by $\|~~\|_{\min}$ the norm induced by $B(H\otimes_2 K)$ on $E\otimes F$ or on its closure $E\otimes_{\min} F$.
A linear map $u\colon \ E\to F$ between operator spaces is called completely bounded (c.b.\ in short) if the associated maps $u_n\colon \ M_n(E)\to M_n(F)$ defined by $u_n([x_{ij}]) = [u(x_{ij})]$ are bounded uniformly over $n$, and we define 
\[
 \|u\|_{cb} = \sup_{n\ge 1} \|u_n\|.
\]
We say that $u$ is completely isometric if $u_n$ is isometric for any $n\ge 1$ and that $u$ is a complete isomorphism if it is an isomorphism with c.b.\ inverse.

By a well known theorem due to Ruan (see \cite{ER,P6}), an operator space $E$ can be characterized up to complete isometry by the sequence of normed spaces $\{M_n(E)\mid n\ge 1\}$. The data of the sequence of norms on the spaces $M_n(E)$ $(n\ge 1)$ constitutes the operator space structure (o.s.s.\ in short) on the vector space underlying $E$. Note that $M_n(E) = B(H_n) \otimes_{\min} E$ where $H_n$ denotes here the $n$-dimensional Hilbert space.

Actually, the knowledge of the o.s.s.\ on $E$ determines that of the norm on $B(H) \otimes E$ for any Hilbert space $H$. Therefore we (may and) will take the viewpoint that the o.s.s.\ on $E$ consists of the family of normed spaces (before completion)
\[
 (B(H)\otimes E, \|\cdot\|_{\min}),
\]
where $H$ is an arbitrary Hilbert space. The reader should keep in mind that we may restrict either to $H=\ell_2$ or to $H=\ell^n_2$ with $n\ge 1$ allowed to vary arbitrarily. (If we fix $n=1$ everywhere, the theory reduces to the ordinary Banach space theory.) To illustrate our viewpoint we note that
\[
 \|u\|_{cb} = \sup_H \|u_H\colon \ B(H)\otimes_{\min} E\to B(H) \otimes_{\min} F\|
\]
where the mapping $u_H$ is the extension (by density) of $id\otimes u$.

The most important examples for this paper are the spaces $L_p(\mu)$ associated to a measure space $(\Omega,{\cl A},\mu)$. Our starting point will be the 3 cases $p=\infty$, $p=1$ and $p=2$. For $p=\infty$, the relevant norm  on $B(H)\otimes L_\infty(\mu)$ is the unique $C^*$-norm, easily described as follows:\ any $f$ in $B(H)\otimes L_\infty(\mu)$ determines a function $f\colon \ \Omega\to B(H)$ taking values in a finite dimensional subspace of $B(H)$ and we have
\begin{equation}\label{eq0.1}
 \|f\|_{B(H) \otimes_{\min} L_\infty(\mu)} = \underset{\omega\in\Omega}{\text{ess sup}} \|f(\omega)\|_{B(H)}.
\end{equation}
For $p=1$, the relevant norm on $B(H)\otimes L_1(\mu)$ is defined using operator space duality, but it can be explicitly written as follows
\begin{equation}\label{eq0.2}
 \|f\|_{B(H)\otimes_{\min} L_1(\mu)} = \sup \left\|\int f(t) \otimes g(t)\ d\mu(t)\right\|_{B(H\otimes_2 K)}
\end{equation}
where the sup runs over all $g$ in the unit ball of $(B(K) \otimes L_\infty(\mu), \|\cdot\|_{\min})$ and over all possible $K$. Equivalently, we may restrict to $H=K=\ell_2$. This is consistent with the standard dual structure on the dual $E^*$ of an operator space. There is an embedding $E^*\subset B({\cl H})$ such that the natural identification
\[
 B(H) \otimes E^* \longleftrightarrow B(E,B(H))
\]
defines for any $H$ an isometric embedding
\[
 B(H) \otimes_{\min} E^* \subset CB(E,B(H)).
\]
With this notion of duality we have $L_1(\mu)^* = L_\infty(\mu)$ completely isometrically. Moreover the inclusion $L_1(\mu) \subset L_\infty(\mu)^*$ is completely isometric, and this is precisely reflected by the formula \eqref{eq0.2}.

To define the o.s.s.\ on $L_2(\mu)$, we will use the complex conjugate $\ovl H$ of a Hilbert space $H$. Note that the map $x\to x^*$ defines an anti-isomorphism on $B(H)$. Since $\ovl{B(H)} = B(\ovl H)$ (canonically), we may view $x\to x^*$ as a linear $*$-isomorphism from $\ovl{B(H)} = B(\ovl H)$ to $B(H)^{op}$ where $B(H)^{op}$ is the same $C^*$-algebra as $B(H)$ but with reversed product. Then for any $x_k\in B(H)$, $y_k\in B(K)$ we have
\[
 \left\|\sum x_k\otimes y_k\right\|_{\min} = \sup\left\{\left\|\left(\sum x_k\otimes y_k\right) (\xi)\right\|_{H\otimes_2K} \mid\ \xi \in H\otimes K, \ \|\xi\|_{H\otimes_2K}\le 1\right\}.
\]
Let us denote by $S_2(H,K)$ the class of Hilbert--Schmidt operators from $H$ to $K$ with norm denoted by $\|\cdot\|_{S_2(H,K)}$ or more simply by $\|~~\|_2$. We may identify canonically $\xi\in H\otimes_2 K$ with an element $\hat\xi\colon \ K^*\to H$ with Hilbert--Schmidt norm $\|\hat\xi\|_2 =\|\xi\|_{H\otimes_2 K}$. Then for any $y\in B(K)$ let ${}^ty\colon \ K^*\to K^*$ denote the adjoint operator. We have then
\[
 \left\|\sum x_k\otimes y_k\right\|_{\min} = \sup\left\{\left\|\sum x_k\hat\xi \  {}^ty_k\right\|_2 \ \Big| \ \|\hat\xi\|_2\le 1\right\}.
\]
Using the identification $\ovl{B(K)} = B(K)^{op}$ via $\bar x\to x^*$, (let $x\to\bar x$ be the identity map on $B(K)$ viewed as a map from $B(K)$ to $\ovl{B(K)}$) we find
\[
 \left\|\sum x_k\otimes \bar y_k\right\|_{B(H\otimes_2\ovl K)} = \sup\left\{\left\| \sum x_k ay^*_k\right\|_2\ \Big| \ a\in S_2(K,H), \ \|a\|_2\le 1\right\}.
\]
We now define the ``natural'' o.s.s.\ on the space $\ell_2$ according to \cite{P3}. This is defined by the following formula:\ for any $f$ in $B(H)\otimes \ell_2$,  of the form $f = \sum^n_1 x_k\otimes e_k$ (here $(e_k)$ denotes the canonical basis of $\ell_2$) we have
\begin{equation}\label{eq0.3}
 \|f\|_{B(H)\otimes_{\min} \ell_2} = \left\|\sum x_k\otimes\bar x_k\right\|^{1/2}_{B(H \otimes_2\ovl H)}.
\end{equation}
The resulting o.s.\ is called ``the operator Hilbert space'' and is denoted by $OH$. Actually, the same formula works just as well for any Hilbert space ${\cl H}$ with an orthonormal basis $(e_i)_{i\in I}$. The resulting o.s.\ will be denoted by ${\cl H}_{oh}$ (so that $OH$ is just another notation for $(\ell_2)_{oh}$).

In this paper our main interest will be the space $L_2(\mu)$. The relevant o.s.s.\ can then be described as follows:\ for any $f$ in $B(H)\otimes L_2(\mu)$ we have  
\begin{equation}\label{eq0.4}
 \|f\|_{B(H)\otimes_{\min} L_2(\mu)_{oh}} = \left\|\int f(\omega)\otimes \ovl{f(\omega)} d\mu(\omega)\right\|^{1/2}_{B(H\otimes_2\ovl H)}.
\end{equation}
It is not hard to see that this coincides with the definition \eqref{eq0.3} when $(e_n)$ is an orthonormal basis of $L_2(\mu)$.

We refer the reader to \cite{P3} for more information on the space ${\cl H}_{oh}$, in particular for the proof that this space is uniquely characterized by its self-duality in analogy with Hilbert spaces among Banach spaces. We note that ${\cl H}_{oh}$ and $H_{oh}$ are completely isometric iff the Hilbert spaces ${\cl H}$ and $H$ are isometric (i.e.\ of the same Hilbertian dimension).

The ``natural o.s.s." on $L_p=L_p(\mu)$ is defined in \cite{P4}
for $1<p<\infty$ using complex interpolation. It is characterized by the following isometric identity: For any finite dimensional Hilbert space $H$
$$B(H)\otimes_{\min} L_p=(B(H)\otimes_{\min} L_\infty,B(H)\otimes_{\min} L_1)_{1/p}.$$
When $p=2$ we recover the o.s.s. defined above
for $L_2(\mu)_{oh}$.

We now turn to multilinear maps.
Let $E_1,\ldots, E_m$ and $F$ be operator spaces. Consider an $m$-linear map $\varphi\colon \ E_1\times\cdots\times E_m\to F$. Let $H_1,\ldots, H_m$ be Hilbert spaces. We set $B_j = B(H_j)$. By multilinear algebra we can associate to $\varphi$ an $m$-linear map
\[
 \widehat\varphi\colon \ B_1\otimes E_1 \times\cdots\times B_m\otimes E_m\to B_1 \otimes\cdots\otimes B_m \otimes F
\]
characterized by the property that ($\forall b_j\in B_j, \forall  e_j\in E_j$) 
\[ \widehat\varphi(b_1\otimes e_1,  \cdots  ,b_m\otimes e_m)=b_1\otimes\cdots\otimes b_m \otimes\varphi(e_1,\cdots  ,e_m).
\]
We say that $\varphi$ is (jointly) completely bounded
(c.b. in short) if 
$\widehat\varphi$ is bounded from
\[
B_1 \otimes_{\min} E_1 \times\cdots\times B_m \otimes_{\min} E_m\quad \text{to}\quad B_1 \otimes_{\min} \cdots\otimes_{\min} B_m \otimes_{\min} F.
\]
It is easy to see that we may reduce to the case when $H_1 = H_2 =\cdots= H_m = \ell_2$ (equivalently we could restrict to finite dimensional Hilbert spaces of arbitrary dimension). With this choice of $H_j$ we set $$\|\varphi\|_{cb} = \|\widehat\varphi\|.$$

\section{Preliminary results}\label{sec2}

\indent 

We first recall Haagerup's version of the Cauchy--Schwarz inequality on which is based a lot of what follows.

Let $H,K$ be Hilbert spaces, $a_k\in B(H), b_k\in B(K)$ $(k=1,\ldots, n)$. We have then
\[
 \left\|\sum a_k\otimes \bar b_k\right\|_{B(H\otimes_2\ovl K)} \le \left\|\sum a_k\otimes \bar a_k\right\|^{1/2}_{B(H\otimes_2\ovl H)} \left\|\sum b_k\otimes \bar b_k\right\|^{1/2}_{B(K\otimes_2\ovl K)}
\]
where $\otimes_2$ denotes the Hilbert space tensor product.

We will sometimes need the following reformulation: let ${\cl H}$ be another Hilbert space and for any $f\in {\cl H}\otimes B(H)$, say $f = \sum x_k\otimes a_k$, and $g\in {\cl H}\otimes B(K)$, say $g = \sum y_\ell \otimes b_\ell$ let $\langle\langle f,g\rangle\rangle \in B(H)\otimes \ovl{B(K)}$ be defined by
\[
 \langle\langle f,g\rangle\rangle = \sum\nolimits_{k,\ell} \langle x_k,y_\ell\rangle a_k\otimes \bar b_\ell.
\]
We have then
\begin{equation}\label{eq2.01}
 \|\langle\langle f,g\rangle\rangle\|_{B(H\otimes_2\ovl K)} \le  \|\langle\langle f,f\rangle\rangle\|^{1/2}\|\langle\langle g,g\rangle\rangle\|^{1/2}.
\end{equation}
More generally for any finite sequences $(f_\alpha)_{1\le \alpha\le N}$ in ${\cl H}\otimes B(H)$ and $(g_\alpha)_{1\le\alpha\le N}$ in ${\cl H} \otimes B(K)$ we have
\begin{equation}\label{eq2.02}
 \left\|\sum\nolimits_\alpha \langle\langle f_\alpha,g_\alpha\rangle\rangle\right\|_{B(H\otimes_2 \ovl K)} \le \left\|\sum\nolimits_\alpha \langle\langle f_\alpha,f_\alpha\rangle\rangle\right\|^{1/2}_{B(H \otimes_2\ovl H)} \left\|\sum\nolimits_\alpha \langle\langle g_\alpha,g_\alpha\rangle\rangle\right\|^{1/2} _{B(K\otimes_2\ovl K)}.
\end{equation}
It is convenient to also observe here that if $E_1,E_2$ are orthogonal subspaces of ${\cl H}$ and if $f_j\in E_j\otimes B(H)$ $(j=1,2)$ then
\begin{equation}\label{eq2.03}
 \langle\langle f_1+f_2, f_1+f_2\rangle\rangle = \langle\langle f_1,f_1\rangle\rangle + \langle\langle f_2,f_2\rangle\rangle.
\end{equation}

We will use an order on $B(H)\otimes \ovl{B(H)}$: \ Let $C_+$ be the set of all finite sums of the form $\sum a_k\otimes \bar a_k$. If $x,y$ are in $B(H)\otimes \ovl{B(H)}$, we write $x \prec y$ (or $y\succ x$) if $y-x \in C_+$. In particular $x\succ 0$ means $x\in C_+$. Any element $x\in B(H)\otimes \ovl{B(H)}$ defines a (finite rank) sequilinear form $\tilde x\colon \ B(H)^* \times B(H)^* \to {\bb C}$. 
More generally, for any complex vector space $E$, we may define similarly
 $x\succ 0$  for any $x\in E\otimes \bar E$.\\
The following criterion is easy to show by linear algebra. 
\begin{lem}\label{lem1.0} Let $x\in B(H)\otimes \ovl{B(H)}$. Then
 $x\in C_+$ iff $\tilde  x$ is positive definite i.e.\ $\tilde x(\xi,\xi)\ge 0$ for any $\xi$ in $B(H)^*$. Moreover, this holds iff
 $\tilde x(\xi,\xi)\ge 0$ for any $\xi$ in
the predual  $B(H)_*\subset B(H)^*$ of $B(H)$. Lastly, if $H$ is separable, there is a countable subset
$D\subset B(H)_*$ such that $x\succ 0$  iff
 $\tilde x(\xi,\xi)\ge 0$ for any $\xi$ in $D$.

\end{lem}
\begin{proof} The first part is a general fact valid for any complex Banach space $E$ in place of
$B(H)$: Assume $x\in E\otimes \bar E$, then $x\in F\otimes \bar F$ for some finite dimensional $F\subset E$, thus the equivalence in Lemma \ref{lem1.0} just reduces to the classical spectral decomposition
of a positive definite matrix. If $E$ is a dual space with a predual $E_*\subset E^*$,
then $E_*$ is $\sigma(E^*,E)$-dense in $E^*$, so the condition
$\tilde x(\xi,\xi)\ge 0$ will hold  for any $\xi\in E^*$ if it does  for any $\xi\in E_*$.  When the predual $E_*$ is separable, the last assertion becomes immediate.
\end{proof}
\begin{rem}  Let $E$ be a complex Banach space. Let $[a_{ij}]$ be a complex $n\times n$ matrix, $x_j\in E$.
Consider $x=\sum a_{ij} x_i\otimes \bar x_j\in E\otimes \bar E$. Then $x\succ 0$ if $[a_{ij}]$ is positive definite, and 
 if the $x_j$'s are linearly independent, the converse also holds. Indeed, by the preceding argument,
 all we need to check is $\tilde x(\xi,\xi)=\sum a_{ij} \xi(x_i)  \overline{\xi( x_j)}\ge 0$.
\end{rem}
The importance of this ordering for us lies in the following fact:

\begin{lem}\label{lem1.1}
If $x,y\in B(H) \otimes\ovl{B(H)}$ and $0\prec x\prec y$ then
\[
 \|x\|_{\min} \le \|y\|_{\min}
\]
where $\|x\|_{\min} = \|x\|_{B(H\otimes_2\ovl H)}$. \end{lem}

\begin{proof}
Let $x = \sum a_k\otimes \bar a_k$. Then the lemma is immediate from the identity
\[
 \left\|\sum a_k\otimes \bar a_k\right\| = \sup\left\{\left(\sum \|\xi a_k\eta\|^2_2\right)^{1/2} \ \Big| \ \|\xi\|_4 \le 1, \|\eta\|_4\le 1\right\}
\]
for which we refer to \cite{P3}. Indeed, if $d=\sum b_j\otimes \bar b_j$ and $y=x+d$
this identity applied to $x+d$ makes it clear that
 $\|x\|\le \|x+d\|=\|y\|$.
\end{proof}

It will be convenient to extend our notation:\ Let $H_1,H_2,\ldots, H_m$ be an $m$-tuple of Hilbert spaces. For any $k=1,\ldots, m$ we set
\[
 H_{m+k} = \ovl H_k.
\]
Let $\sigma$ be any permutation of $\{1,\ldots, 2m\}$. For any element $x$ in $B(H_1) \otimes\cdots\otimes B(H_{2m})$ we denote by
\[
 \sigma \cdot x\in B(H_{\sigma(1)}) \otimes\cdots\otimes B(H_{\sigma(2m)})
\]
the element obtained from $x$ by applying $\sigma$ to the factors, i.e.\ if $x = t_1\otimes\cdots\otimes t_{2m}$ then $\sigma\cdot x = t_{\sigma(1)} \otimes\cdots\otimes t_{\sigma(2m)}$ and $x\to\sigma\cdot x$ is the linear extension of this map.

Let ${\cl H} = H_1\otimes\cdots\otimes H_m$. Now if we are given a permutation $\sigma$ and $x,y$ in $B(H_{\sigma(1)})\otimes\cdots\otimes B(H_{\sigma(2m)})$ we note that
\[
 \sigma^{-1}\cdot x, \sigma^{-1}\cdot y \in B({\cl H}) \otimes \ovl{B({\cl H})}.
\]
We will write (abusively) $x\prec y$ (or $y\succ x$) if we have
\[
 \sigma^{-1}\cdot x \prec \sigma^{-1}\cdot y.
\]
Of course this order depends on $\sigma$ and although our notation does not keep track of that, we will need to remember $\sigma$, but hopefully no confusion should arise. While we will use various choices for $\sigma$, we never change our choice in the middle of a calculation, e.g. when adding two ``positive" terms.

For instance we allow ourselves to write that $\forall a_k \in B(H)$ $\forall b_k\in B(K)$ we have
\begin{equation}\label{diag}
 \sum a_k\otimes \bar a_k \otimes b_k\otimes \bar b_k \succ 0
\end{equation}
in $B(H\otimes \ovl H \otimes K\otimes \ovl K)$, where implicitly we are referring to the permutation $\sigma$ that takes $H\otimes \ovl H \otimes K \otimes \ovl K$ to $H\otimes K \otimes \ovl H \otimes \ovl K$. In particular, we note that with this convention $\forall x,y\in B(H)\otimes \ovl{B(H)}$ $\forall b\in B(K)$ we have
\begin{equation}\label{eq-or}
 x\prec y \Rightarrow b \otimes x \otimes \bar b \prec b \otimes y \otimes \bar b.
\end{equation}
Note that since the minimal tensor product is commutative, we still have,  for any $x,y$ in $B(H_{\sigma(1)}) \otimes\cdots\otimes B(H_{\sigma(2m)})$, that
\begin{equation}\label{eq1.1-}
 0 \prec x \prec y \Rightarrow \|x\|_{\min} \le \|y\|_{\min}.
\end{equation}
From the obvious identity $(x,y\in B(H))$
\[
 (x+y) \otimes \ovl{(x+y)} + (x-y) \otimes \ovl{(x-y)} = 2(x\otimes \bar x + y\otimes \bar y)
\]
it follows that
\begin{equation}\label{eq1.1}
(x+y)\otimes \ovl{(x+y)} \prec 2(x\otimes \bar x + y\otimes \bar y).
\end{equation}
Note that if we set $\Phi(x)=x\otimes \bar x$,
then the preceding
expresses the ``order convexity" of this function:
$$\Phi((x+y)/2) \prec (\Phi(x) +\Phi(y))/2.$$
More generally, for any finite set $x_1,\ldots, x_n$ in $B(H)$ we  have
$\Phi(n^{-1}\sum\nolimits^n_1 x_k) \prec n^{-1}\sum\nolimits^n_1\Phi(x_k),$ or
\begin{equation}\label{eq1.3}
\left(\sum\nolimits^n_1 x_k\right) \otimes \ovl{\left(\sum\nolimits^n_1 x_k\right)} \prec n \sum\nolimits^n_1 x_k \otimes \bar x_k.
\end{equation}
We need to record below several variants of  the ``order convexity" of $\Phi$.\\ From now on we assume that $H$ is a {\it separable} Hilbert space.\\
More generally, for any $x$ in $B(H)\otimes L_2(\Omega, {\cl A}, {\bb P})$ and any $\sigma$-subalgebra ${\cl B}\subset {\cl A}$,
we may associate to   $x\otimes \bar x$ the function $x\dot\otimes \bar x\colon \ \omega\to B(H)\otimes\overline{ B(H)}$ defined by
$x\dot\otimes \bar x(\omega)=x(\omega)\otimes \bar x(\omega)$.
We have then almost surely
\begin{equation}\label{eq1.1+}
  0 \prec {\bb E}^{\cl B}(x\dot\otimes \bar x), 
\end{equation} and more precisely (again almost surely)
\begin{equation}\label{eq1.2}
 ({\bb E}^{\cl B}x) \dot\otimes \ovl{({\bb E}^{\cl B}x)} \prec {\bb E}^{\cl B}(x \dot\otimes \bar x).
\end{equation}
Indeed, \eqref{eq1.1+} follows from  Lemma \ref{lem1.0} (separable case) and the right hand side of \eqref{eq1.2} is equal to
$$
 ({\bb E}^{\cl B}x) \dot\otimes \ovl{({\bb E}^{\cl B}x)} +{\bb E}^{\cl B}( y \dot\otimes \bar y) \quad \text{where}\quad y = x-{\bb E}^{\cl B}x.
$$
When ${\cl B}$ is the trivial algebra, we obtain
\begin{equation}\label{eq1.1++}
  0 \prec \int x\dot\otimes \bar x, 
\end{equation} and hence for any measurable subset $A\subset \Omega$
\begin{equation}\label{eq1.1+++}
   \int_A  x\dot\otimes \bar x \  d{\bb P} \prec \int_\Omega  x\dot\otimes \bar x\  d{\bb P}  ,
\end{equation} and also
\begin{equation}\label{eq1.1iv}
   (\EE   x) \otimes (\EE  \bar x)  \prec \EE(  x\dot\otimes \bar x) .
\end{equation}

We need to observe that for any integer $m\ge 1$ we have
\begin{equation}\label{eq1.6}
 0 \prec x \Rightarrow 0 \prec x ^{\otimes m}
\end{equation}
and more generally
\begin{equation}\label{eq1.7}
 0 \prec x \prec y \Rightarrow 0 \prec x^{\otimes m} \prec y^{\otimes m}.
\end{equation}
Furthermore, if $x_1,y_1\in B(H_1)\otimes \ovl{B(H_1)}$ and $x_2,y_2\in B(H_2)\otimes \ovl{B(H_2)}$
\begin{equation}\label{eq1.7bbis}
 0 \prec x_1 \prec y_1 \ {\rm and}\   0 \prec x_2 \prec y_2 \Rightarrow 0 \prec x_1{\otimes } x_2 \prec y_1{\otimes } y_2,
\end{equation}
where the natural permutation is applied to $x_1{\otimes } x_2$ and  $ y_1{\otimes } y_2$, allowing to view them as
elements of $B(H_1\otimes H_2)  \otimes \ovl{B(H_1\otimes H_2)}$.

Returning to \eqref{eq1.2}, we note that it implies
$$  ( ({\bb E}^{\cl B}x) \dot\otimes \ovl{({\bb E}^{\cl B}x)} )^{\otimes 2}\prec ( {\bb E}^{\cl B}(x\dot\otimes \bar x))^{\otimes 2}\prec {\bb E}^{\cl B}(x\dot\otimes \bar x\dot\otimes  x\dot\otimes \bar x),$$
and hence
$$  \|( ({\bb E}^{\cl B} x) \dot\otimes \ovl{({\bb E}^{\cl B} x)} )^{\otimes 2}\|\le \| {\bb E}^{\cl B}(x\dot\otimes \bar x\dot\otimes  x\dot\otimes \bar x)\|.$$
More generally, iterating this argument, we obtain
for any integer $k\ge 1$

\begin{equation}\label{eq1.8}
  ( ({\bb E}^{\cl B}x) \dot\otimes \ovl{({\bb E}^{\cl B}x)} )^{\otimes 2^k}\prec  {\bb E}^{\cl B}( (x\dot\otimes \bar x)^{\otimes 2^k} ).
  \end{equation}
  In particular
  \begin{equation}\label{eq1.8+}
( ({\bb E} x) \dot\otimes \ovl{({\bb E} x)} )^{\otimes 2^k}  \prec  {\bb E}( (x\dot\otimes \bar x)^{\otimes 2^k} ),\end{equation}
and consequently for any finite sequence $x_1,\cdots,x_n\in B(H)\otimes L_2(\Omega, {\cl A}, {\bb P})$
\begin{equation}\label{eq1.8++}
\|\sum\nolimits_j ( ({\bb E}^{\cl B} x_j) \dot\otimes \ovl{({\bb E}^{\cl B} x_j)} )^{\otimes 2^k}\|\le \| \sum\nolimits_j {\bb E}^{\cl B}( (x_j\dot\otimes \bar x_j)^{\otimes 2^k} )\|.\end{equation}

In a somewhat different direction, for any measure $\mu$, let $f\in B(H)\otimes L_2(\mu)$, let $P$ be any orthogonal projection
on $L_2(\mu)$ and let $g=(I\otimes P)(f)$. Then
\begin{equation}\label{eq1.9}
0\prec \int g\dot\otimes \bar g d\mu\prec \int f\dot\otimes \bar f d\mu.
\end{equation}
Indeed, this is immediate by \eqref{eq2.03}.

\section{Definition of $\pmb{\Lambda_{2m}}$}\label{sec3}

\indent 

Our definition of the operator space $\Lambda_{2m}(\mu)$ is based on the case $m=1$, i.e.\ on the operator Hilbert space $OH$, studied at length in \cite{P3}. The latter is based on the already mentioned  Cauchy--Schwarz inequality due to Haagerup as follows:\ Let $H,K$ be Hilbert spaces and let $a_k\in B(H)$, $b_k\in B(K)$
\begin{equation}\label{eq2.1}
 \left\|\sum a_k\otimes b_k\right\| \le \left\|\sum a_k\otimes \bar a_k\right\|^{1/2}  \left\|\sum b_k\otimes \bar b_k\right\|^{1/2}.
\end{equation}
This is usually stated with $\sum a_k\otimes \bar b_k$ on the left hand side, but since the right hand side is unchanged if we replace $b_k$ by $\bar b_k$ we may write this as well. It will be convenient for our exposition to use the functional version of \eqref{eq2.1} as follows:\ For any Hilbert spaces $H,K$ and any $f\in B(H)\otimes L_2(\mu)$ and $g\in B(K) \otimes L_2(\mu)$, we denote
by $f\dot \otimes g$ the $B(H)\otimes B(K)$-valued function defined by  $$(f\dot \otimes g)(\omega)=f(\omega) \otimes g(\omega).$$ 
Of course, using the identity $   B(H)\otimes \ovl{B(H)}\simeq B(H)\otimes  {B(\ovl H)} $,
this extends the previously introduced notation  for $f\dot \otimes \bar f\colon\ \Omega\to B(H)\otimes \ovl{B(H)} $.\\   
Similarly, given $n$ measurable functions $f_j\colon\ \Omega\to B(H_j)$, we denote
by $f_1\dot\otimes\cdots \dot\otimes f_n$ the pointwise product viewed as a function with values
in $B(H_1)  \otimes\cdots  \otimes B(H_n)$.

Note that if $f_j$ corresponds to an element in $B(H_j) \otimes L_{p_j} $
with  $p_j>0$ such that $\sum p_j^{-1}=p^{-1}$, then by H\"older's inequality, 
$f_1\dot\otimes\cdots \dot\otimes f_n \in B(H_1)  \otimes\cdots  \otimes B(H_n) \otimes L_{p} $.

By \eqref{eq2.01} applied with ${\cl H}=L_2(\mu)$,   we have
\begin{equation}\label{eq2.2pre}
 \left\|\int f\dot\otimes \bar g\ d\mu\right\|_{B(H\otimes \ovl K)} \le \left\|\int f\dot\otimes \bar f\ d\mu\right\|^{1/2}_{B(H\otimes\ovl H)} \left\|\int g\dot\otimes \bar g\ d\mu\right\|^{1/2}_{B(K\otimes\ovl K)}.
\end{equation}
Replacing $\bar g$ by $g$, we obtain
\begin{equation}\label{eq2.2}
 \left\|\int f\dot\otimes g\ d\mu\right\|_{B(H\otimes K)} \le \left\|\int f\dot\otimes \bar f\ d\mu\right\|^{1/2}_{B(H\otimes\ovl H)} \left\|\int g\dot\otimes \bar g\ d\mu\right\|^{1/2}_{B(K\otimes\ovl K)}.
\end{equation}
We note that this functional variant of \eqref{eq2.1} appears in unpublished work by
Furman and Shalom (personal communication). 

We will also invoke the following variant: for any $\psi \in B(\ell_2)\otimes L_\infty$ with
norm $\|\psi\|_{\min}\le 1$ we have
\begin{equation}\label{eq2.2bis}
 \left\|\int f\dot\otimes g \dot\otimes\psi \ d\mu\right\|_{B(H\otimes K\otimes \ell_2)} \le \left\|\int f\dot\otimes \bar f\ d\mu\right\|^{1/2}_{B(H\otimes\ovl H)} \left\|\int g\dot\otimes \bar g\ d\mu\right\|^{1/2}_{B(K\otimes\ovl K)}.
\end{equation}
This (which can be 
interpreted as saying that the product map $L_2\times L_2\to L_1$ is jointly completely contractive) can be
verified  rather easily using complex interpolation, see e.g.  the proof of Lemma \ref{lem4.1} below
for a  more detailed argument.

For simplicity in this section we abbreviate $L_p(\mu)$ or $L_p(\Omega,\mu)$ and we simply write $L_p$ instead.

We start by a version of H\"older's inequality adapted to our needs that follows easily from \eqref{eq2.2}. The proof uses an iteration idea already appearing in \cite{Buch}.

\begin{lem}\label{lem2.1}
Let $m\ge 1$ be any integer. Then for any $f_1,\ldots, f_{2m}$ in $B(H)\otimes L_{2m}$ we have $ f_1\dot\otimes\cdots\dot\otimes f_{2m}\in B(H)^{\otimes 2m}\otimes L_{1}$ and 
\begin{equation}\label{eq2.3}
  \left\|\int f_1\dot\otimes\cdots\dot\otimes f_{2m}\ d\mu\right\| \le \prod^{2m}_{k=1} \left\|\int f^{\dot\otimes m}_k \dot\otimes \bar f^{\dot\otimes m}_k\ d\mu\right\|^{\frac1{2m}},
\end{equation}
where we   denote $f^{\dot\otimes m} = f\dot\otimes\cdots\dot\otimes f$ ($m$ times).
\end{lem}

\begin{proof}
By homogeneity we may (and do) normalize and assume that
\begin{equation}
 \left\|\int f^{\dot\otimes m}_k \dot\otimes \bar f^{\dot\otimes m}_k\ d\mu\right\|\le 1. \tag*{$\forall k=1,\ldots, 2m$}
\end{equation}
Let $$C = \max\big\{\big\|\int g_1\dot\otimes\cdots \dot\otimes g_{2m}\ d\mu\big\|\big\}$$ where 
the maximum runs over all $g_k$ in the set $\{f_1,\ldots, f_{2m}, \bar f_1,\ldots, \bar f_{2m}\}$. \\
It clearly suffices to prove $C\le 1$. 
 In the interest of the reader,
we first do the proof in the simplest case $m=2$. 
We have by \eqref{eq2.2}
$$\|\int f_1\dot\otimes\cdots\dot\otimes f_{4}d\mu  \|\le \|\int f_1\dot\otimes f_2 \dot\otimes \ovl{  f_1\dot\otimes f_2}d\mu\|^{1/2}  \|\int f_3\dot\otimes f_4 \dot\otimes \ovl{  f_3\dot\otimes f_4}d\mu\|^{1/2}  $$
which we may rewrite as
\begin{equation}\label{p4}\|\int f_1\dot\otimes\cdots\dot\otimes f_{4}d\mu  \|\le \|\int f_1\dot\otimes \bar f_1 \dot\otimes  {  f_2\dot\otimes \bar f_2}d\mu\|^{1/2}  \|\int f_3\dot\otimes \bar f_3 \dot\otimes  {  f_4\dot\otimes \bar f_4}d\mu\|^{1/2}  ,\end{equation}
and by  \eqref{eq2.2} again we have
$$ \|\int f_1\dot\otimes \bar f_1 \dot\otimes  {  f_2\dot\otimes \bar f_2}d\mu\|^{1/2} \le  \|\int f_1\dot\otimes \bar f_1 \dot\otimes  {  \bar f_1\dot\otimes  f_1}d\mu\|^{1/2} \|\int f_2\dot\otimes \bar f_2 \dot\otimes  {  \bar f_2\dot\otimes  f_2}d\mu\|^{1/2}\le 1$$
and similarly for the other factor in \eqref{p4}.
Thus we obtain the announced inequality for $m=2$.

To check  the general case, let us denote 
\[
 I(f_1,\ldots, f_{2m}) = \left\|\int f_1\dot\otimes\cdots\dot\otimes f_{2m}\ d\mu\right\|.
\]
By  \eqref{eq2.2} we find
\[
 I(f_1,\ldots, f_{2m}) \le (I(f_1,\ldots, f_m, \bar f_1,\ldots, \bar f_m)C)^{1/2}.
\]
Note that $I(f_1,\ldots, f_{2m})$ is invariant under permutation of entries. Thus we have
\[
 I(f_1,\ldots, f_m, \bar f_1,\ldots, \bar f_m) = I(f_1,\bar f_1, f_2,\bar f_2,\ldots, f_m,\bar f_m).
\]
Using \eqref{eq2.2} again we find
\[
 I(f_1,\ldots, f_m, \bar f_1,\ldots, \bar f_m) \le (I(f_1,\bar f_1,f_1,\bar f_1, f_2,\bar f_2,\ldots)C)^{1/2}
\]
and continuing in this way we obtain
\[
 I(f_1,\ldots, f_{2m}) \le I(f_1,\bar f_1, f_1,\bar f_1,\ldots, f_1,\bar f_1)^\theta C^{1-\theta}
\]
where $0<\theta<1$ is equal to $2^{-K}$ with $K$ the number of iterations.\\
Since we assume $I(f_1,\bar f_1,\ldots, f_1,\bar f_1) \le 1$ we find
\[
 I(f_1,\ldots, f_{2m}) \le C^{1-\theta}.
\]
But we may replace $f_1,\ldots, f_{2m}$ by $g_1,\ldots, g_{2m}$ and the same argument gives us
\[
 I(g_1,\ldots, g_{2m}) \le C^{1-\theta}
\]
and hence $C\le C^{1-\theta}$, from which $C\le 1$ follows immediately.
\end{proof}

\begin{pro}\label{pro2.2}
Let $m\ge 1$ be an integer and $p=2m$. There is an isometric embedding
\[
 L_p(\mu) \subset B({\cl H})
\]
so that for any $f\in B(H)\otimes L_p(\mu)$ and any $H$ we have
\[
 \|f\|_{B(H\otimes {\cl H})} = \left\|\int f^{\dot\otimes m} \dot\otimes \bar f^{\dot\otimes m} \ d\mu\right\|^{\frac1{2m}}
\]
where we again denote $f^{\dot\otimes m} = f\dot\otimes\cdots\dot\otimes f$ ($m$ times).
\end{pro}

\begin{proof} Let $B=B(H)$. With
the notation in Lemma~\ref{lem2.1} we have
\begin{equation}\label{eq2.4}
 \left\|\int f^{\dot\otimes m} \dot\otimes \bar f^{\dot\otimes m}\ d\mu\right\|^{\frac1{2m}} = \max\{ I(f,\bar f, g_2,\bar g_2,\ldots, g_m,\bar g_m) ^{1/2}\}
\end{equation}
where the supremum runs over all $g_2,\ldots, g_m$ in $B\otimes L_{2m}$ such that $\big\|\int g^{\dot\otimes m}_k \dot\otimes \bar g^{\dot\otimes m}_k\ d\mu\big\| \le 1$ for any $k=1,\ldots, m$.\\
Indeed, it follows easily from \eqref{eq2.3} that the latter maximum is attained for the choice of $g_2 = \cdots= g_m = \lambda f$ with $\lambda = \big\|\int f^{\dot\otimes m} \dot\otimes \bar f^{\dot\otimes m}\ d\mu\big\|^{-1/2m}$. Thus we can proceed as in \cite{P3} for the case $m=1$: \ We assume $H = \ell_2$ to fix ideas. Let ${\cl S}$ denote the collection of all $G = g_2\dot\otimes\cdots\dot\otimes g_m$ where $(g_2,\ldots, g_m)$ runs over the set appearing in \eqref{eq2.4}. Note that by \eqref{eq2.3} we know that for any $f$ in $B\otimes L_p$ we have $f\dot\otimes G \in B(H^{\otimes m})\otimes L_2$. Then for any $G$ in ${\cl S}$ we introduce the linear map
\[
 u_G\colon \ L_p(\mu)\to B(H^{\otimes m}) \otimes_{\min} (L_2)_{oh}
\]
defined by
\[
 u_G(f) = f\dot\otimes G.
\]
Then \eqref{eq2.3} and \eqref{eq0.4} imply $\|u_G\|\le 1$. We then define the embedding
\[
 u\colon \ L_p(\mu)\to \bigoplus_{G\in {\cl S}} B(H^{\otimes m-1}) \otimes_{\min} (L_2)_{oh}
\]
by setting
\[
 u(f) = \bigoplus_G u_G(f).
\]
Then \eqref{eq2.4} gives us that for any $f$ in $B\otimes L_p$ we have
\[
 \|(id\otimes u)(f)\| = \left\|\int f^{\dot\otimes m} \dot\otimes \bar f^{\dot\otimes m}\ d\mu\right\|^{\frac1{2m}}.
\]
Thus the embedding $u$ has the required properties.
\end{proof}

\begin{defn}\label{defn2.3}
We denote by $\Lambda_p(\mu)$ $(p=2m)$ the operator space appearing in Proposition~\ref{pro2.2}. Note that $\Lambda_p(\mu)$ is isometric to $L_p(\mu)$ and for any $H$ and for any $f\in B(H)\otimes\Lambda_p(\mu)$ we have
\begin{equation}\label{eq2.4+}
 \|f\|_{B(H)\otimes_{\min}\Lambda_p(\mu)}  = \left\|\int f^{\dot\otimes m} \dot\otimes \bar f^{\dot\otimes m}\ d\mu\right\|^{\frac1{2m}}.
\end{equation}
\end{defn}

\begin{pro}\label{pro2.4}
Let $f$ be as in Definition~\ref{defn2.3} with $p=2m$. Then
\begin{equation}\label{eq2.4++}
 \|f\|_{B(H)\otimes_{\min} \Lambda_p} = \|f^{\dot\otimes m}\|^{1/m}_{B(H^{\otimes m})\otimes_{\min} L_2} = \|f^{\dot\otimes m} \dot\otimes \bar f^{\dot\otimes m}\|^{\frac1{2m}}_{B(H^{\otimes m} \otimes {\bar H}^{\otimes m}) \otimes_{\min}L_1}
\end{equation}
where $L_2$ and $L_1$ are equipped respectively with their $oh$ and maximal operator space structure. So if we make the convention
that $\Lambda_1=L_1$ equipped with its maximal o.s.s. then we have
$$\|f\|_{B(H)\otimes_{\min} \Lambda_{2m}} =
\|(f\dot\otimes \bar f)^{\dot\otimes m}\|_{B( {(H  \otimes {\bar H})}^{\otimes m})\otimes_{\min} \Lambda_1}^{\frac1{2m}}.$$
\end{pro}

\begin{proof}
The first equality is immediate since it is easy to check that for any $g\in B(H)\otimes L_2$ we have
\[
 \|g\|_{\min} = \left\|\int g\dot\otimes \bar g \ d\mu\right\|^{1/2}_{B(H\otimes\ovl H)}.
\]
For the second one, we use \eqref{eq0.2}: for any $\varphi$ in $B(H)\otimes L_1$
\[
 \|\varphi\|_{B(H)\otimes_{\min} L_1} = \sup\left\{\left\|\int \varphi\dot\otimes \psi \ d\mu\right\|\right\}
\]
where the supremum runs over all $\psi$ in $B(\ell_2)\otimes L_\infty$ with $\|\psi\|_{\min} \le 1$. Applying this to $\varphi(t) = f(t)^{\otimes m} \otimes \ovl{f(t)}^{\otimes m}$ we find using 
  \eqref{eq2.2bis} 
\[
 \|f^{\otimes m}\otimes \bar f^{\otimes m}\|_{B(H)\otimes_{\min} L_1} \le  \|f\|^{2m}_{B(H)\otimes_{\min} \Lambda_p} ,
\]
and the choice of $\psi\equiv 1$ shows that this inequality is actually an equality.
\end{proof}

\n{\bf Notation:} For any $f\in B(H)\otimes L_p(\mu)$ (or equivalently $f\in B(H)\otimes\Lambda_p(\mu)$), we denote
$$\|f\|_{(p)}= \|f\|_{B(H)\otimes_{\min}\Lambda_p(\mu)} .$$
The preceding Proposition shows that if $p=2m$, we have
$$\|f\|_{(p)}=\| f^{\dot\otimes {p/2}}\|_{(2)}^{2/p}=\|(f\dot\otimes \bar f)^{\dot\otimes {p/2}}\|_{(1)}^{1/p}.$$
Note that by \eqref{eq1.7}, if $g,f\in B(H)\otimes L_p(\mu)$ and $p=2m$
\begin{equation}\label{eq3.as}\left( g(\omega) \dot\otimes \bar g(\omega)\prec f(\omega) \dot\otimes \bar f (\omega)    \quad\forall_{a.s} \ \omega\in \Omega\right)\Rightarrow \|g\|_{(p)} \le \|f\|_{(p)}.
\end{equation}
\begin{cor} The conditional expectation ${\bb E}^{\cl B}$ is  completely contractive on $\Lambda_p$ for any $p\in 2\N$.
\end{cor}
\begin{proof} The case $p=2$ is clear by the homogeneity of the spaces $OH$. Using \eqref{eq1.2}
one easily passes from $p$ to $2p$. By induction, this proves the Corollary for 
any $p$ of the form $p=2^m$. For $p$ equal to an arbitrary even integer, we need
a different argument. Let $f\in B(H)\otimes \Lambda_p$ and  $g={\bb E}^{\cl B} f$. By the 
classical property of conditional expectations,
we have
$$\|g\|^p_{(p)}=\|\int  g\dot\otimes \bar g \dot\otimes   g\dot\otimes \bar g\cdots d\mu\|= \|  \int  f\dot\otimes \bar g \dot\otimes   g\dot\otimes \bar g\cdots d\mu \|$$
and hence by \eqref{eq2.3}
$\|g\|^p_{(p)}\le  \|f\|_{(p)}\|g\|^{p-1}_{(p)}$ or equivalently $\|g\|_{(p)}\le  \|f\|_{(p)}$.
\end{proof}

With the above definition, Lemma \ref{lem2.1} together with \eqref{eq2.2bis} immediately implies
(see the beginning of \S \ref{sec7} for the definition of jointly completely contractive multilinear maps):
\begin{cor}\label{cor2.4-} Let $p=2m$.
The product mapping $(f_1,\ldots, f_{2m})\mapsto f_1\times \ldots \times f_{2m}$
is a   jointly completely contractive multilinear map from
$\Lambda_p(\mu)^{2m}$ to $\Lambda_1(\mu)=L_1(\mu)$.
\end{cor}
\begin{proof} By Lemma \ref{lem2.1} and \eqref{eq2.2bis}
we have  for any $\psi \in B(\ell_2)\otimes L_\infty$ with
norm $\|\psi\|_{\min}\le 1$ 
$$\|\int   f_1\dot\otimes\cdots\dot\otimes f_{p}\dot\otimes\psi\|
 \le \prod_1^p\|f_j\|_{(p)}.$$
Taking the supremum over all such $\psi$ we obtain
$$\|  f_1\dot\otimes\cdots\dot\otimes f_{p}\|_{(1)} \le \prod_1^p\|f_j\|_{(p)},$$
which is nothing but a reformulation of the assertion of this corollary.
\end{proof}
\begin{cor}\label{cor2.4}
Let $p\ge q\ge 2$ be even integers. If $\mu$ is a probability, the inclusion
$L_p(\mu) \subset L_q(\mu) $ is a complete contraction from
$\Lambda_p(\mu)$ to $\Lambda_q(\mu)$.
\end{cor}
\begin{proof} Take $p=2m$, $q=2n$ and $r=2m-2n$ with $n< m$.
Let $f\in B(H)\otimes L_p(\mu)$ and $g\in B(H)\otimes L_p(\mu)$.
Consider  
$f_1\otimes\cdots\otimes f_{2m}=(f\otimes \bar f)^{n}\otimes {g}^{\otimes   r}$,
and let us choose for $g$ the constant function that is identically equal to the identity operator on $H$.
Then, using \eqref{eq2.4+},  \eqref{eq2.3} yields
$$  \|f\|_{B(H)\otimes_{\min}\Lambda_q(\mu)}^q  \le   \|f\|_{B(H)\otimes_{\min}\Lambda_p(\mu)}^{q}\times 1^{p-q} $$
and hence $  \|f\|_{B(H)\otimes_{\min}\Lambda_q(\mu)}   \le   \|f\|_{B(H)\otimes_{\min}\Lambda_p(\mu)}$.
\end{proof}

\begin{rem}
Let $f\in B(H) \otimes L_\infty(\mu)$. We may view $f$ as an essentially bounded $B(H)$-valued function. Then if $\mu$ is finite we have
\begin{equation}\label{eq3.9}
 \|f\|_{L_\infty(\mu;B(H))} = \lim_{p\to\infty}\|f\|_{(p)}.
\end{equation}
Indeed, we may assume $\mu(\Omega)=1$ and it clearly suffices to prove that for any measurable subset $A\subset\Omega$ with $\mu(A)>0$ we have
\begin{equation}\label{eq3.10}
 \left\|\mu(A)^{-1} \int\limits_A f\ d\mu\right\| \le\lim_{p\to\infty} \|f\|_{(p)}.
\end{equation}
To verify this let $f_A = \mu(A)^{-1} \int\limits_A f\ d\mu$. By \eqref{eq1.8} with $p=2^k$ we have
\begin{align*}
 \|f_A\| = \|(f_A\otimes \bar f_A)^{\otimes p}\|^{\frac1{2p}} &\le \left\| \mu(A)^{-1} \int\limits_A (f\dot\otimes \bar f)^{\dot\otimes p}\right\|^{\frac1{2p}}\\
&\le \mu(A)^{-\frac1{2p}} \|f\|_{(2p)}
\end{align*}
where we also use \eqref{eq1.1+++} and Lemma \ref{lem1.1}, 
and hence letting $p\to\infty$ we obtain \eqref{eq3.10}.
\end{rem}

\section{Martingale inequalities in $\pmb{\Lambda_p}$}\label{sec4}

\indent 

Our main result is the following one. This is an operator valued version of Burkholder's martingale inequalities. Although our inequality seems very different from the one appearing in \cite{PX1}, the method used to prove it is rather similar.

Throughout this section $H$ is an arbitrary Hilbert space and we set $B=B(H)$. We give ourselves a probability space $(\Omega, {\cl A}, {\bb P})$ and we set $L_p = L_p(\Omega, {\cl A}, {\bb P})$. We also give ourselves a filtration $({\cl A}_n)_{n\ge 0}$ on $(\Omega, {\cl A}, {\bb P})$. We assume that ${\cl A}_0$ is trivial and that ${\cl A}_\infty = \sigma\big(\bigcup_{n\ge 0} {\cl A}_n\big)$ is equal to ${\cl A}$.

We may view any $f$ in $B\otimes L_p$ as a $B$-valued random variable $f\colon \ \Omega\to B$. We denote ${\bb E}_nf = {\bb E}^{{\cl A}_n}f$, $d_0 = {\bb E}_0f$ and $d_n = {\bb E}_nf-  {\bb E}_{n-1}f$ for any $n\ge 1$. We will say that $f\in B\otimes L_p$ is a test function if it is ${\cl A}_N$-measurable for some $N\ge 1$, or equivalently if $f$ can be written as a \emph{finite} sum $f = \sum^\infty_0\ d_n$. We then denote
\[
 S(f) = \sum\nolimits^\infty_0 d_n \dot\otimes \bar d_n.
\]
We also denote
\[
 \sigma(f) = d_0 \dot\otimes \bar d_0  + \sum\nolimits^\infty_1 {\bb E}_{n-1}(d_n \dot\otimes \bar d_n).
\]
We recall the notation
\[
 x^{\dot\otimes m} = x\dot\otimes \cdots\dot\otimes x
\]
where $x$ is repeated $m$ times. Note that $\sigma(f)\in B\otimes \ovl B\otimes  L_{p/2}$ and also that
\[
 (d_n\dot\otimes \bar d_n)^{\dot\otimes p/2} \in (B\otimes\ovl B)^{\otimes p/2} \otimes L_1.
\]

\begin{thm}\label{thm3.1}
 For any $p\ge 2$ of the form $p=2^k$ for some $k\ge 1$ there are positive constants $C_1(p), C_2(p) $ such that for any test function $f$ in $B\otimes L_p$ we have
\begin{equation}\label{eq3.1+}
 C_1(p)^{-1}\|S(f)\|^{1/2}_{B\otimes_{\min} \ovl B\otimes_{\min} \Lambda_{p/2}} \le \|f\|_{B \otimes_{\min} \Lambda_p} \le C_2(p)\|S(f)\|^{1/2}_{B\otimes_{\min} \ovl B\otimes_{\min} \Lambda_{p/2}}.
\end{equation}

\end{thm}

\noindent{\it First part of the proof.}
We start by the case $p=4$. Let $f = \sum d_n$, $S=S(f)$ and $\sigma=\sigma(f)$. Then
\[
 f\dot\otimes \bar f = S + a + b
\]
where $a = \sum d_n \dot\otimes \bar f_{n-1}$ and $b = \sum f_{n-1}\dot\otimes \bar d_n$. Let $g=a+b$ so that $f\dot\otimes \bar f - S=g$. Note that by \eqref{eq1.1} applied pointwise (or, say, alomst surely) 
\[
 g\dot\otimes \bar g \prec 2(a\dot\otimes \bar a + b\dot\otimes \bar b)
\]
and hence ${\bb E}(g\dot\otimes \bar g) \prec 2{\bb E}(a\dot\otimes \bar a) + 2{\bb E}(b\dot\otimes \bar b)$. By Lemma \ref{lem1.1} we have
\begin{equation}\label{eq3.1}
 \|{\bb E}(g\dot\otimes \bar g)\| \le 2\|{\bb E}(a\dot\otimes \bar a)\| + 2{\bb E}(b\dot\otimes \bar b)\|.
\end{equation}
 Note that for any $f$
\[
 \|{\bb E}(f\dot\otimes \bar f)\|^{1/2} = \|f\|_{B\otimes_{\min} \Lambda_2}.
\]
By orthogonality ${\bb E}(a\dot\otimes \bar a) ={\bb E}( \sum d_n \dot\otimes \bar f_{n-1} \dot\otimes \bar d_n \dot\otimes f_{n-1})$. By \eqref{eq1.2}   we have
\[
 \bar f_{n-1} \dot\otimes f_{n-1} \prec {\bb E}_{n-1}(\bar f\dot\otimes f)
\]
therefore by \eqref{eq-or}
\[
 d_n \dot\otimes \bar d_n \dot\otimes \bar f_{n-1} \dot\otimes f_{n-1} \prec d_n \dot\otimes \bar d_n \dot\otimes {\bb E}_{n-1}(\bar f\dot\otimes f)
\]
and hence
\[
{\bb E}(d_n\dot\otimes \bar d_n \dot\otimes \bar f_{n-1} \dot\otimes f_{n-1}) \prec {\bb E}({\bb E}_{n-1}(d_n \dot\otimes \bar d_n) \dot\otimes \bar f\dot\otimes f)
\]
which yields (after a suitable permutation) by Lemma \ref{lem1.1}
\[
 \|{\bb E}(a\dot\otimes \bar a)\| \le \|{\bb E}(\sigma\dot\otimes \bar f\dot\otimes f)\|,
\]
and hence by Haagerup's Cauchy--Schwarz inequality
\begin{equation}\label{eq3.2}
\|{\bb E}(a\dot\otimes \bar a)\| \le  \|{\bb E}(\sigma\dot\otimes \bar \sigma)\|^{1/2} \|{\bb E}(\bar f\dot\otimes f\dot\otimes f\dot\otimes \bar f)\|^{1/2}.
\end{equation}
Similarly we find
\begin{equation}\label{eq3.3}
 \|{\bb E}(b\dot\otimes \bar b)\| \le  \|{\bb E}(\sigma\dot\otimes \bar \sigma)\|^{1/2} \|{\bb E}(f\dot\otimes \bar  f\dot\otimes \bar f\dot\otimes  f)\|^{1/2}.
\end{equation}
We now claim that
\[
 \|{\bb E}(\sigma\dot\otimes\bar\sigma)\|^{1/2} \le 2\|{\bb E}(S\dot\otimes\ovl S)\|^{1/2}.
\]
Using this claim the conclusion is easy:\ By \eqref{eq3.1}, \eqref{eq3.2}, \eqref{eq3.3} we have
\[
 \|{\bb E}(g\dot\otimes \bar g)\| \le 8\|{\bb E}(S\dot\otimes \ovl S)\|^{1/2} \|{\bb E}(f\dot\otimes \bar f\dot\otimes f\dot\otimes \bar f)\|^{1/2}.
\]
But now $g\to \|{\bb E}(g\dot\otimes \bar g)\|^{1/2} = \|g\|_{B\otimes_{\min} L_2}$ is a norm so by its subadditivity, recalling $g=f\dot\otimes \bar f -S$, we have
\[
 \big|\|f\dot\otimes \bar f\|_{B\otimes_{\min} L_2} - \|S\|_{B\otimes_{\min} L_2}\big|\le \|g\|_{B\otimes_{\min} L_2} \le 8^{1/2}{xy}
\]
where $x^2 = \|f\dot\otimes\bar f\|_{B_{\min}\otimes L_2}$ and $y^2 = \|S\|_{B\otimes_{\min} L_2}$. Equivalently, we have
\[
 |x^2-y^2|\le 8^{1/2}{xy}
\]
from which it immediately follows that
\[
 \max\left\{\frac{x}y, \frac{y}x\right\} \le \sqrt 2 + \sqrt 3.
\]
 Thus, modulo our claim, we obtain the announced inequality 
\eqref{eq3.1+} for $p=4$. To prove the claim we note that it is a particular case of the ``dual Doob inequality'' appearing in the next Lemma.
\qed

\begin{lem}\label{lem3.2}
Let $\theta_1,\ldots, \theta_N$ be arbitrary in $B\otimes L_4$. Let $\alpha = \sum^N_1 {\bb E}_n(\theta_n \dot\otimes \bar\theta_n)$ and $\beta = \sum^N_1 \theta_n\dot\otimes \bar\theta_n$. Then
\[
 \|{\bb E}(\alpha\dot\otimes\bar\alpha)\|^{1/2} \le 2\|{\bb E}(\beta\dot\otimes\bar \beta)\|^{1/2}.
\]
\end{lem}

\begin{proof}
Let $\alpha_n = {\bb E}_n(\theta_n\dot\otimes \bar\theta_n)$ and $\beta_n = \theta_n\dot\otimes \bar\theta_n$. Note that $\beta_n\succ 0$ and $\alpha_n\succ 0$. We have $\alpha \dot\otimes \bar\alpha = \sum_{n,k} \alpha_n \dot\otimes \bar\alpha_k$ and hence

\begin{align*}
 {\bb E}(\alpha\dot\otimes\bar\alpha) &= {\bb E}\left(\sum_{n\le k} {\bb E}_n\beta_n \dot\otimes \ovl{{\bb E}_k\beta_k}\right) + {\bb E}\left(\sum_{n>k} {\bb E}_n \beta_n \dot\otimes \ovl{{\bb E}_k\beta_k}\right)\\
&= {\bb E} \left(\sum\nolimits_{n\le k} ({\bb E}_n\beta_n)\dot\otimes \bar\beta_k\right) + {\bb E}\left(\sum\nolimits_{n>k} \beta_n \dot\otimes (\ovl{{\bb E}_k\beta_k})\right)\\
&= \text{I} + \text{II.}
\end{align*}
Using a suitable permutation (as explained before \eqref{eq-or}) we have by \eqref{eq-or} or \eqref{eq1.7bbis}
\[
 ({\bb E}_n\beta_n) \dot\otimes \bar\beta_k\succ 0\quad {\rm and}\quad \beta_n \dot\otimes (\ovl{{\bb E}_k\beta_k})\succ 0.\]
 Therefore, using \eqref{eq1.1+},  it follows from Lemma \ref{lem1.1}  that
\[
 \|\text{I}\|\le \left\|\sum\nolimits_{n,k} {\bb E}(({\bb E}_n\beta_n) \dot\otimes \bar \beta_k)\right\| = \|{\bb E}(\alpha \dot\otimes \bar\beta)\|,
\]
and hence by \eqref{eq2.2}
\[
 \|\text{I}\| \le \|{\bb E}(\alpha\dot\otimes \bar\alpha)\|^{1/2} \|{\bb E}(\beta \dot\otimes \bar\beta)\|^{1/2}.
\]
A similar bound holds for $\|\text{II}\|$. Thus we obtain 
\[
 \|{\bb E}(\alpha\dot\otimes\bar\alpha)\| \le \|\text{I}\| + \|\text{II}\| \le 2 \|{\bb E}(\alpha\dot\otimes\bar\alpha)\|^{1/2} \|{\bb E}(\beta\dot\otimes\bar\beta)\|^{1/2}.
\]
After division by $\|{\bb E}(\alpha\dot\otimes \bar\alpha)\|^{1/2}$ we find the inequality in Lemma~\ref{lem3.2}.
\end{proof}

We will need to extend Lemma~\ref{lem3.2} as follows:

\begin{lem}\label{lem3.3} Let  $m\ge 1$ be any integer. Let $\theta_1,\ldots, \theta_N$ be arbitrary in $B\otimes L_{2m}$.
Let $\alpha,\beta$ be as in the preceding lemma. Then  
\begin{equation}\label{eq3.4}
\|{\bb E}(\alpha ^{\dot\otimes m})\|\le m^m\|{\bb E}(\beta^{\dot\otimes m})\|.
\end{equation}
\end{lem}

\begin{proof} 
Note that up to a permutation of factors $(\alpha )^{\dot\otimes 2}$
and $\alpha\dot\otimes\bar\alpha$ are the same. In any case
$\|{\bb E}((\alpha )^{\dot\otimes 2})\|=\|{\bb E}(\alpha\dot\otimes\bar\alpha)\|$, so the case $m=2$ follows from the preceding lemma (and $m=1$ is trivial).\\
We use the same notation as in the preceding proof. We can write
\[
 (\alpha )^{\dot\otimes m} = \sum_{n(1),\ldots, n(m)} \alpha_{n(1)} \dot\otimes  \alpha_{n(2)} \dot\otimes\cdots\dot\otimes  \alpha_{n(m)}.
\]
We can partition the set of $m$-tuples $n = (n(1),\ldots, n(m))$ into subsets $S_1,\ldots, S_m$ so that we have 
\begin{align}
 n(1) &=\max_j n(j), \tag*{$\forall n\in S_1$}\\
 n(2) &= \max_j n(j), \tag*{$\forall n\in S_2$}
\end{align}
and so on. Let then
$T(j) = {\bb E} \left(\sum_{n\in S_j} \alpha_{n(1)} \dot\otimes\cdots \dot\otimes \bar\alpha_{n(m)}\right)$. Consider $j=1$ for simplicity, arguing as in the preceding proof, we have
\begin{align*}
T(1) &= \sum_{n\in S_1} {\bb E}(\alpha_{n(1)} \dot\otimes\cdots \dot\otimes  \alpha_{n(m)})\\
&= \sum_{n\in S_1} {\bb E} (\beta_{n(1)} \dot\otimes  \alpha_{n(2)}\cdots \dot\otimes  \alpha_{n(m)})\\
&\prec \sum_{n(1),\ldots, n(m)} {\bb E}(\beta_{n(1)} \dot\otimes \alpha_{n(2)} \cdots \dot\otimes \alpha_{n(m)}) = {\bb E} (\beta\dot\otimes   \alpha^{\dot\otimes m-1})
\end{align*}
and hence by \eqref{eq2.3}
\[
 \|{\bb E}T(1)\|\le \|{\bb E}(\beta  \dot\otimes  \alpha^{\dot\otimes m-1}) \| \le \|{\bb E}(\beta^{\dot\otimes m})\|^{\frac1m} \|{\bb E}(\alpha^{\dot\otimes m})\|^{\frac{m-1}{m}}.
\]
A similar bound holds for each $T(j)$ $(j=1,\ldots, m)$. Thus we find
\[
 \|{\bb E}(\alpha^{\dot\otimes m})\| = \left\|\sum\nolimits^{m}_1 {\bb E} T(j)\right\| \le m \|{\bb E}(\beta^{\dot\otimes m})\|^{\frac1m} \|{\bb E}(\alpha^{\dot\otimes m}\|^{\frac{m-1}{m}},
\]
and again after a suitable division we obtain \eqref{eq3.4}.
\end{proof}

\begin{proof}[Proof of Theorem \ref{thm3.1} in the dyadic case.] Assume that $d_n\dot\otimes \bar d_n$ is ${\cl A}_{n-1}$-measurable. Note that this holds in the dyadic case when $\Omega =\{-1,1\}^{\bb N}$ as well as for the filtration naturally associated to the Haar orthonormal system. In that case we can give a short proof of the following inequality for any $m\ge 1$
\begin{equation}\label{eq3.8}
\|{\bb E}(S(a)^{\dot\otimes m})\|^{\frac1m} \le m \|{\bb E}(S^{\dot\otimes 2m})\|^{\frac1{2m}} \|{\bb E}((f\dot\otimes \bar f)^{\dot\otimes 2m})\|^{\frac1{2m}},
\end{equation}
 and a similar bound for $\|{\bb E}(S(b)^{\dot\otimes m})\|^{\frac1m}$.\\
Indeed, recall $a = \sum d_n \dot\otimes \bar f_{n-1}$. Note that $S(a)$ is up to permutation the same as 
\[
 S_1 = \sum d_n \dot\otimes \bar f_{n-1} \dot\otimes f_{n-1} \dot\otimes \bar d_n.
\]
By \eqref{eq1.2} and \eqref{eq-or}, we have
\[
 S_1 \prec \sum d_n \dot\otimes {\bb E}_{n-1}(\bar f\dot\otimes f) \dot\otimes \bar d_n = \sum {\bb E}_{n-1}(d_n \dot\otimes \bar f \dot\otimes f   \dot\otimes \bar d_n)
\]
where the last equality holds because $d_n\dot\otimes \bar d_n$ is assumed $(n-1)$-measurable.\\
By \eqref{eq1.7} this implies
\[
 S^{\dot\otimes m}_1 \prec \left(\sum {\bb E}_{n-1}(d_n \dot\otimes \bar f \dot\otimes f \dot\otimes \bar d_n)\right)^{\dot\otimes m}
\]
and hence by \eqref{eq3.4} (recall $S=S(f)$ and the permutation invariance of the  norm)
\begin{align*}
 \|{\bb E}(S^{\dot\otimes m}_1)\| &\le m^m\left\|{\bb E}\left(\left(\sum d_n \dot\otimes \bar f \dot\otimes f\dot\otimes \bar d_n\right)^{\dot\otimes m}\right)\right\|\\
&= m^m\|{\bb E}((S\dot\otimes \bar f \dot\otimes f)^{\dot\otimes m})\|\\
\intertext{and hence by Lemma \ref{lem2.1}}
&\le m^m\|{\bb E}(S^{\dot\otimes 2m})\|^{1/2} \|{\bb E}((f\dot\otimes \bar f)^{\dot\otimes 2m})\|^{1/2}.
\end{align*}
Thus we obtain \eqref{eq3.8} as announced:
\[
 \|{\bb E}(S(a)^{\dot\otimes m})\| = \|{\bb E}(S^{\dot\otimes m}_1)\| \le m^m \|{\bb E}(S^{\dot\otimes 2m})\|^{1/2} \|{\bb E}((f\dot\otimes \bar f)^{\dot\otimes 2m})\|^{1/2}.
\]
The proof with $b$ in place of $a$ is identical. 

Using \eqref{eq3.8} and the analogue for $b$, it is easy to show (assuming
that $d_n\dot\otimes \bar d_n$ is ${\cl A}_{n-1}$-measurable) that the validity of \eqref{eq3.1+} for $p=2m$ implies its validity for $p=4m$. Indeed, let us assume \eqref{eq3.1+} for $p=2m$ and let $C =C_2(2m)$. We find, recalling $f\dot\otimes \bar f-S=a+b$
\begin{align*}
 \|f\dot\otimes\bar f-S\|_{B\otimes_{\min} \Lambda_{2m}} &\le \|a\|_{B\otimes_{\min} \Lambda_{2m}} + \|b\|_{B\otimes_{\min} \Lambda_{2m}}\\
&\le C(\|S(a)\|^{1/2}_{B\otimes_{\min} \Lambda_m} + \|S(b)\|^{1/2}_{B\otimes_{\min} \Lambda_m})\\
\intertext{and hence by \eqref{eq3.8}}
&\le 2Cm^{1/{2}} \|S\|^{1/2}_{B\otimes_{\min}\Lambda_{2m}} \|f\dot\otimes \bar f\|^{1/2}_{B\otimes_{\min} \Lambda_{2m}}.
\end{align*}
Therefore, again setting $x =\|f\dot\otimes \bar f\|_{B\otimes_{\min}\Lambda_{2m}}$ and $y = \|S\|_{B\otimes_{\min}\Lambda_{2m}}$ we find
\[
 |x-y| \le 2Cm^{1/{2}} \sqrt{xy},
\]
and we conclude as before that $x$ and $y$ are comparable, so that  
\eqref{eq3.1+}  holds for $p=4m$. 
\end{proof}

The next corollary is now immediate from the dyadic case. However,
we will later show that it is valid for any $p$ in $2\bb N$ (see Corollary \ref{cor9.1b}).
\begin{cor}  Assume
$p\ge 2$ of the form $p=2^k$ for some $k\ge 1$.
If $\Omega=\{-1,+1\}^{\bb N}$, the closed span of 
the coordinates $(\vp_n)$ (or equivalently, of the Rademacher functions
on $\Omega=[0,1]$) in $\Lambda_p$ 
 is completely isomorphic to the space $OH$, i.e. to $\ell_2$ equipped
 with the   o.s.s. of OH. Moreover, the orthogonal projection $P$ onto it  is c.b. on $\Lambda_p$.
\end{cor}
\begin{proof} Let $f=\sum x_n \vp_n$ ($x_n\in B(H)$).
By  \eqref{eq3.1+} , $\|f\|_{(p)}$ is equivalent to
$\|\sum x_n \otimes \bar x_n\|^{ 1/2}$ and the latter is equal
to the norm of $\sum e_n\otimes x_n$ in $OH\otimes_{\min} B$
where $(e_n)$ is any orthonormal basis of $OH$. Thus the closed span
of $(\vp_n)$ in $\Lambda_p$ is isomorphic to $OH$.
We skip the proof of the complementation because we give
the details for that  in the proof of Proposition \ref{probuch} below.
\end{proof}
\begin{proof}[Proof of the right hand side of \eqref{eq3.1+}]
We will use induction on $k$. The case $k=1$ is clear. Assume that the right hand side of \eqref{eq3.1+} holds for $p=m$, we will show it for $p=2m$. 
With the preceding notation, recall $g=a+b$ and hence our assumption yields
\begin{equation}
 \|g\|_{B\otimes\ovl B\otimes\Lambda_m}\le \|a\|_{B\otimes\ovl B\otimes\Lambda_m} +\|b\|_{B\otimes\ovl B\otimes\Lambda_m}  \le C_2(m) (\|S(a)\|^{1/2}_{\bullet} + \|S(b)\|^{1/2}_\bullet)
\end{equation}
where the dot stands for $B\otimes_{\min} \ovl B \otimes_{\min} \ovl B \otimes_{\min} B \otimes_{\min}  \Lambda_{{m}/2}$. Since by  \eqref{eq1.1}
 $$\bar f_{n-1}\dot\otimes f_{n-1} \prec 2(\bar f_n  \dot\otimes f_n + \bar d_n \dot\otimes d_n)$$ 
 we have using \eqref{eq-or} and
   \eqref{eq1.1-} (in a suitable permutation)
\[
 \|S(a)\|_{\bullet}^{1/2} = \left\|\sum d_n \dot\otimes \bar f_{n-1} \dot\otimes \bar d_n \dot\otimes f_{n-1} \right\|^{1/2}_\bullet \le \text{I} + \text{II}
\]
where
\[
2^{-1/2} \text{I} = \left\|\sum d_n\dot\otimes \bar f_n \dot\otimes \bar d_n \dot\otimes f_n\right\|^{\frac12}_\bullet \quad \text{and}\quad 2^{-1/2}\text{II} = \left\|\sum d_n \dot\otimes \bar d_n \dot\otimes \bar d_n \dot\otimes d_n \right\|^{\frac12}_\bullet.
\]
Note that for any $F$ in $B\otimes \ovl B \otimes \ovl B \otimes B \otimes \Lambda_{m/2}$
we have $ \|F\|_\bullet = \|{\bb E}((F\dot\otimes\ovl F)^{\dot\otimes m/4})\|^{\frac2m}$. \\
Recall that, by \eqref{eq1.2}, $f_n\dot\otimes \bar f_n \prec {\bb E}_n(f\dot \otimes \bar f)$. Thus we have by 
\eqref{eq-or} \eqref{eq1.1-}  and \eqref{eq1.7} 
\begin{align*}
2^{-1/2} I &\le \left\|{\bb E}\left(\left(\sum d_n \dot\otimes \bar d_n \dot\otimes {\bb E}_n(f \dot\otimes \bar f)\right)^{\dot\otimes m/2}\right)\right\|^{\frac1m}\\
&= \left\|{\bb E}\left(\left(\sum {\bb E}_n(d_n\dot\otimes \bar d_n \dot\otimes f \dot\otimes \bar f)\right)^{\dot\otimes m/2}\right)\right\|^{\frac1m}\\
\intertext{and hence by \eqref{eq3.4} and \eqref{eq2.3} (or actually \eqref{eq2.2})}
&\le (m/2)^{1/2}\|{\bb E}((S\dot\otimes f\dot\otimes \bar f)^{\dot\otimes m/2})\|^{\frac1m}\\
&\le (m/2)^{1/2}\|{\bb E}(S^{\dot\otimes m})\|^{\frac1{2m}} \|{\bb E}((f\dot\otimes \bar f)^{\dot\otimes m})\|^{\frac1{2m}}.
\end{align*}
Moreover, recalling \eqref{diag}, we have obviously 
$0 \prec d_n \otimes \bar d_n \otimes \bar d_k \otimes  d_k $ for all $n,k$ and hence
$\sum d_n\dot\otimes \bar d_n \dot\otimes \bar d_n \dot\otimes d_n \prec S \dot\otimes \ovl S$. Therefore, again by \eqref{eq1.1-} and \eqref{eq1.7}  
\[
2^{-1/2} \text{II} \le \|{\bb E}((S\dot\otimes\ovl S \dot\otimes \ovl S \dot\otimes S)^{m/4})\|^{1/m} = \|{\bb E}(S^{\dot\otimes m})\|^{\frac1m}.
\]
Let $x = \|{\bb E}(S^{\dot\otimes m})\|^{\frac1m}$ and $y = \|{\bb E}(f\dot\otimes \bar f)^{\dot\otimes m}\|^{\frac1m}$. This yields
\[
 \|S(a)\|_\bullet^{\frac12} \le \sqrt{m} \sqrt{xy} +\sqrt{2}  x
\]
and a similar bound for $S(b)$. Thus we obtain
\[
 \|g\|_{B\otimes\ovl B \otimes\Lambda_m} \le 2 C_2(m)(\sqrt{m} \sqrt{xy} +\sqrt{2}  x).
\]
Since $g=f\dot\otimes\bar f-S$ we have
\[
 \big|\|f\dot\otimes \bar f\|_{B\otimes \ovl B\otimes\Lambda_m} - \|S\|_{B\otimes\ovl B\otimes\Lambda_m}\big| \le \|g\|_{B\otimes\ovl B\otimes\Lambda_m}
\]
and hence we obtain
\[
 |y-x|\le 2 C_2(m)(\sqrt{m} \sqrt{xy} +\sqrt{2}  x)
\]
From the latter it is clear that there is a constant $C_2(2m)$ such that 
\[
 \sqrt y \le C_2(2m)\sqrt x
\]
and this is the right hand side of \eqref{eq3.1+} for $p=2m$,
\end{proof} 

To prove the general case of both sides of \eqref{eq3.1+},
  the following Lemma will be crucial. We will use this only for $m=1$, but
  the inductive argument curiously requires to prove it for all dyadic $m$.
\begin{lem}\label{lem4.10}
Let $f\in B(H)\otimes L_{4mp}$ be a test function. As before we set $f_n = {\bb E}_nf$ and $d_n = f_n-f_{n-1}$ for  all $n\ge 1$. Let $p=2^k$ for some integer $k\ge 0$. Then, for any integer $m\ge 1$
of the form $m=2^\ell$ for some $\ell\ge 0$,  there is a constant $C=C(m,p)$ such that
\begin{equation}\label{eq4.50}
 \left\|\sum\nolimits^\infty_1(d_n\dot\otimes \bar d_n)^{\dot\otimes m} \dot\otimes (f_{n-1} \dot\otimes \bar f_{n-1})^{\dot\otimes m}\right\|_{(p)} \le C\|S\|^m_{(2mp)} \|f\|^{2m}_{(4mp)}.
\end{equation}
\end{lem}

\begin{proof}
We use induction on $k$ starting from $p=1$.
We may assume $d_0=0$ for simplicity. Let $$I(m,p) = \big\|\sum^\infty_1 (d_n\dot\otimes \bar d_n)^{\dot\otimes m} \dot\otimes (f_{n-1} \dot\otimes \bar f_{n-1})^{\dot\otimes m}\big\|_{(p)}.$$ By \eqref{eq1.8} and \eqref{eq-or} (and by the self-adjointness of ${\bb E}_{n-1}$) we have
\begin{align*}
I(m,1) &\le \left\|{\bb E}\left(\sum\nolimits^\infty_1 (d_n\dot\otimes \bar d_n)^{\dot\otimes m} \dot\otimes {\bb E}_{n-1} ((f\dot\otimes \bar f)^{\dot\otimes m})\right) \right\|\\
&= \left\|{\bb E}\left(\sum\nolimits^\infty_1 {\bb E}_{n-1}((d_n\dot\otimes \bar d_n)^{\dot\otimes m}) \dot\otimes (f\dot\otimes \bar f)^{\dot\otimes m}\right)\right\|,\\
\intertext{and hence by \eqref{eq2.2}}
&\le \|{\bb E}(\sigma_m\dot\otimes\bar\sigma_m)\|^{1/2} \|{\bb E}((f\dot\otimes \bar f)^{\dot\otimes 2m})\|^{1/2}
\end{align*}
where we have set
\[
 \sigma_m = \sum\nolimits^\infty_1 {\bb E}_{n-1}((d_n\dot\otimes \bar d_n)^{\dot\otimes m}).
\]
Note that by Lemma~\ref{lem3.2}
\[\|\sigma_m \|_{(2)} \le 2 \|\sum\nolimits^\infty_1 (d_n\dot\otimes \bar d_n)^{\dot\otimes m} \|_{(2)} 
\]
but obviously (recalling \eqref{diag}) $\sum(d_n\dot\otimes \bar d_n)^{\dot\otimes m} \prec (\sum d_n\dot\otimes \bar d_n)^{\dot\otimes m}$ and hence (recalling \eqref{eq1.7})
\begin{equation}\label{eq4.51}
\left(\sum (d_n\dot\otimes \bar d_n)^{\dot\otimes m} \right)^{\dot\otimes 2}\prec S^{\dot\otimes 2m}
\end{equation}
so we obtain
\[
 \|{\bb E}(\sigma_m\dot\otimes \bar\sigma_m)\|^{1/2} \le 2\|S\|^m_{(2m)}.
\]
Thus we find
\[
 I(m,1) \le 2\|S\|^m_{(2m)} \|f\|^{2m}_{(4m)},
\]
so that \eqref{eq4.50} holds for $p=1$ and any $m\ge 1$ with $C(m,1)=2$. 

Let us now denote by \eqref{eq4.50}$_p$ the inequality \eqref{eq4.50} meant for a given fixed $p$ but for any $m\ge 1$. We will show that for any $p\ge 2$
\[
 \eqref{eq4.50}_{p/2} \Rightarrow \eqref{eq4.50}_{p}.
\]
Assuming that $m\ge 1$ is fixed, let $x_n = (d_n\dot\otimes \bar d_n)^{\dot\otimes m}$ and $y_n = (f_{n-1} \dot\otimes \bar f_{n-1})^{\dot\otimes m}$. We write
\begin{align*}
 \sum x_n \dot\otimes y_n &= a+b\\
\text{with}\qquad a = \sum {\bb E}_{n-1}(x_n) \dot\otimes y_n \quad&\text{and}\quad
b = \sum(x_n - {\bb E}_{n-1}(x_n)) \dot\otimes y_n.
\end{align*}
We have $I(m,p) = \|\sum x_n\dot\otimes y_n\|_{(p)}$ and hence
\begin{equation}\label{eq4.52}
 I(m,p) \le \|a\|_{(p)} + \|b\|_{(p)},
\end{equation}
so it suffices to majorize $a$ and $b$ separately. We have by \eqref{eq1.8}
\[
 a \prec \sum {\bb E}_{n-1}(x_n) \dot\otimes {\bb E}_{n-1}((f\dot\otimes \bar f)^{\dot\otimes m}) = \sum {\bb E}_{n-1}({\bb E}_{n-1}(x_n) \dot\otimes (f\dot\otimes \bar f)^{\dot\otimes m}))
\]
and hence by \eqref{eq3.4} and \eqref{eq2.2}
\begin{align*}
 \|a\|_{(p)} &\le p \left\|\sum {\bb E}_{n-1}(x_n) \dot\otimes (f\dot\otimes \bar f)^{\dot\otimes m}\right\|_{(p)}\\
&\le p \|\sigma_m \dot\otimes (f\dot\otimes \bar f)^{\dot\otimes m}\|_{(p)}\\
&= p \|\sigma^{\dot\otimes p}_m \dot\otimes (f\dot\otimes \bar f)^{\dot\otimes mp}\|^{1/p}_{(1)}\\
&\le p\|\sigma^{\dot\otimes p}_m\|^{\frac1{p}}_{(2)} \|(f\dot\otimes \bar f)^{\dot\otimes mp}\|^{\frac1{p}}_{(2)}\\
&= p \|\sigma_m\|_{(2p)} \|f\|^{2m}_{(4mp)}.
\end{align*}
But now by \eqref{eq3.4} again
\begin{align*}
 \|\sigma_m\|_{(2p)} &\le 2p \left\|\sum (d_n\dot\otimes \bar d_n)^{\dot\otimes m}\right\|_{(2p)}\\
\intertext{and hence by \eqref{eq4.51}}
\|\sigma_m\|_{(2p)} &\le 2p\|S\|^m_{(2mp)}.
\end{align*}
Thus we obtain
\begin{equation}\label{eq4.53}
 \|a\|_{(p)} \le p (2p) \|S\|^m_{(2mp)} \|f\|^{2m}_{(4mp)}.
\end{equation}
We now turn to $b$. Note that since $y_n$ is ``predictable'' $\{(x_n-{\bb E}_{n-1}(x_n)) \dot\otimes y_n\}$ is a martingale difference sequence. We will apply the right hand side of \eqref{eq3.1+} to $b$. Note that
\[
 S(b) \approx \sum (x_n - {\bb E}_{n-1}(x_n)) \dot\otimes \ovl{(x_n - {\bb E}_{n-1}(x_n))} \dot\otimes y_n \dot\otimes \bar y_n,
\]
and hence by \eqref{eq1.1}
\[
 \frac12 S(b) \prec \sum x_n \dot\otimes \ovl x_n \dot\otimes y_n \dot\otimes \ovl y_n + \sum {\bb E}_{n-1} (x_n) \dot\otimes {\bb E}_{n-1}(\ovl x_n) \dot\otimes y_n \dot\otimes \ovl y_n.
\]
By \eqref{eq1.2} we get (since $y_n$ is predictable)
\[
 \frac12 S(b) \prec \sum x_n \dot\otimes \ovl x_n \dot\otimes y_n \dot\otimes \ovl y_n + \sum {\bb E}_{n-1} (x_n \dot\otimes \ovl x_n \dot\otimes y_n \dot\otimes \ovl y_n)
\]
and hence
\[
 \frac12\|S(b)\|_{(p/2)} \le \left\|\sum x_n \dot\otimes \ovl x_n \dot\otimes y_n \dot\otimes \ovl y_n\right\|_{(p/2)} + \left\|\sum {\bb E}_{n-1}(x_n\dot\otimes \ovl x_n \dot\otimes y_n \dot\otimes \ovl y_n)\right\|_{(p/2)}.
\]
By \eqref{eq3.4} this yields
\[
 \|S(b)\|_{(p/2)} \le 2(1+(p/2)) \left\|\sum x_n \dot\otimes \ovl x_n \dot\otimes y_n \dot\otimes \ovl y_n\right\|_{(p/2)}.
\]
But since
\[
 \sum x_n \dot\otimes \ovl x_n \dot\otimes y_n \dot\otimes \ovl y_n \approx \sum (d_n \dot\otimes \bar d_n)^{\dot\otimes 2m} \dot\otimes (f_{n-1} \dot\otimes \bar f_{n-1})^{\dot\otimes 2m}
\]
we may use the induction hypothesis \eqref{eq4.50}$_{p/2}$
 (with $m$ replaced by $2m$) and we obtain
\[
 \|S(b)\|_{(p/2)} \le 2(1+(p/2)) C(2m, p/2) \|S\|^{2m}_{(2mp)} \|f\|^{4m}_{(4mp)}.
\]
By the right  hand side of \eqref{eq3.1+}$_p$ we then find
\begin{align*}
 \|b\|_{(p)} &\le C_2(p) \|S(b)\|^{1/2}_{(p/2)}\\
&\le C'(m,p) \|S\|^m_{(2mp)} \|f\|^{2m}_{(4mp)}
\end{align*}
for some constant $C'(m,p)$. Thus we conclude by \eqref{eq4.52} and \eqref{eq4.53}
\[
 I(m,p) \le (p (2p)  + C'(m,p)) \|S\|^m_{(2mp)} \|f\|^{2m}_{4mp}.
\]
In other words we obtain \eqref{eq4.50}$_p$. This completes the proof of \eqref{eq4.50}$_p$ for $p=2^k$ by induction on $k$.
\end{proof}

\begin{proof}[Proof of Theorem \ref{thm3.1} (General case)]
We will show that \eqref{eq3.1+}$_p$ $\Rightarrow$ \eqref{eq3.1+}$_{2p}$. We again start from
\[
 f\dot\otimes \bar f - S =a+b
\]
where $a  = \sum d_n \dot\otimes \bar f_{n-1}$ and $b = \sum f_{n-1}\dot\otimes \bar d_n$. By the right  hand side of \eqref{eq3.1+}$_p$ we have
\begin{align*}
 \|f\dot\otimes \bar f - S\|_{(p)} &\le 2C_2(p) \left\|\sum d_n\dot\otimes \bar d_n \dot\otimes f_{n-1} \dot\otimes \bar f_{n-1}\right\|^{1/2}_{(p/2)}\\
\intertext{and hence by \eqref{eq4.50}}
&\le 2C_2(p) C(1,p/2)^{1/2} \|S\|^{1/2}_{(p)} \|f\|_{(2p)}.
\end{align*}
Thus we find a fortiori setting $C'' = 2C_2(p) C(1,p/2)^{1/2}$
\[
 \left|\|f\dot\otimes \bar f\|_{(p)} - \|S\|_{(p)}\right| \le C'' \|S\|^{1/2}_{(p)} \|f\|_{(2p)}.
\]
Thus setting again $x = \|f\|_{(2p)}$, $ y = \|S\|^{1/2}_{(p)}$ we find
\[
 |x^2-y^2| \le C'' xy
\]
and we conclude that $x$ and $y$ must be equivalent quantities, or equivalently that \eqref{eq3.1+}$_{2p}$ holds. By induction this completes the proof.
\end{proof}

\section{Burkholder-Rosenthal inequality}\label{sec5}

Let $2<p<\infty$ be fixed. The usual form of the Burkholder-Rosenthal inequality expresses the equivalence, for scalar valued martingales, of  $\|\sum d_n\|_p$ and
\begin{equation}\label{eq5.10}
  {BR}_{\infty}=\|\sigma\|_p + \|\sup |d_n|\|_p.
\end{equation}
It is easy to deduce from that the equivalence of that same norm with 
\begin{equation}\label{eq5.11}
  {BR}_{q}= \|\sigma\|_p + \left\|\left(\sum |d_n|^q\right)^{1/q}\right\|_p
\end{equation}
for any $q$ such that $2<q\le \infty$. \\ 
Indeed, we have obviously ${BR}_{\infty} \le {BR}_{q}$. Conversely,
using (here $\frac1q = \frac{1-\theta}2 + \frac\theta\infty$)
\[
 \left\|\left(\sum |d_n|^q\right)^{1/q}\right\|_p \le \|S\|^{1-\theta}_p \|\sup|d_n|\|^\theta_p
\]
and the equivalence $\|S\|_p \simeq \|\sum d_n\|_p$, one can easily deduce that
there is a constant $C'$ such that 
\[ {BR}_{q} \le C'   \|\sum d_n\|^{1-\theta}_p {BR}^\theta_{\infty } .
\]
Thus 
 an inequality of the form
 \[
\left\|\sum d_n\right\|_p   \le C {BR}_{\infty}
\]
implies ``automatically'' 
 \[
{BR}_{\infty} \le {BR}_{q} \le C'  C^{1-\theta} {BR}_{\infty}.
\]
Similarly,  
 \[
\left\|\sum d_n\right\|_p   \le C {BR}_{q}\Rightarrow
 \left\|\sum d_n\right\|_p   \le C  C'   \|\sum d_n\|^{1-\theta}_p {BR}^\theta_{\infty }\Rightarrow
 \left\|\sum d_n\right\|_p   \le (C  C')^{1/\theta} {BR}_{\infty }.
\] 
Thus, modulo simple manipulations, the Burkholder-Rosenthal inequality reduces to the equivalence
for some $q$ such that $2<q\le \infty$ of  $\|\sum d_n\|_p$ and ${BR}_{q}$.

Note that the one sided inequality expressing that ${BR}_{q} =\|\sigma\|_p + \|(\sum |d_n|^q)^{1/q}\|_p$ is dominated by $\|\sum d_n\|_p$ reduces obviously to
\[
 \|\sigma\|_p \le C\left\|\sum|d_n|^2\right\|^{1/2}_{p/2}
\]
that holds for $p\ge 2$ by Burkholder--Davis--Gundy dualization of Doob's inequality. Therefore,
the novelty of the Burkholder-Rosenthal inequality is the fact that there
is a constant $C'''$ such that
$$\|\sum d_n\|_p \le C''' {BR}_{q}.$$
In the original Rosenthal inequality, restricted to sums of independent $d_n$'s,
or in the non-commutative version of  \cite{JX0,JX2},
the value $q=p$ is the most interesting choice. In the inequalities below,
for $p=2^k\ge 4$, we 
will work with $q=4$.\\

We will use the following extension of \eqref{eq2.2}.

\begin{pro}\label{pro3.4}
For any integer $m\ge 1$ and finite sequences $(a_k)$,  $(b_k)$ in $B(H) \otimes L_{2m}$ we have
\begin{equation}\label{eq3.9-}
 \left\|{\bb E}\left(\left(\sum a_k\dot\otimes b_k\right)^{\dot\otimes m}\right)\right\|\le \left\|{\bb E}\left(\left(\sum a_k\dot\otimes \bar a_k\right)^{\dot\otimes m}\right)\right\|^{1/2} \left\|{\bb E}\left(\left(\sum b_k\dot\otimes \bar b_k\right)^{\dot\otimes m}\right)\right\|^{1/2}.
\end{equation}
More generally, consider finite sequences $(a^{(j)}_k), (b^{(j)}_k)$ in $B(H)\otimes L_{2m}$ for $j=1,\ldots, m$. \\ Let $T_j = \sum_k a^{(j)}_k\dot{\otimes}\, b^{(j)}_k$ and let $\alpha_j = \sum_k a^{(j)}_k \otimes \ovl{a^{(j)}_k}$ and $\beta_j =  \sum_k b^{(j)}_k\otimes \ovl{b^{(j)}_k}$. We have then 
\begin{equation}\label{eq3.9+}
\|{\bb E}(T_1\dot{\otimes}\cdots \dot{\otimes}\, T_m)\| \le \|{\bb E}(\alpha_1\dot{\otimes} \cdots\dot{\otimes}\, \alpha_m)\|^{1/2} \|{\bb E}(\beta_1\otimes\cdots\otimes \beta_m)\|^{1/2}.
\end{equation}

\end{pro}

\begin{proof}
Up to permutation, ${\bb E}((\sum a_k\dot\otimes b_k)^{\dot\otimes m})$ is the same as
\[
 {\bb E}\left(\sum_{k(1),\ldots, k(m)} a_{k(1)} \dot\otimes\cdots\dot\otimes a_{k(m)} \dot\otimes b_{k(1)} \dot\otimes\cdots\dot\otimes b_{k(m)}\right).
\]
Therefore, by \eqref{eq2.2} we have 
\begin{align*}
 \left\|{\bb E}\left(\left( \sum a_k\dot\otimes b_k\right)^{\dot\otimes m}\right)\right\| \le &\left\|{\bb E}\sum_{k(1),\ldots, k(m)} a_{k(1)} \ldots a_{k(m)} \dot\otimes \bar a_{k(1)} \ldots \bar a_{k(m)}\right\|^{\frac12}\\
\times &\left\|{\bb E}\sum_{k(1),\ldots, k(m)} b_{k(1)} \ldots b_{k(m)} \dot\otimes \bar b_{k(1)}\ldots \bar b_{k(m)}\right\|^{\frac12}\\
= &\left\|{\bb E}\left(\left(\sum a_k\dot\otimes \bar a_k\right)^{\dot\otimes m}\right) \right\|^{\frac12} \left\|{\bb E}\left(\left(\sum b_k\dot\otimes \bar b_k\right)^{\dot\otimes m}\right)\right\|^{\frac12}.
\end{align*}
Up to permutation $T_1\dot{\otimes}\cdots\dot{\otimes}\, T_m$ is the same as
\[
 \sum_{k(1),\ldots, k(m)} a^{(1)}_{k(1)} \dot{\otimes} \cdots\dot{\otimes}\, a^{(m)}_{k(m)} \dot{\otimes}\, b^{(1)}_{k(1)} \dot{\otimes}\cdots \dot{\otimes}\, b^{(m)}_{k(m)},
\]
 which can be written as $\sum_k a_k\dot{\otimes}\, b_k$ with $k=(k(1),\ldots, k(m))$, $a_k = a^{(1)}_{k(1)} \dot{\otimes}\cdots\dot{\otimes}\, a^{(m)}_{k(m)}$ and $b_k = b^{(1)}_{k(1)} \dot{\otimes}\cdots\dot{\otimes}\, b^{(m)}_{k(m)}$. Therefore \eqref{eq3.9+} follows from the  $m=1$ case  of \eqref{eq3.9-}.

\end{proof}

\begin{rem}
Let $H=\ell_2$ and $B=B(H)$. The preceding Proposition shows that
\begin{equation}\label{eq4.20}
\left\|\sum a_k \dot\otimes \bar a_k\right\|^{\frac12}_{(m)} = \sup\left\{\left\|\sum a_k \dot\otimes b_k\right\|_{(m)}\right\}
\end{equation}
where the supremum runs over the set $\cl D$ of all finite sequences $(b_k)$ in $B\otimes L_{2m}$ such that $\|\sum b_k\dot\otimes b_k\|_{(m)}\le 1$. (Indeed the sup is attained for $b_k = \bar a_k$, suitably normalized.) Thus \eqref{eq4.20} allows us to define an o.s.s.\ on the space $L_{2m}(\Omega, \mu; \ell_2)$, corresponding to ``$\Lambda_{2m}$ with values in $OH$''. Indeed, we can proceed as before for $\Lambda_p$: \ we consider the subspace $E_0 \subset L_{2m}\otimes \ell_2$ formed of all finite sums $\sum a_k\otimes e_k$ $(a_k\in L_{2m})$ and we define 
\[
 J\colon \ E_0 \longrightarrow \bigoplus_{(b_k)\in \cl D} (\Lambda_m \otimes_{\min} B)
\]
by
\[
 J\left(\sum a_k\otimes e_k\right) = \bigoplus_{(b_k)\in \cl D} \sum a_k \dot\otimes b_k.
\]
This produces an o.s.s.\ on $L_{2m}(\mu;\ell_2)$. It is easy to see that if $a\in L_{2m}$ is fixed in the unit sphere, the restriction of $J$ to $a\otimes\ell_2$ induces on $\ell_2$ the o.s.s.\ of $OH$ while if $x\in \ell_2$ is fixed in the unit sphere, restricting $J$ to $L_{2m}\otimes x$ induces on $L_{2m}$ the o.s.s.\ of $\Lambda_{2m}$.\\
Note in passing that, in sharp contrast with \cite{P4}, except for the preceding special case, we do not have any reasonable definition to propose for the ``vector valued"
analogue of the $\Lambda_p$ spaces.
\end{rem}

As a consequence we find an analogue of Stein's inequality (here we could obviously replace ${\bb E}_{n-1}$
by ${\bb E}_{n}$):

\begin{cor}\label{cor3.5} Let $x_n$ be an arbitrary
finite sequence in $B(H)\otimes L_{4m}$. \\
Let 
$v = \sum {\bb E}_{n-1}(x_n\dot\otimes \bar x_n) \dot\otimes {\bb E}_{n-1}(\bar x_n \dot\otimes x_n)$ and $\delta = \sum x_n \dot\otimes \bar x_n \dot\otimes \bar x_n \dot\otimes x_n$. Then for any integer $m\ge 1$
\begin{equation}\label{eq3.11}
 \|{\bb E}(v^{\dot\otimes m})\| \le m^{m}\|{\bb E}(\delta^{\dot\otimes m})\|.
\end{equation}
\end{cor}

\begin{proof} Let $w = \sum {\bb E}_{n-1}( x_n \dot\otimes \bar x_n \dot\otimes \bar x_n \dot\otimes  x_n)$.
By \eqref{eq1.2} we have $v\prec w$. Then by \eqref{eq1.7}, \eqref{eq1.1++}, Lemma \ref{lem1.1} and \eqref{eq3.4},
we have
$$\|{\bb E}(v^{\dot\otimes m})\| \le \|{\bb E}(w^{\dot\otimes m})\| \le  m^{m}\|{\bb E}(\delta^{\dot\otimes m})\|.$$
\end{proof}

\begin{lem}\label{lem3.7}
Let $p=2^k\ge 4$ as before. Let
$  \delta = \sum d_n \dot\otimes \bar d_n \dot\otimes \bar d_n \dot\otimes d_n$.  There is a constant $C_4(p)$ such  that
\begin{equation}\label{eq3.12}
 \|{\bb E}(S^{\dot\otimes p/2})\|^{1/p} \le C_4(p) [\|{\bb E}(\sigma^{\dot\otimes p/2})\|^{\frac1p} + \|{\bb E}(\delta^{\dot\otimes p/4})\|^{\frac1p}].
\end{equation}
\end{lem}

\begin{proof} 
Note that
\[
 S-\sigma = \sum dc_n
\]
where $c_n = d_n\dot\otimes\bar d_n - {\bb E}_{n-1}(d_n\dot\otimes\bar d_n)$. Thus by the right hand side of \eqref{eq3.1+} we have
\[
 \|S-\sigma\|_{B\otimes_{\min}\ovl B \otimes_{\min} \Lambda_{p/2}} \le C_2(p/2) \|S(c)\|^{1/2}_{B \otimes_{\min} \ovl B \otimes_{\min} \ovl B \otimes_{\min}  B\otimes_{\min} \Lambda_{p/4}}.
\]
By \eqref{eq1.1}
\[
 \frac12 dc_n \dot\otimes d\bar c_n \prec d_n \dot\otimes \bar d_n \dot\otimes \bar d_n \dot\otimes d_n + {\bb E}_{n-1}(d_n\dot\otimes \bar d_n) \dot\otimes {\bb E}_{n-1}(\bar d_n\dot\otimes d_n)
\]
therefore 
\begin{align*}
 \frac12 S(c) \prec  \delta + v,
\end{align*}
where we now set $v=\sum {\bb E}_{n-1}(d_n\dot\otimes \bar d_n) \dot\otimes {\bb E}_{n-1}(\bar d_n\dot\otimes d_n)$.
Thus we find
\begin{align*}
 \Big|\|S\|_{B\otimes_{\min} \ovl B\otimes_{\min}\Lambda_{p/2}} - \|\sigma\|_{B\otimes_{\min}\ovl B \otimes_{\min} \Lambda_{p/2}}\Big| &\le ||S-\sigma\|_{B\otimes_{\min} \ovl B\otimes_{\min} \Lambda_{p/2}}\\
&\le C_2(p/2) \sqrt 2(\|\delta\|^{\frac12}_\bullet + \|v\|^{\frac12}_\bullet)
\end{align*}
where $\|~~\|_\bullet$ is the norm in $B \otimes_{\min} \ovl B \otimes_{\min}  \ovl B \otimes_{\min} B \otimes_{\min} \Lambda_{p/4}$. By \eqref{eq3.11} we have
\[
 \|v\|_\bullet \le (p/4) \|\delta\|_\bullet
\]
and hence
\[
 \|S\|_{B\otimes \ovl B\otimes \Lambda_{p/2}} \le \|\sigma\|_{B\otimes\ovl B\otimes \Lambda_{p/2}} +  \ C_2(p/2) \sqrt 2 (1+(p/4)^{\frac12}) \|\delta\|_\bullet ^{\frac12}.
\]
Taking the square root of the last inequality we obtain \eqref{eq3.12}.
\end{proof}

\bigskip\bigskip
We now give a version (corresponding to
${BR}_{q}$ with $q=4$) for $\Lambda_p$ of the Burkholder-Rosenthal inequality :

\begin{thm}\label{thm3.1+}
 For any $p\ge 4$ of the form $p=2^k$ for some $k\ge 1$ there are positive constants $ C'_1(p)$, and $C'_2(p)$ such that for any test function $f$ in $B\otimes L_p$ we have
\begin{equation}\label{eq3.1++}
 C'_1(p)^{-1}[f]_p \le \|f\|_{B\otimes_{\min} \Lambda_p} \le C'_2(p)[f]_p
\end{equation}
where
\[
 [f]_p= \|\sigma(f)\|^{1/2}_{B\otimes_{\min} \ovl B\otimes_{\min} \Lambda_{p/2}} + \left\|  {\bb E} ( \sum d_n\dot\otimes \bar d_n\dot\otimes d_n\dot\otimes \bar d_n)^{\dot\otimes p/4}\right\|^{1/p}.
\]
\end{thm}

\begin{proof}
Note that $\ovl S$ and $S$ are the same after a transposition of the two factors, thus the same is true for $S\dot\otimes \ovl S$ and $S^{\dot\otimes 2}$, and we have
\[
 \|{\bb E}(S\dot\otimes \ovl S)\| = \|{\bb E}(S^{\dot\otimes 2})\|
\]
and similarly for any even  $m\ge 1$
\[
 {\bb E}((S\dot\otimes \ovl S)^{\dot\otimes m/2})\|= \|{\bb E}(S^{\dot\otimes m})\|.
\]
Recall that in a suitable permutation we may write $0 \prec d_n \dot\otimes \bar d_n \dot\otimes \bar d_k \dot\otimes d_k$ for all $n,k$ and hence
\[
 \sum d_n \dot\otimes \bar d_n \dot\otimes \bar d_n \dot\otimes d_n \prec  \sum\nolimits_{n,k} d_n \dot\otimes \bar d_n \dot\otimes \bar d_k \dot\otimes d_k = S\dot\otimes \ovl S,
\]
and hence for any even integer $m$
$$     (\sum d_n \dot\otimes \bar d_n \dot\otimes \bar d_n \dot\otimes d_n)^{\dot\otimes m/2} \prec   S^{\dot\otimes m}.$$
Therefore 
\begin{equation}\label{eq3.7}
 \left\|{\bb E}\left(\sum d_n\dot\otimes \bar d_n\dot\otimes d_n\dot\otimes \bar d_n\right)^{\dot\otimes m/2}\right\| \le \|{\bb E}(S^{\dot\otimes m})\|.
\end{equation}
Let $\sigma=\sigma(f)$.
Now if $p=2m$,    \eqref{eq3.4} implies
\begin{equation}\label{eq3.7+}
\|\sigma\|^{1/2}_{B\otimes_{\min} \ovl B\otimes_{\min} \Lambda_{p/2}} = \|{\bb E}(\sigma^{\dot\otimes m})\|^{\frac1 {2m}} \le m^{1/2}
 \|{\bb E}(S^{\dot\otimes m})\|^{\frac1 {2m}},
\end{equation}
and hence by \eqref{eq3.7} and \eqref{eq3.7+}
\begin{align*}
 [f]_p
&\le (m^{1/2} +1) \|{\bb E}(S^{\dot\otimes m})\|^{1/2m}\\
&= (m^{1/2} +1)\|S\|^{1/2}_{B\otimes_{\min}\ovl B\otimes_{\min}\Lambda_m}.
\end{align*}
Thus the left hand side of
\eqref{eq3.1++} follows from \eqref{eq3.1+}.
Since the converse inequality follows from Lemma \ref{lem3.7} and \eqref{eq3.1+}, this completes the proof.
\end{proof}

\section{Hilbert transform}\label{sec6}

\indent 

Consider the Hilbert transform on $L_p({\bb T},dm)$. We will show that this defines a completely bounded operator on $\Lambda_p({\bb T},m)$ again for $p\ge 2$ of the form $p=2^k$ with $k\in {\bb N}$. The proof is modeled on Marcel Riesz's proof as presented in Zygmund's classical treatise on trigonometric series. One of the first references using this trick in a broader context is Cotlar's paper \cite{Co}. Let $f$ be a trigonometric polynomial with coefficients in $B(H)$, i.e.\ $f = \sum_{n\in {\bb Z}} \hat f(n)e^{int}$ with $\hat f\colon \ {\bb Z}\to B(H)$ finitely supported. The Hilbert transform $Tf$ is defined by
\begin{equation}\label{eq5.0}
 Tf = \sum\nolimits_{n\in {\bb Z}} \varphi(n) \hat f(n) e^{int}
\end{equation}
where $\varphi(0)=0$ and
\begin{equation}
 \varphi(n) = -i \text{ sign}(n).\tag*{$\forall n\in {\bb Z}$}
\end{equation}
Note that $T^2 = -id$ on the subspace $\{f\mid \hat f(0)=0\}$. We will use the following classical identity valid for any pair $f,g$ of complex valued trigonometric polynomials
\begin{equation}\label{eq5.1}
 T(fg - (Tf)(Tg)) = fTg + (Tf)g.
\end{equation}
This can be checked easily as a property of $\varphi$ since it reduces to the case $f=z^n$, $g=z^m$ $(n,m\in {\bb Z})$. A less pedestrian approach is to recall that if $f$ is \emph{real valued}, $Tf$ is characterized as the unique real valued $v$, the ``conjugate function'', actually here also a trigonometric polynomial, such that $\hat v(0) = 0$ and $z\mapsto f(z) + iv(z)$ is the boundary value of an analytic function (actually a polynomial in $z$) inside the unit disc $D$. Then \eqref{eq5.1} boils down to the observation that since $(f+iTf)(g+iTg)$ is the product of two analytic functions on $D$, $f(Tg)+(Tf)g $ must be the ``conjugate'' of $fg - (Tf)(Tg)$.
The complex case follows from the real one:
 for a complex valued $f$, we define $Tf= T(\Re(f)) +i T(\Im (f))$ and \eqref{eq5.1} remains valid.
 From \eqref{eq5.1} in the ${\bb C}$-valued case, it is immediate to deduce that for any pair $f,g$ of $B(H)$-valued trigonometric polynomials we have
\begin{equation}\label{eq5.2}
T(f\dot\otimes g - (Tf) \dot\otimes (Tg)) = f\dot\otimes (Tg) + (Tf)\dot\otimes g
\end{equation}
where (as before) the notation $f\dot\otimes g$ stands  for the $B(H)\otimes B(H)$ valued function $z\to f(z) \otimes g(z)$ on ${\bb T}$, and where we  still denote by $T$
the mapping (that should be denoted by $T\otimes I$) taking 
$f\otimes b$ $(f\in L_2, b\in B(H))$ to $(Tf )\otimes b$. Now, it is a simple exercise to check that for any such $f$
\[
\ovl{Tf} = T(\bar f)
\]
(this is an equality between two $\ovl{B(H)}$ valued functions). Therefore we have also:
\begin{equation}\label{eq5.3}
 T(f\otimes\bar g - (Tf) \dot\otimes T(\bar g)) = f\otimes T(\bar g) + Tf\otimes\bar g.
\end{equation}

We can now apply the well known Riesz--Cotlar trick to our situation:

\begin{thm}\label{thm5.1}
 For any $p\ge 2$, of the form $p=2^k$ with $k\in {\bb N}$, the Hilbert transform $T$ is a c.b.\ mapping on $\Lambda_p({\bb T},m)$.
\end{thm}

\begin{proof}
If we restrict (as we may) to functions such that $\hat f(0) =0$, we have $T^2=-id$ and hence \eqref{eq5.3} implies
\begin{equation}\label{eq5.4}
 Tf\dot\otimes T\bar f - f\dot\otimes \bar f = T(f\dot\otimes (\ovl{Tf}) + (Tf) \dot\otimes \bar f).
\end{equation}
We can then again use induction on $k$. Assume the result  known for $p$, i.e.\ that there is a constant $C$ such that
\[
 \|Tf\|_{B(H) \otimes_{\min} \Lambda_p} \le C\|f\|_{B(H)\otimes_{\min} \Lambda_p}.
\]
We will prove that the same holds for $2p$ in place of $p$ (with a different constant). Let $B=B(H)$.
By \eqref{eq5.4}, we have
\begin{align*}
 \|Tf\dot\otimes \ovl{Tf} - f\otimes\bar f\|_{B\otimes_{\min}\ovl B \otimes_{\min} \Lambda_p} &\le 2C \|f\otimes \ovl{Tf}\|_{B\otimes_{\min}\ovl B \otimes\Lambda_p}.\\
\intertext{By \eqref{eq2.3}, this term is}
&\le 2C\|f\otimes\bar f\|^{1/2}_{B\otimes_{\min} \ovl B \otimes_{\min} \Lambda_p} \|Tf \otimes \ovl{Tf}\|^{1/2}_{B\otimes_{\min} \ovl B \otimes \Lambda_p}.
\end{align*}
We have
\[ \|f\|^2_{B\otimes_{\min} \Lambda_{2p}} = \|f\dot\otimes \ovl{f}\|_{B\otimes_{\min} \ovl B \otimes_{\min} \Lambda_p} \ {\rm and} \ 
 \|Tf\|^2_{B\otimes_{\min} \Lambda_{2p}} = \|Tf\dot\otimes \ovl{Tf}\|_{B\otimes_{\min} \ovl B \otimes_{\min} \Lambda_p}.
\]
Therefore, denoting this time $x = \|Tf\|_{B\otimes_{\min}\Lambda_{2p}}$ and $y  = \|f\|_{B\otimes_{\min} \Lambda_{2p}}$,  and using 
$$| \  \|Tf\dot\otimes \ovl{Tf} \|_{(p)} - \| f\otimes\bar f\|_{(p)}    |\le  \|Tf\dot\otimes \ovl{Tf} - f\otimes\bar f\|_{(p)} $$
we find again
\[
 |x^2-y^2| \le 2C xy.
\]
Thus we conclude that $x$ and $y$ are ``equivalent,'' completing the proof with $2p$ in place of $p$.
\end{proof}

\section{Comparison with $\pmb{L_p}$}\label{sec7}

Let $B=B(H)$ with (say) $H=\ell_2$. 
Let $E_1,\cdots, E_m$ and $G$ be operator spaces. Recall that an $m$-linear mapping
$$u:\  E_1\times\cdots\times E_m\to G$$
  is called (jointly) completely bounded (j.c.b. in short)
if the associated $m$-linear mapping  from
$$\hat u\colon \  (B\otimes_{\min} E_1)\times\cdots\times (B\otimes_{\min} E_m)\to B\otimes_{\min}\cdots\otimes_{\min}B\otimes_{\min} G$$ is bounded. We set $\|u\|_{cb}=\|\hat u\|$, and we say that $u$
is  (jointly)  completely contractive if $\|u\|_{cb}\le 1$.
Note the obvious stability of these maps under composition: for instance if $F,L$ are operator spaces and
if $v\colon \ G \times F \to L$ is bilinear and  j.c.b. then the $(m+1)$-linear mapping
$w;\ E_1\times\cdots\times E_m\times F \to L$ defined by
$$w(x_1,\cdots,x_m,y)= v(u(x_1,\cdots,x_m),y)$$
is also  j.c.b. with $\|w\|_{cb}\le \|u\|_{cb}\|v\|_{cb}$.
Moreover, if in the above definition we   replace $B$ by the space $K$ of compact operators
on $\ell_2$, the definition and the value of $\|u\|_{cb}$  is unchanged. 
This allows to extend (following \cite{P4}) the complex interpolation theorem for multilinear mappings.
In particular, we have

\begin{lem}\label{lem4.1} Let $1\le p,q,r\le \infty$ be such that $1/p+1/q=1/r$.
Then 
the pointwise product  from $L_p\times L_q$   to $L_r$
is completely contractive. 
More generally,
if $1\le p_j\le \infty$ ($1\le j\le N$) are such that
$\sum 1/p_j=1/r$, the product map
$L_{p_1}\times \cdots\times L_{p_N}\to L_r$ is completely contractive. In particular, if $p$ is any positive integer,
the pointwise product $P_p$ from $L_p\times\cdots\times L_p$ ($p$-times) to $L_1$
is completely contractive.

\end{lem}
\begin{proof} The three cases either $q=\infty, p=r$ or  $p=\infty, q=r$ or  $q=p',r=1$ are obvious.
By interpolation and then exchanging the roles of $p$ and $q$, this implies the general case.
By the preceding remark, one can iterate
and the second assertion becomes clear. 
\end{proof}
\begin{thm}\label{thm4.1}
Let $p=2m$, $m\in {\bb N}$. The identity map $L_p\to \Lambda_p$ is completely contractive.
\end{thm}

\begin{proof}
By the preceding Lemma \ref{lem4.1}, $
 P_{p}\colon \ L_p\times\cdots\times L_p\to L_1
$
is completely contractive.
Therefore
\[
 \left\|\int f_1\dot\otimes\cdots\dot\otimes f_{p/2}\dot\otimes \bar f_1\dot\otimes\cdots \dot\otimes \bar f_{p/2}\right\| \le \left(\prod^{p/2}_1\|f_j\|_{B\otimes_{\min} L_p}\right)^2.
\]
Thus taking $f_1 =\cdots= f_{p/2}$ we get by \eqref{eq2.4+}

\[
 \|f\|^{p}_{B\otimes_{\min}\Lambda_p} \le \|f\|^{p}_{B\otimes_{\min} L_p}
\]
and we obtain $\|L_p\to \Lambda_p\|_{cb}=1$.
\end{proof}

\begin{rem}\label{unic} The preceding argument (together with Corollary \ref{cor2.4-}) shows that the o.s.s. on $\Lambda_p$ is essentially the minimal
one on $L_p$  such that $P_p\colon  \Lambda_p\times\cdots\times\Lambda_p\to L_1 $
is completely contractive. More precisely, assume $p\in \bb N$. Let $Q_p\colon\ L_p\times\cdots\times L_p\times
\bar L_p\times\cdots\times \bar L_p$ (where $L_p$ and  $\bar L_p$ are repeated $p/2$ times) be the $p$-linear mapping taking $(f_1,\cdots,f_{p/2},\bar g_1,\cdots,\bar g_{p/2})$ to
$\int f_1 \cdots f_{p/2}\bar g_1,\cdots,\bar g_{p/2} d\mu$.
Then
 if $X_p$ is  an o.s. isometric to $L_p$,
such that $Q_p\colon  X_p\times\cdots\times X_p \times
\bar X_p\times\cdots\times \bar X_p\to L_1 $ is completely contractive,
  the identity map $X_p\to \Lambda_p$ is completely contractive.\\
  In the case of $L_p$ itself with its interpolated o.s.s. we could consider $P_p$ instead of  $Q_p$   
  because the  map  $ f \mapsto \bar f$ is a completely isometric antilinear isomorphism 
  from $L_p$ to itself, and hence defines a completely isometric  {\it linear} isomorphism 
  from $L_p$ to $\bar L_p$. (This can be checked easily by interpolation
  starting from $p=\infty$ and-by duality-$p=1$.)
\end{rem}

This remark leads to:
\begin{cor} For any integer $p\ge 1$, we have a completely contractive inclusion
$$(\Lambda_p,L_\infty)_{1/2}\to \Lambda_{2p}.$$
\end{cor}
\begin{proof} Let $X_{2p}=(\Lambda_p,L_\infty)_{1/2}$. We will argue    as in the proof of Lemma \ref{lem4.1},
applying complex interpolation to
  the product map $P_{2p}:\ (x_1,\cdots,x_{2p})\mapsto x_1\cdots x_{2p}$. Clearly,
  by Lemma \ref{lem4.1},  $P_{2p}$
  is completely contractive both as a map
  from $(\Lambda_p)^p\times (\bar L_\infty)^p $  to $L_1 $ and as one 
  from $(L_\infty)^p\times (\bar \Lambda_p)^p $
  to $L_1 $. Therefore, by interpolation,
 $P_{2p}\colon  X_{2p}\times\cdots\times X_{2p} \times
\bar X_{2p}\times\cdots\times \bar X_{2p}\to L_1 $ is completely contractive. The preceding remark 
(applied with $2p$ in place of $p$) then yields this corollary.\end{proof}

\begin{rem} We wish to compare here the operator spaces $L_p$ and $\Lambda_p$. We already know that they are different
since the Khintchine inequalities lead to two different operator spaces in both cases, but we can give a more precise 
quantitative estimate.

Let us denote by $L^n_p$ the space $L_p(\Omega_n,\mu_n)$ when $\Omega_n = [1,\ldots, n]$ and $\mu_n$ is the uniform probability measure on $\Omega$. We then set
\[
 \Lambda^n_p = \Lambda_p(\Omega_n,\mu_n).
\]
 We claim that for any even integer $p>2$, there is $\delta_p>0$ such that for any $n$ the identity map (denoted id) satisfies
\[
 \|id\colon \ \Lambda^n_p \to L^n_p\|_{cb} \ge \delta_p n^{\frac1p(\frac12-\frac1p)}.
\]
To prove this, we will use an adaptation (with $\Lambda_p$ instead of $L_p$) of the results in \cite{Har,P5}. Indeed, by 
  Corollary \ref{cor11.2} below, using the classical ``Rudin examples'' of $\Lambda(p)$-sets, one can show that the space $\Lambda^n_p$ contains a subspace $E_n\subset \Lambda^n_p$ with $\dim E_n = d(n) \ge n^{2/p}$ and  such that the inclusion $OH_{d(n)}\subset E_n$ satisfies  $\|OH_{d(n)}\to E_n\|_{cb}\le \chi_p$. 
  Moreover, there is a projection $P_n\colon \ \Lambda^n_p\to E_n$ with $\|P_n\|_{cb}\le \chi_p$.  Here $p$ is an even integer $>2$ and $\chi_p$ is a constant depending only on $p$.
In addition, by \cite{Har}, the same space $E_n$ considered in $L^n_p$ is (uniformly over $n$) completely isomorphic to $R_p(d(n)) \cap C_p(d(n))$ (intersection of row and column space in $S^{d(n)}_p$). In fact we use only the easy direction of this result, namely
that $\|E_n\to R_p(d(n)) \|_{cb}\le 1$ and $\|E_n\to C_p(d(n))\|_{cb}\le 1$.

\n It follows that there is a constant $\delta_p>0$ such that if $id$ denotes the identity map we have
\[
 \|id\colon\ \Lambda^n_p\to L^n_p\|_{cb}\ge \delta_p\|OH_{d(n)}\to C_p(d(n))\|_{cb}.
\]
Recall that  (see \cite[p. 219]{P6}) $\|OH(d)\to C_\infty(d)\|_{cb} = d^{1/4}$.
Thus by interpolation, we have for any $d$ if $\frac1p =\frac\theta2$
\[
 \|C_p(d) \to C_\infty(d)\|_{cb} \le \|OH(d)\to C_\infty(d)\|^\theta_{cb} = d^{\theta/4},
\]
also
\[
 \|OH(d)\to C_p(d)\|_{cb} \|C_p(d)\to C_\infty(d)\|_{cb} \ge \|OH(d) \to C_\infty(d)\|_{cb},
\]
therefore we find
\[
 \|OH(d)\to C_p(d)\|_{cb} \ge d^{1/4}d^{-\theta/4}
\]
and we conclude for some $\delta'_p >0$
\[
 \|id \colon\ \Lambda^n_p\to L^n_p\|_{cb} \ge \delta'_p(n^{2/p})^{\frac{1-\theta}4} = \delta'_p n^{\frac{\theta(1-\theta)}4}.
\]
A similar argument applies to compare $\Lambda^n_p$ with either $\min(L^n_p)$ or $\max(L^n_p)$. Using the projections $P_n$, we easily deduce that for some constant $\chi'_p>0$
\[
 \|\Lambda^n_p \to \max(L^n_p)\|_{cb}\ge \chi'_p  \|OH_{d(n)}\to \max(\ell^{d(n)}_2)\|_{cb}
\]
and
\[
 \|\min(L^n_p)\to \Lambda^n_p\|_{cb} \ge \chi'_p\|\min(\ell^{d(n)}_2) \to OH_{d(n)}\|_{cb}.
\]
But it is known (see \cite[p. 220]{P6}) that for any $d$
\[
 \|OH_d\to \max(\ell^d_2)\|_{cb} = \|\min(\ell^d_2)\to OH_d\|_{cb} \simeq cd^{1/2}
\]
where $c>0$ is independent of $d$. Thus we obtain
\[
 \|\Lambda^n_p\to \max(L^n_p)\|_{cb}\ge c \chi'_p d(n)^{1/2} \simeq c'n^{1/p}
\]
and similarly 
\[
 \|\min(L^n_p)\to \Lambda^n_p\|_{cb} \ge c'n^{1/p}.
\]
\end{rem}

\section{The non-commutative case}\label{sec8}

Let $\cl M$ be a von Neumann algebra equipped with a normal semi-finite faithful trace $\tau$, and let $L_p(\tau)$ be the associated ``non-commutative'' $L_p$-space. The preceding procedure works equally well in the non-commutative case, but requires a little more care. 
To define the o.s.s.\ on $L_p(\tau)$ that will be of interest to us we consider $f$ in $B(H)\otimes L_p(\tau)$ of the form $f = \sum^n_1 b_k\otimes x_k$ and we define $f^*\in \ovl{B(H)} \otimes L_p(\tau)$ by
\[
 f^* = \sum\nolimits^n_1 \bar b_k \otimes x^*_k.
\]
Consider $f= \sum^n_1 b_k\otimes x_k\in B(H)\otimes L_p(\tau)$ as above and $g=\sum c_j \otimes y_j\in B(K)\otimes L_q(\tau)$ $(p,q\ge 1)$. We denote 
by $f\dot\otimes g\in B(H)\otimes B(K)\otimes  L_r(\tau)$ $\big(r\ge 1, \frac1r = \frac1p +\frac1q\big)$ the element defined by
\[
 f\dot\otimes g = \sum\nolimits_{k,j} b_k\otimes c_j \otimes x_k y_j.
\]  
Given $f\in B(H) \otimes L_1(\tau)$ we denote $\hat\tau = id_{B(H)} \otimes\tau\colon \ B(H) \otimes L_1(\tau) \to B(H)$. More explicitly if $f$ is as above (here $p=1$) we set
\[
 \hat\tau(f) = \sum b_k \tau(x_k),
\]
and since the norm and the cb-norm coincide for linear forms, we have
\[
\| \hat\tau(f) \|\le \|f\|_{B\otimes_{\min} L_1(\tau)} .
\]
Then, by the trace property, if $r=1$, $\hat\tau (f\dot\otimes g)$  and $\hat\tau (g\dot\otimes f)$ are the same up
to transposition of the two factors, and hence have the same minimal norm. More generally,
given finite sequences  $f_\ell \in B(H)\otimes L_p(\tau)$ as above and $g_\ell\in B(K)\otimes L_q(\tau)$ ($r\ge 1,  1/r =  1/p + 1/q$), the same reasoning yields
\begin{equation}\label{eq8.1}\|\hat\tau(\sum\nolimits_\ell  f_\ell\dot\otimes g_\ell  )\|=\|\hat\tau( \sum\nolimits_\ell  g_\ell\dot\otimes f_\ell )  \|.
\end{equation}
This identity \eqref{eq8.1} will considerably facilitate the generalization of most of the preceding proofs
to the non-commutative case, in a rather easier fashion than for the corresponding steps in \cite{PX1}.
 
Now, Haagerup's version of the Cauchy--Schwarz inequality for the Hilbert space $\ell_2(L_2(\tau))$ becomes:
\begin{lem}\label{lem8.0}
 Let $f_k,g_k\in B\otimes L_2(\tau)$ $(k=1,\ldots, N)$. Then
\begin{equation}\label{eq8.0}
 \left\|\sum\nolimits^N_1 \widehat\tau(f^*_k\dot{\otimes} g_k)\right\|\le \left\|\sum \widehat\tau(f^*_k \dot{\otimes} f_k)\right\|^{1/2} \left\|\sum \widehat\tau \left(\sum g^*_k\dot{\otimes}g_k\right)\right\|^{1/2},
\end{equation}
and actually this is valid when $\tau$ is any (not necessarily tracial) state on $M$.
\end{lem}

\begin{proof}
Let ${\cl H}$ be the Hilbert space obtained (after quotient and completion) from $M$ equipped with the scalar product $\langle x,y\rangle = \tau(y^*x)$. Then this lemma appears as a particular case of \eqref{eq2.02}.
\end{proof}
We will use repeatedly  the identification
\[
 \ovl{B(H)} = B(\ovl H).
\]
 The operator space $\Lambda_p(\tau)$ will be defined as isometric to $L_p(\tau)$ but with an o.s.s.\ such that for any $f$ in $B(H)\otimes L_p(\tau)$ ($p$ an even integer) we have
\begin{equation}\label{eq7.1}
\|f\|_{B(H)\otimes_{\min} \Lambda_p(\tau)} = \|\hat\tau(f^*\dot\otimes f \otimes\cdots \dot\otimes f^*\dot\otimes f)\|^{\frac1p}_{B(\ovl H \otimes_2 H \otimes_2 \cdots \otimes_2 \ovl H \otimes_2 H)}
\end{equation}
where $f^*\dot\otimes f$ and $\ovl H\otimes_2 H$ are repeated $p/2$-times.

To prove that \eqref{eq7.1} really defines a norm (and an o.s.s.) on $B(H)\otimes L_p(\tau)$ we proceed exactly as in the commutative case by first establishing a H\"older type inequality:

\begin{lem}\label{lem7.1}
Let $p\ge 2$ be an even integer. Consider $f_j\in B(H_j) \otimes L_p(\tau)$. Let
\[
 \|f_j\|_{(p)} = \|\hat\tau(f^*_j\dot\otimes f_j\dot \otimes\cdots \dot\otimes f^*_j\dot \otimes f_j)\|^{1/p}_{B(\ovl H_j \otimes H_j \otimes\cdots\otimes \ovl H_j \otimes H_j)}
\]
where $f^*_j\dot\otimes f_j$ is repeated $p/2$ times. We have then
\begin{equation}\label{eq7.2}
\|\hat\tau(f_1 \dot\otimes \cdots\dot\otimes f_p)\| \le \prod^p_{j=1} \|f_j\|_{(p)}.
\end{equation}
\end{lem}

\begin{proof}
We will use repeatedly the fact that the minimal tensor product is commutative i.e.\ a permutation $\sigma$ of the factors induces a complete isometry (and actually a $*$-isomorphism) from $B(H_1 \otimes_2\cdots \otimes H_n)$ to
\[
 B(H_{\sigma(1)} \otimes_2\cdots \otimes_2 H_{\sigma(n)}).
\]
Thus for any $x = \sum b^1_j \otimes\cdots \otimes b^n_j \otimes x_j\in B(H_1) \otimes\cdots \otimes B(H_n) \otimes L_p(\tau)$, if we denote $\sigma[x] = \sum b^{\sigma(1)}_j \otimes \cdots \otimes b^{\sigma(n)}_j \otimes x_j$ we have 
\[
 \|x\|_{\min} = \|\sigma[x]\|_{\min}.
\]
Let $y = \sigma[x]$. To indicate that one can pass from $x$ to $y$ by a permutation, it will be convenient to write $x\approx y$.

Thus $x\approx y$ guarantees $\|x\|_{\min} = \|y\|_{\min}$. For example, let
\[
 f_1\in B(H_1) \otimes L_p(\tau) \qquad f_2\in B(H_2) \otimes L_p(\tau).
\]
Then $\hat\tau((f_1\dot\otimes f_2)^* )\approx\hat\tau(f^*_2\dot\otimes f^*_1)$ and hence
\[
 \|\hat\tau((f_1\dot\otimes f_2)^*)\| = \|\hat\tau(f^*_2 \dot\otimes f^*_1)\|.
\]
Also using the trace property we have for any $f$ in $B(H)\otimes L_2(\tau)$ 
\[
 \hat\tau(f^*\dot\otimes f) \approx \hat\tau(f\dot\otimes f^*)
\]
and hence
\begin{equation}\label{eq7.3}
 \|f\|_{(2)} = \|f^*\|_{(2)}.
\end{equation}
More generally, for any $f_1,\ldots, f_p$ as before we have by \eqref{eq8.1}
\begin{equation}\label{eq7.4}
 \|\hat\tau(f_1 \dot\otimes\cdots \dot\otimes f_p) \|_{B(H_1 \otimes_2\cdots \otimes_2 H_p)}=\|\hat\tau(f_p \dot\otimes f_1 \dot\otimes \cdots\dot \otimes f_{p-1})\|_{B(H_p \otimes_2\cdots \otimes_2 H_{p-1})}.
\end{equation}
 In particular this gives us for any $j$
\begin{equation}\label{eq7.5}
 \|f_j\|_{(p)} = \|\hat\tau(f_j\dot\otimes f^*_j \dot\otimes\cdots\dot\otimes f_j \dot\otimes f^*_j)\|^{\frac1p}
\end{equation}
or equivalently
\begin{equation}\label{eq7.6}
 \|f_j\|_{(p)} = \|f^*_j||_{(p)}.
\end{equation}

To prove the Lemma, we start with $p=2$. In that case \eqref{eq7.2} reduces to  \eqref{lem8.0}. 
 Let us denote by \eqref{eq7.2}$_p$ the inequality \eqref{eq7.2} for a given value of $p$.
We will show
\[
 \eqref{eq7.2}_p \Rightarrow \eqref{eq7.2}_{2p}.
\]
This covers only the case $p=2^k$, but actually the argument used earlier for $L_p(\mu)$ (see Lemma  \ref{lem2.1}) when $p$ is an even integer  can be easily adapted to the  case of $L_p(\tau)$
(note that the  invariance of $I(f_1,\cdots,f_p)=\hat \tau (f_1\dot\otimes\cdots \dot\otimes f_p)$   under cyclic permutations suffices to adapt this argument here). We leave the details to the reader at this point.

So assume \eqref{eq7.2}$_p$ proved for some integer $p\ge 2$. Consider
\[
 f_j \in B(H_j) \otimes L_{2p}(\tau)\qquad j=1,\ldots, 2p.
\]
Let
$$
 g_j = f_{2j-1} \dot\otimes f_{2j} \in B(H_{2j-1} \otimes_2 H_{2j})\otimes L_p(\tau)\qquad (j=1,\ldots, p)  .$$
 By  \eqref{eq7.2}$_p$ we have
\begin{equation}\label{eq7.7}
\|\hat\tau (g_1\dot\otimes \cdots\dot\otimes g_p)\| \le \prod^p_1 \|g_j\|_{(p)}.
\end{equation}
Moreover using \eqref{eq7.4} we find
\[
 \|g_j\|_{(p)} = \|\hat\tau(f^*_{2j-1} \dot\otimes f_{2j-1} \dot\otimes f_{2j}  \dot\otimes f^*_{2j}\dot\otimes\cdots)\|^{\frac1p}
\]
where the preceding expression is repeated $p/2$ times.\\
By \eqref{eq7.2}$_p$ we have
\begin{align*}
 \|g_j\|_{(p)} &\le \|f^*_{2j-1} \dot\otimes f_{2j-1}\|^{\frac12}_{(p)} \|f_{2j} \dot\otimes f^*_{2j}\|^{\frac12}_{(p)}\\
\intertext{and hence by \eqref{eq7.5}}
&\le \|f_{2j-1}\|_{(2p)} \|f_{2j}\|_{(2p)}.
\end{align*}
Thus we find that \eqref{eq7.7} implies \eqref{eq7.2}$_{2p}$.
\end{proof}

We then have just like in the commutative case:

\begin{thm}\label{thm7.2}
Let $p \ge 2$ be an even integer. The space $L_p(\tau)$ can be equipped with an o.s.s.\ so that denoting by $\Lambda_p(\tau)$ the resulting operator space we have for any $H$ and any $f$ in $B(H)\otimes L_p(\tau)$
\[
 \|f\|_{B(H)\otimes_{\min} \Lambda_p(\tau)} = \|f\|_{(p)}.
\]
\end{thm}

\begin{proof}
We may reduce consideration to $H=\ell_2$ for simplicity of notation. We have then by \eqref{eq7.2}
\begin{equation}\label{eq7.8}
 \|f\|_{(p)} = \sup\|\hat \tau(f\dot\otimes f_2 \dot\otimes\cdots\dot\otimes f_p)\|
\end{equation}
where the supremum runs over all $f_j$ in $B(H)\otimes L_p(\tau)$ ($2\le j\le p$) with $\|f_j\|_{(p)} \le 1$. We then define for any $x$ in $L_p(\tau)$ 
\begin{equation}\label{eq7.9}
 J(x) = \oplus [\hat \tau( x \dot\otimes f_2 \dot\otimes\cdots\dot\otimes f_p)]
\end{equation}
where the direct sum runs over all choices of $(f_j)$ $(j\ge 2)$ as before. Then \eqref{eq7.8} ensures that
\[
 \|f\|_{(p)} = \|(id_{B(H)} \otimes J) (f)\|_{\min}.
\]
Thus $J$ defines an isometric embedding of $L_p(\tau)$ into some $B({\cl H})$ (here ${\cl H}$ is a suitably ``huge'' direct sum) as in \eqref{eq7.9} so that the associated o.s.s.\ satisfies the desired property. 
\end{proof}
By exactly the same argument as for Corollary \ref{cor2.4} above, we have
\begin{cor}\label{cor2.4nc}
Let $p\ge q\ge 2$ be even integers. If $\tau(1)=1$, the inclusion
$\Lambda_p(\tau) \subset \Lambda_q(\tau) $ is a complete contraction from
$\Lambda_p(\tau)$ to $\Lambda_q(\tau)$.
\end{cor}
It is important for the sequel to observe that $0\prec \widehat \tau(f^*\dot{\otimes} f)$ for any $f$ in $B\otimes L_2(\tau)$. This follows from a very general fact on sesquilinear forms.

\begin{lem}\label{lem8.3}
Let $B$ and $E$ be complex vector spaces. Let $x\in (B\otimes E) \otimes (\ovl {B\otimes  E})$ be such that $ x\succ 0$, meaning by this that $x$ can be written as a finite sum $x = \sum t_k\otimes \bar t_k$ with $t_k\in B\otimes E$. 
We will use the natural identification $\ovl {B\otimes  E}=\ovl {B}\otimes  \ovl{E}$. Let $\varphi\colon \ E\otimes\ovl E \to {\bb C}$ be a bilinear form (equivalently $\varphi$ defines a sesquilinear form on $E\times E$). Let $y = (\varphi \otimes \text{\rm id}_{B\otimes\ovl B})(x) \in B\otimes \ovl B$ (more precisely
here $\varphi$ acts on the second and fourth  factors, so,  to indicate this,  the notation $y=(\varphi)_{24}(x)$ would be 
less abusive). If $\varphi$ is positive definite (meaning that $\varphi(a\otimes\bar a) \ge 0$ $\forall a\in E$), then $y\succ 0$.
\end{lem}

\begin{proof}
Note that we may as well assume $B$ and $E$ finite dimensional. Consider then $t = \sum b_k\otimes a_k\in B\otimes E$, $\xi\in B^*$ and $s= (\xi\otimes \text{id}_E)(t)\in E$. We have $(\xi\otimes\bar\xi\otimes \text{id}_{E\otimes\ovl E})(t\otimes\bar t) = s\otimes \bar s \succ 0$ and hence $(\xi\otimes\bar\xi)(y) = \varphi(s\otimes\bar s) \ge 0$. By the proof of Lemma~\ref{lem1.0} we conclude that $y\succ 0$.
\end{proof}

In particular, since $\tau(a^*a) = \tau(aa^*)\ge 0$ for any $a$ in $L_2(\tau)$, this implies:

\begin{lem}\label{lem8.4}
For any $f$ in $B\otimes L_2(\tau)$, we have
\[
\widehat \tau(f^*\dot{\otimes} f)\succ 0\quad (\text{and } \widehat  \tau(f\dot{\otimes} f^*)\succ 0).
\]
\end{lem}

\begin{rem}
By the classical property of conditional expectations, if $1\le p,p'\le \infty$ are conjugate (i.e.\ $p' = p(p-1)^{-1}$) and if $T\colon \ L_p(\tau)\to L_p(\tau)$ is the conditional expectation with respect to a (von~Neumann) subalgebra of $M$, then: $\forall x\in L_p(\tau)$ $\forall y\in L_{p'}(\tau)$ we have 
\[
 \tau(T(x)y) = \tau(xT(y)) = \tau(T(x)T(y)).
\]
Therefore for any $f\in B(H_1)\otimes L_p(\tau)$ and $g\in B(H_2)\otimes L_{p'}(\tau)$ we have:
\begin{equation}\label{eq10.1}
 \widehat\tau(T(f)\dot{\otimes} g)  = \widehat\tau(f\dot{\otimes} T(g)) = \widehat\tau(T(f)\dot{\otimes} T(g))
\end{equation}
where we still denote abusively by $T$ the operator $I\otimes T$ acting either on $B(H_1)\otimes L_p(\tau)$ or on $B(H_2) \otimes L_{p'}(\tau)$. Moreover, it is easy to check that $T(f^*)=T(f)^*$
for any $f\in B\otimes L_p(\tau)$.
\end{rem}

In the rest of this section we continue to abusively denote  by $T$  the operator $I\otimes T$ on $B\otimes L_p(\tau)$.

\begin{lem}\label{lem10.2}
 Let $T\colon \ L_p(\tau)\to L_p(\tau)$ be the conditional expectation with respect to a von~Neumann subalgebra   ${\cl N}\subset {\cl M}$. Let $p=2m$ be an even integer. Then for any $f$ in $B\otimes L_p(\tau)$ we have
\begin{equation}\label{eq10.2}
 \|Tf\|_{(p)} \le \|f\|_{(p)}.
\end{equation}
\end{lem}

\begin{proof}
By \eqref{eq10.1}, we have 
\[
 \widehat\tau(T(f)\dot{\otimes} T(f)^* \dot{\otimes} \cdots \dot{\otimes} T(f) \dot{\otimes} T(f)^*) = \widehat\tau(f\dot{\otimes} T(f)^* \dot{\otimes} \cdots \dot{\otimes} T(f) \dot{\otimes} T(f)^*).
\]
Indeed, just observe that if $g=T(f)^* \dot{\otimes} \cdots \dot{\otimes} T(f) \dot{\otimes} T(f)^*$ then $T(g)=g$. Therefore by \eqref{eq7.2} we have
\[
 \|T(f)\|^p_{(p)} \le \|f\|_{(p)} \|T(f)\|^{p-1}_{(p)}
\]
and hence after a suitable division we obtain \eqref{eq10.2}.
\end{proof}

\begin{rem}\label{rem8.9} In the preceding situation
for any $f$ in $B\otimes L_2(\tau)$, let $f_0=T(f)$ and $d_1=f-T(f)$. We have then
$T(f^*\dot\otimes f)=  f_0^*\dot\otimes f_0 + T(d_1^*\dot\otimes d_1)$,
and hence (since $\hat \tau T =\hat \tau  $)
$\hat \tau(f^*\dot\otimes f)=  \hat \tau(f_0^*\dot\otimes f_0) + \hat \tau(d_1^*\dot\otimes d_1)$
Therefore, by Lemma \ref{lem8.4}, we have both $\hat \tau(f_0^*\dot\otimes f_0) \prec
 \hat \tau(f^*\dot\otimes f) $ and $\hat \tau(d_1^*\dot\otimes d_1) \prec
 \hat \tau(f^*\dot\otimes f) $.

\end{rem}

By Corollary \ref{cor2.4nc}, assuming $\tau(1)=1$,   for all even integers $p\ge q\ge 2$,
and any $f\in B\otimes \cl M$, we have $\|f\|_{(p)}\le  \|f\|_{(q)}$, so that it is again natural to define
$$\|f\|_{(\infty)}=\lim_{p\to \infty} \|f\|_{(p)}.$$
This norm is clearly associated to a well defined o.s.s. on $\cl M$, so 
we are led to the following
\begin{defn} Assume $\tau(1)=1$. We will denote 
by $\Lambda_\infty( \cl M,\tau)$ the
Banach space $ \cl M$ equipped with the o.s.s. determined by
the identities  
$$  \forall f\in B\otimes \cl M\qquad  \|f\|_{B\otimes_{\min} \Lambda_\infty( \cl M,\tau)} 
 =\|f\|_{(\infty)}=\sup_{p\in 2\N}\|f\|_{(p)}.$$
\end{defn}
We warn the reader that in sharp contrast with the commutative case, in general
$\Lambda_\infty( \cl M,\tau)$  is not completely isometric to $ \cl M$. 
See \S \ref{sec9bis} below for more on this, including the case study of  $ \cl M=M_n$ equipped with
its normalized trace.

\section{Comparisons}\label{compa}

We need to recall the definition of the  ``opposite''  of an operator space
  $E\subset B(H)$. The ``opposite''  of $E$, denoted by $E^{op}$,  is the 
  same Banach space as $E$, but equipped with the following norms
on $M_n(E)$. For any $(a_{ij})$ in $M_n(E)$ we define
$$\|(a_{ij})\|_{M_n(E^{op})} {\buildrel {\rm def}\over =}
\|(a_{ji})\|_{M_n(E)}.$$
Equivalently, $E^{op}$ can be defined as the
operator space structure on $E$ for which the
transposition:\ $x\to {}^tx \in B(H^*)$ defines a
completely isometric embedding of $E^{op}$ into $B(H^*)$.

Let $\cl M$ be a von Neumann algebra equipped with a normal semi-finite faithful trace $\tau$, and let 
$L_p(\tau)$ be the associated ``non-commutative'' $L_p$-space.
We need to recall the definition of the ``natural" o.s.s. on $L_p(\tau)$ in the sense of
\cite{P4} (we follow the clarification in \cite[p. 139]{P6} that is particularly important at this point).
We set $L_\infty(\tau)=\cl M$. Of course we view   $L_\infty(\tau)$ as an operator space completely isometric to $\cl M$.
The space $L_1(\tau)$ is classically defined as the completion of $\{x\in \cl M\mid \tau(|x|)<\infty\}$
for the norm $x\mapsto \|x\|_1=\tau(|x|)$. It can be identified isometrically with ${\cl M}_*$ 
via the mapping $x\mapsto \varphi_x$ defined by $\varphi_x(a)=\tau(xa)$. The space ${\cl M}_*\subset {\cl M}^{*}$ 
is equipped with the o.s.s. induced by the dual of the von Neumann algebra ${\cl M}$ (this duality uses Ruan's theorem, see e.g. \cite{P6,ER}). The ``natural" o.s.s. on $L_1(\tau)$ is defined as the one
transferred from the space ${\cl M}_*^{op}$ via the preceding isometric identification $x\mapsto \varphi_x$.
In short we declare that $L_1(\tau)={\cl M}_*^{op}$ completely isometrically. Then
using complex interpolation, we define the ``natural" o.s.s. on $L_p(\tau)$ ($1<p<\infty$)
by the completely isometric identity $L_p(\tau)=(L_\infty(\tau),L_1(\tau))_{1/p}$.\\
For example, when $(\cl M,\tau)=(B(\ell_2),tr)$, the space  $L_p(\tau)$ can be identified with
the Schatten $p$-class. The column (resp. row) matrices in $B(\ell_2)$ form an operator space 
usually denoted by $C$ (resp. $R$). However, when considered as a subspace
of $L_1(\tau)$ they are completely isometric to $R$ (resp. $C$), while
when considered as  subspaces
of $L_2(\tau)$ they both are completely isometric to $OH$.
The non-commutative case of \S  \ref{sec7} requires us to introduce yet another o.s.s. on
$L_p(\tau)$. \\
We set   
again    ${\cl L}_\infty(\tau)={\cl M}$ but we set 
${\cl L}_1(\tau)=  {\cl M}_*  $ (so that ${\cl L}_1(\tau)={ L}_1(\tau)^{op}$)
 and we
 denote by ${\cl L}_p(\tau)$ the operator space defined by
$$ {\cl L}_p(\tau)=({\cl L}_\infty(\tau),{\cl L}_1(\tau))_{1/p}.$$
The space  ${\cl L}_p(\tau)$ is isometric to  ${ L}_p(\tau)$ but in the non-commutative case
its o.s.s. is different. For instance
if $(\cl M,\tau)=(B(\ell_2),tr)$, 
the column (resp. row) matrices in ${\cl L}_p(\tau)$ form an operator space 
that is  completely isometric to $C$ (resp. $R$), for \emph{all} $1\le p\le \infty$.
In sharp contrast, the o.s.s. of the subspace formed  of the diagonal matrices is the same
in ${\cl L}_p(\tau)$ or ${L}_p(\tau)$, and it can be identified completely isometrically
with ${\ell}_p$ equipped with its natural o.s.s. . In particular, ${\cl L}_2(\tau)$ is isometric to the Hilbert-Schmidt class $S_2$,
but the column (resp. row) matrices in ${\cl L}_2(\tau)$ are completely isometric to $C$ (resp. $R$)
while the diagonal ones are completely isometric to $OH$.

\begin{pro}\label{propro} Let  $1\le p,q,r\le \infty$ be such that $r^{-1}=p^{-1}+q^{-1}$. The product mapping
$$(x,y)\mapsto xy$$
is (jointly) completely contractive from
${\cl L}_p(\tau)\times {\cl L}_q(\tau)$ to ${\cl L}_r(\tau)$.
\end{pro}
\begin{proof} We start by the two cases $p=r=1,q=\infty$ and $q=r=1,p=\infty$. We need to show that
$(x,\varphi_y)\mapsto \varphi_{xy}$  (resp. $(\varphi_x,y)\mapsto  \varphi_{xy}$)
  are 
(jointly) completely contractive  from ${\cl M}\times {\cl M}_*$  to ${\cl M}_*$ (resp. from ${\cl M}_*\times {\cl M}$  to ${\cl M}_*$).
Consider $x=[x_{kl}]\in M_m({\cl M})$ with $\|x\|\le 1$ and $y=[y_{ij}]\in M_n({\cl M}_*)$ with $\|y\|\le 1$ 
(resp. $x=[x_{kl}]\in M_m({\cl M}_*)$ with $\|x\|\le 1$ and $y=[y_{ij}]\in M_n({\cl M})$ with $\|y\|\le 1$).
It suffices  to show that in both cases we have  $\|[ \varphi_{x_{kl} y_{ij}}]\|_{M_{mn}({\cl M}_*)}\le 1$. Equivalently, we need to show
that the map
$u:\ \cl M\to M_n\otimes M_m$ defined by
$u(a)=\sum e_{ij}\otimes e_{kl} \ \tau(y_{ij} a x_{kl})  $
satisfies $\|u\|_{cb}\le 1 $.   \\
Consider $v:\ {\cl M}\to M_m\otimes {\cl M}$ defined by
$v(a) = \sum e_{kl}\otimes  ax_{kl}=(I\otimes a) x$. Clearly $\|v\|_{cb}\le \|x\|\le 1.$
Let $w:\ {\cl M}\to M_n$ defined by
$w (a)=\sum e_{ij} \tau(y_{ij} a)$. Then $\|w \|_{cb}= \|y\|_{M_n(  {\cl M}_{*} )}\le 1.$
We have $(I\otimes w ) v (a)=\sum e_{kl}\otimes e_{ij} \ \tau(y_{ij} a x_{kl}) $
which is $u(a)$ up to permutation of the tensor product. Therefore
$\|u\|_{cb}\le \|I\otimes w \|_{cb} \|v \|_{cb}\le \|w \|_{cb} \|v \|_{cb}\le 1$\\
(resp. let  $V :\ {\cl M}\to M_n\otimes {\cl M}$ and $W :\ {\cl M}\to M_m$
be defined  
by $V (a) = \sum e_{ij}\otimes  ay_{ij} $
and $W (a)=\sum e_{kl} \tau(x_{kl} a)$, then
we have $u(a)=(I\otimes W) V(a)$ and we conclude similarly).
\end{proof}
\begin{cor}\label{corcor} For any $p\in 2\N$ and any  $f\in B\otimes L_p(\tau)$ we have
$$\|f\|_{(p)}= \|f\|_{B\otimes  { {\Lambda}_p(\tau)  } }\le \max\{  \|f\|_{B\otimes_{\min}  { {\cl L}_p(\tau)  } }, \|f^*\|_{\bar B \otimes_{\min}  { {\cl L}_p(\tau)  }}\}.$$
In other words, the identity defines a completely contractive map 
 ${\cl L}_p(\tau)\cap  { {\cl L}_p(\tau)}^{op}\to {\Lambda}_p(\tau)  $ (where
 ${\cl L}_p(\tau)\cap  { {\cl L}_p(\tau)}^{op}$ denotes the o.s.s. on ${L}_p(\tau)$
 induced by the embedding $x\mapsto x \oplus x \in {\cl L}_p(\tau)\oplus  { {\cl L}_p(\tau)}^{op}$.
\end{cor}
\begin{proof} By iteration,
the preceding statement implies
that for any integer $N$ the product mapping  is  (jointly) completely contractive from 
${\cl L}_{p_1}(\tau)\times\cdots\times {\cl L}_{p_N}(\tau)$ to ${\cl L}_{r}(\tau)$  when $1/r=\sum 1/p_j$.
Equivalently, setting $B_1=\cdots=B_N=B$, the mapping
$(f_1,\cdots,f_N)\mapsto   f_1 \dot\otimes\cdots \dot\otimes f_N$
is contractive from $B_1\otimes_{\min} {\cl L}_{p_1}(\tau)\times \cdots \times B_N\otimes_{\min} {\cl L}_{p_N}(\tau)$ to $B_1\otimes_{\min} \cdots\otimes_{\min} B_N\otimes_{\min} {\cl L}_{r}(\tau)$.
A fortiori, when $r=1$,
$(f_1,\cdots,f_N)\mapsto  \hat \tau (f_1 \dot\otimes\cdots \dot\otimes f_N)$
is contractive from $B_1\otimes_{\min} {\cl L}_{p_1}(\tau)\times \cdots \times B_N\otimes_{\min} {\cl L}_{p_N}(\tau)$ to $B_1\otimes_{\min} \cdots\otimes_{\min} B_N$.
Therefore, if $p$ is an even integer,  we have
$$\|f\|^p_{(p)}= \|\hat \tau ( f^*\dot\otimes f\dot\otimes \cdots\dot\otimes   f^*\dot\otimes f)\|  \le      \|f\|_{B\otimes_{\min}  { {\cl L}_p(\tau)  } }^{p/2}   \|f^*\|_{\bar B \otimes_{\min}  { {\cl L}_p(\tau)  }}    ^{p/2}.$$
A fortiori we obtain the announced result. 
Note that $\bar x\mapsto x^*$ is a completely isometric \emph{linear} isomorphism both from  
$\bar {\cl M}$ to ${\cl M}^{op}$ and  
 from  
$\bar {\cl M}_*$  to ${\cl M}_*^{op}$, and hence also  
from $\overline { {\cl L}_p(\tau)
}$ to ${ {\cl L}_p(\tau)
}^{op}$ for all $1\le p\le \infty$.  Therefore, if $f=\sum b_j \otimes x_j$
we have
$\|f^*\|_{\bar B \otimes_{\min}  { {\cl L}_p(\tau)  }}=\|\sum \bar b_j \otimes x^*_j\|_{\bar B \otimes_{\min}  { {\cl L}_p(\tau)  }}
=\|\sum \bar b_j \otimes \bar x_j\|_{\bar B \otimes_{\min}  \overline{ {\cl L}_p(\tau)^{op}  }}=
\|\sum b_j \otimes   x_j\|_{ B \otimes_{\min}   { {\cl L}_p(\tau)^{op}  }}$. 
Thus $\|f^*\|_{\bar B \otimes_{\min}  { {\cl L}_p(\tau)  }}=\|f\|_{B \otimes_{\min}  { {\cl L}_p(\tau)^{op}   }}$,
whence the last assertion. \end{proof}

We will now examine
the particular case when $(\cl M,\tau)=(B(\ell_2),tr)$. Recall that  
 $R$ (resp.    $C$) is the subspace of $\cl M=B(\ell_2)$ formed by all row (resp. column) matrices.
More generally,    we denote by ${ R}_p$  (resp.    ${ C}_p$)    the operator space obtained by equipping
 $R$ (resp.    $C$) with the o.s.s. induced by 
 $ L_p(\tau) $.   We also denote by ${ R}^n_p$  (resp.    ${ C}^n_p$) the $n$-dimensional
 version of ${ R}_p$  (resp.    ${ C}_p$).\\
 Similarly, we will denote by $\widetilde{ R}_p$  (resp.    $\widetilde{ C}_p$) 
 the operator space obtained by equipping
 $R$ (resp.    $C$) with the o.s.s. induced 
 by  $ \Lambda_p(\tau) $.   \\
 Furthermore, let  $\widetilde{ D}_p$ be the operator subspace of $ \Lambda_p(\tau) $ formed of all the diagonal matrices.
As a Banach space this is isometric to $\ell_p$, and it is easy to check
that as an operator space $\widetilde{ D}_p$ is completely isometric to the space
$\lambda_p=\Lambda_p(\N,\mu)$ with $\mu$ equal to   the counting measure on $\N$.\\
Let $b_j\in B$ ($j=1,\cdots,n$) and let $f=\sum b_j \otimes e_{1j}\in B\otimes R$
(resp. $g=\sum b_i \otimes e_{i1} \in B\otimes C $).
Then $f\dot\otimes f^*=\sum b_j \otimes \bar b_j \otimes e_{11}$
(resp. $g^*\dot\otimes g=\sum \bar b_j \otimes   b_j \otimes e_{11}$). Note that $\|\sum \bar b_j \otimes   b_j\|^{1/2}=\|\sum b_j \otimes \bar b_j\|^{1/2}.$
Therefore, viewing $f$ and $g$ as  elements of $B\otimes   \Lambda_p(\tau)$,
for any $p\in 2\N$,  we have  
$$\|f\|_{(p)}=\|g\|_{(p)}=\|\sum b_j \otimes \bar b_j\|^{1/2}.$$
Thus we find:

\begin{lem} The spaces $\widetilde{ R}_p$ and $\widetilde{ C}_p$ are both completely isometric to $OH$
for any $p\in 2 \N$, while $\widetilde{ D}_p$ is completely isometric to $\lambda_p$.
\end{lem}

Again let $b_j\in B$ ($j=1,\cdots,n$) and let $f=\sum b_j \otimes e_{1j}$. 
We have $\|f\|_{ B\otimes L_p(\tau)} = \sup\{ \| \sum b_j a b_j^* \|_p^{1/2}\mid \|a\|_p\le 1\}$
(see \cite[p. 83-84]{P3} or \cite{X2} for details). In case $b_j=e_{j1}$, this gives us  $\|f\|_{ B\otimes L_p(\tau)} =n^{1/2p}$.
Therefore the natural inclusion $R^n_p \to R^n_2$ has c.b. norm $\ge n^{1/4-1/2p}$.
Similarly, using instead $b_j=e_{1j}$, we find $\|f\|_{ B\otimes L_p(\tau)} =n^{(1/2)(1-1/p)}$ and hence 
$\|R^n_2 \to R^n_p\|_{cb}\ge n^{1/4-1/2p}$.
This shows:
\begin{lem} 
For any $p\in 2 \N$ and any integer $n\ge 1$, 
the $n$-dimensional identity maps satisfy
$$ \| L_p(M_n,tr) \to \Lambda_p(M_n,tr)\|_{cb} \ge n^{1/4-1/2p}\ {\rm and}\  
 \| \Lambda_p(M_n,tr) \to L_p(M_n,tr)\|_{cb} \ge n^{1/4-1/2p} 
.$$
\end{lem}

\section{Connection with CB maps on $OH$}\label{sec9bis}

Given a Hilbert space $H$ we denote by $OH$ the operator Hilbert space
isometric to $H$, as defined in \cite{P3}. This means that whenever $(T_j)$ is an orthonormal basis of $OH$, for any finitely supported  family $(b_j)$ in $B$
we have
\begin{equation}\label{oh}\|\sum b_j \otimes T_j\|=\| \sum\bar b_j \otimes b_j\|^{1/2}. 
\end{equation}
Assume  $\cl M\subset B(H)$ and $\tau(1)=1$. We will compare the limit o.s.s. of  $\Lambda_p({\cl M} ,\tau)$  when $p\to \infty$  to the one induced on $\cl M$ by
$CB(OH)$ equipped with its usual operator space structure.\\
The latter can be described as follows (see e.g. \cite{ER}): Whenever $E,F$ are operator spaces
the space $CB(E,F)$ of all c.b. maps from $E$ to $F$ is equipped with
the (unique) o.s.s. determined by the isometric identity
$$\forall N\ge 1\qquad M_N(CB(E,F))=CB(E,M_N(F)).$$
More generally, we have an isometric embedding
\begin{equation}\label{cb+}
B \otimes_{\min} CB(E,F)\subset CB(E,B \otimes_{\min} F).
\end{equation}
If either $E$ or $F$ is finite dimensional, we may identify completely isometrically $CB(E,F)$ with $E^*\otimes_{\min} F$. When $E=F$, we denote simply
 $CB(E)=CB(E,E)$. Thus in particular  $CB(OH_n)$ can be identified with
 $OH_n^*\otimes_{\min} OH_n$, or equivalently by the selfduality
 of $OH_n$, with $\overline{OH_n}\otimes_{\min} OH_n$ or
 ${OH_n}\otimes_{\min} \overline{OH_n}$. We first recall a well known fact.

  \begin{lem}\label{mn0} Let $E,F,G$ be   operator spaces.
    Let $B'=B(H')$ for some Hilbert space $H'$. Then 
  for any $f=\sum b_j\otimes x_j\in B\otimes CB(F,G)$ and  $g=\sum b'_k \otimes y_k\in B'\otimes CB(E,F)$ we have
 \begin{equation}\label{g5}\|f \dot\otimes g\|_{ B\otimes_{\min} B' \otimes_{\min} CB(E,G) }\le  \|f  \|_{   B \otimes_{\min} CB(F,G) }\ \|  g\|_{ B'   \otimes_{\min} CB(E,F) },\end{equation}
 where, as before, we denote $f \dot\otimes g=\sum_{j,k} b_j\otimes b'_k \otimes x_j   y_k\in B\otimes B'\otimes CB(E,G)$.\\ In other words,
 the composition $(x,y)\mapsto xy$ is (jointly) completely contractive
 from $CB(F,G)\times CB(E,F)$ to $CB(E,G)$.
   \end{lem}
  \begin{proof} 
To prove \eqref{g5}, note that $f$ (resp. $g$) defines a c.b. map 
${\tilde f}:\ F \to  G\otimes_{\min} B$ (resp. ${\tilde g}:\ E \to F\otimes_{\min} B'$) and 
$\|f\|_{ B\otimes_{\min} CB(F,G)}=\|{\tilde f}\|_{cb}$ (resp. $\|g\|_{ B'\otimes_{\min} CB(E,F)}=\|{\tilde g}\|_{cb}$). Indeed,  recall that, if we wish,  $G\otimes_{\min} B$ can be identified with $B\otimes_{\min} G$. Similarly,
$f\dot \otimes g$ defines a c.b. map $\Psi: E \to  G \otimes_{\min} B\otimes_{\min} B'$
such that $\|f\dot \otimes g\|_{ B\otimes_{\min} B'\otimes_{\min} CB(E,G)}=\|\Psi\|_{cb}$.
But since
  $\Psi = (  {\tilde f} \otimes Id_{B'} ) \circ {\tilde g}$ we have
$\|\Psi\|_{cb}\le \|{\tilde f}\|_{cb}\|{\tilde g}\|_{cb}$ and \eqref{g5} follows.
    \end{proof}
     \begin{rem}\label{rcb} In particular, the preceding Lemma implies a fortiori that
     that if $D,E,F,G$ are operator spaces
     and if $u\in CB(D,E)$ and $v\in CB(F,G)$ are fixed complete contractions,
     then the mapping $x\mapsto vxu$ is a complete contraction
     from $CB(E,F)$ to $CB(D,G)$. Indeed, the latter can be viewed as the restriction
     of the triple product map to $\C v \times CB(E,F)\times \C u$.
      \end{rem}
 \begin{thm}\label{mn}
Let  $(\cl M, \tau)$ be as before with $\cl M \subset B(H)$ and   $\tau(1)=1$.
 Let us denote by $\underline{ \cl M}$ the operator space
 obtained by equipping $\cl M$ with the o.s.s. induced by  $CB(OH)$.
 Then  $$\Lambda_\infty(\cl M,\tau)=\underline{ \cl M} $$   completely isometrically.
 \end{thm}
 The proof of this Theorem will require some observations
 about the space  $CB(OH)$ that may be of independent  interest.

 The following rather striking identity \eqref{g} appears as analogous to Gelfand's axiom 
  (namely $\|x\|^2=\|x^*x\|$) for $C^*$-algebras. It seems to express that 
  $CB(OH)$ is an o.s. analogue of a $C^*$-algebra...
  \begin{thm}   Let us denote simply by $\cl B$ the operator
  space $CB(OH)$. (Note that $\cl B$ is isometric to $B(H)$ as a Banach space.)
  For any $f \in B\otimes \cl B$ we have
\begin{equation}\label{g}\|f\|_{  B\otimes_{\min} \cl B }^2=
 \|  f^*\dot\otimes f \|_{  \bar B\otimes_{\min} B\otimes_{\min} \cl B}= \|  f\dot\otimes f^* \|_{ B\otimes_{\min} \bar B\otimes_{\min} \cl B}.\end{equation}
Moreover, we also have
 \begin{equation}\label{g2}\|f ^*\|_{  \bar B \otimes_{\min} \cl B }  = \|f \|_{  B\otimes_{\min} \cl B}.\end{equation}
  \end{thm}
 \begin{proof} 
 Let $H_i\subset H$ be an increasing net of finite dimensional subspaces
 with dense union.
 Assuming $f=\sum b_j\otimes x_j$,
 let $f(i)=\sum b_j\otimes P_{H_i} {x_j}_{|H_i}\in B\otimes CB(OH_i)$.
 Then, using the homogeneity of $OH$
 in the sense of \cite[p. 19]{P3} or \cite{P4}, one checks that each side of either \eqref{g} or  \eqref{g2} is equal to the supremum
 over $i$ of the expression obtained after substituting $f_i$ for $f$.
 Thus it suffices to prove \eqref{g} or  \eqref{g2} when $\dim(H)<\infty$.\\
 In that case, denoting by $T_j$ an orthonormal basis of
 $OH_n$, and using the identity $CB(OH_n)=OH_n \otimes_{\min} \overline{OH_n}$, we may write  any $f\in B\otimes \cl B$  
 as $f=\sum b_{ij}\otimes T_i \otimes  \bar T_j$ with $b_{ij}\in B$,
 and $\|f\|_{  B\otimes \cl B }=\|\sum b_{ij}\otimes T_i \otimes \bar  T_j\|_{B\otimes_{\min}  {OH_n}\otimes_{\min} \overline{OH_n}}$.
 Using $\|x\|=\|\bar x\|$ for any operator $x$,
 and permuting the second and third factors, we have then obviously (the norm being the min-norm)
 $$\|\sum b_{ij}\otimes T_i \otimes   \bar T_j\|=\|\sum \bar b_{ij}\otimes  \bar  T_i \otimes T_j\|=\|\sum \bar b_{ij}\otimes T_j  \otimes  \bar  T_i \|=\|\sum \bar b_{ji}\otimes T_i \otimes\bar  T_j\|,
 $$
 and this is clearly equivalent to \eqref{g2}.\\
 Let $y_i=\sum_j b_{ij}\otimes \bar T_j$. By \eqref{oh},
  we have
$ \|\sum b_{ij}\otimes  T_i \otimes \bar T_j\|_{B\otimes_{\min} {OH_n}\otimes_{\min} \overline{OH_n}}=\| \sum y_i \otimes T_i\|=\| \sum \bar y_i  \otimes   y_i\|^{1/2}=\| \sum  y_{jk} \otimes  T_j \otimes \bar T_k\| ^{1/2} $
where $y_{jk}=\sum_i \bar b_{ij} \otimes   b_{ik}$. Using again
the identity $CB(OH_n)={OH_n}\otimes_{\min}  \overline{OH_n}$ 
we find $  f^* \dot\otimes f =  \sum  y_{jk} \otimes T_j \otimes \bar  T_k$.
Thus, we have $\|f\|=\|f^*\dot\otimes f\|^{1/2}$, and 
  by \eqref{g2} we obtain \eqref{g}.
\end{proof}
\begin{rem} Note that after iteration,  for any $p=2^m$ ($m\ge 1$),  \eqref{g} yields
 \begin{equation}\label{gbis}
 \|f\|_{  B\otimes_{\min} \cl B }^p=
 \|  f^*\dot\otimes f \dot\otimes f^*\dot\otimes f\cdots \|_{ \bar B\otimes_{\min}   B\otimes_{\min} \cdots \otimes_{\min} \cl B}.
 \end{equation}
 \end{rem}
\begin{cor}\label{cor00} Let $H_I=\oplus_{i\in I} H_i$ be an orthogonal decomposition of  a Hilbert space $H_I$.
We have then  a completely isometric embedding
$$ \oplus_{i\in I} CB(OH_i)\subset CB(OH_I).$$
\end{cor}
 \begin{proof}
 Let $u:\  \oplus_{i\in I} CB(OH_i)\to CB(OH_I)$ denote this embedding.
   It is easy to reduce the proof to the finite case
  so we assume $|I|<\infty$. Since the coordinatewise inclusions and  projections  relative to $O  H_I$ are all
completely contractive, it is easy to check
using Lemma \ref{mn0}  that $\|u\|_{cb}\le |I|<\infty$.
  Consider now $f\in B\otimes_{\min}  (\oplus_{i\in I} CB(OH_i))$ and let $g=(Id_B \otimes u)(f)\in B\otimes_{\min}  CB(OH_I)$. 
We need to show that $\|g\|= \|f\|$. 
Note that   $(Id_{\bar B\otimes B\otimes\cdots} \otimes u)(f^*\dot\otimes f)^{\dot\otimes m})=(g^*\dot\otimes g)^{\dot\otimes m}$. Thus,
by  \eqref{g5} and \eqref{g} 
we have for any integer $m$ 
$$\|g\| = \|(g^*\dot\otimes g)^{\dot\otimes m}\| ^{1/2m}  \le (|I|\|(f^*\dot\otimes f)^{\dot\otimes m}\| )^{1/2m}=|I|^{1/2m} \|f\|$$
so that letting $m\to \infty$ we obtain $\|g\|\le \|f\|$.
Since the converse inequality follows easily  from Remark \ref{rcb} 
applied to the coordinate projections, we have equality.
 \end{proof}
 \begin{thm}\label{mn00} Let $E\subset B(H)$ be any operator space. Let us denote again
 by $\underline E$ the operator space obtained by inducing on $E$ the o.s.s. of $CB(OH)$.
 Let $F\subset B(K)$ be another  operator space.
 Then for any $u\in CB(E,F)$ we have
 $$\|u\|_{CB(\underline E,\underline F)}\le \|u\|_{CB(  E,  F)}.$$
 In particular, if $u:\ E\to F$ is completely isometric,
 then $u:\ \underline  E\to \underline  F$ also is.
 \end{thm}
\begin{proof} We may clearly assume $F=B(K)$ and $\underline F=CB(OK)$ and by the same argument
as in the preceding proof, we may assume $\dim(K)=n<\infty$. Assume $\|u\|_{CB(  E,  F)}\le  1$. Then
$u$ extends to a  c.b. map $\hat u:\ B(H) \to B(K)$ with the same cb-norm.
Since $M_n(B(H)^*)= M_n(B(H)_*)^{**}$ isometrically (see e.g. \cite[p. 75]{ER}) $\hat u$ is then a point norm limit
of normal maps with cb-norm $\le 1$, so we may assume that $u$ is normal on $E=B(H)$.
Then (see \cite[p. 45]{SS}) there is a factorization of $u$
of the form $u(x)=V \rho(x) W$ with $\|V\|\le 1, \|W\|\le 1$ where $\rho$ is an ``ampliation",
i.e.   $\rho$ takes its values in $B(\oplus_{i\in I} H_i)$ for some set $I$ with $H_i=H$ for all $i\in I$
and $\rho(x)=\oplus_{i\in I} \rho_i(x)$ with $\rho_i(x)=x$ for all $i\in I$. This reduces the Lemma
to the case when $u$ is an ampliation 
and to the case when $u$ is of the form  $u(x)=V x W$. 

Let us first assume  $u(x)=V x W$ with $V:\ H\to K$ and $W:\ K\to H$ of norm 1.
By the homogeneity of $OH$ we know that the cb norm of $V:\ OH\to OK$ is 1, and similarly
for $W:\ OK\to OH$. Then by Remark \ref{rcb}  $\|u\|_{CB(CB(OH),CB(OK))}\le 1$.   \\
We now assume that $u$ is an ampliation i.e. 
$u=u_I$ where $u_I(x)=\oplus_{i\in I} u_i(x)\in B(\oplus_{i\in I} H_i)$
with $H_i=H$ and  $u_i(x)=x$ for all $i\in I$. Let $ H_I= \oplus_{i\in I} H_i$. 
By Corollary \ref{cor00} $u_I$ is a complete isometry from $CB(OH)$ to $CB(OH_I)$.
Since both multiplications and ampliations  have been checked,
the proof of the first assertion is complete. The second assertion is then
 immediate.
 \end{proof}
 \begin{cor}\label{mn01} Let $E\subset B(H)$ be any operator space.
 The o.s.s.  of $\underline E$   (induced on $E$ by that of $CB(OH)$)
 is independent of the completely isometric embebdding $E\subset B(H)$,
 i.e. it
 depends only on the o.s.s. of $E$.
 \end{cor}
 
\begin{proof}[Proof of Theorem \ref{mn}]
We first give a simple argument for the special case 
when $\cl M=M_n$ equipped with its normalized trace $\tau_n$. We will show that $\Lambda_\infty(M_n,\tau_n)$ can be identified
 completely isometrically with  $CB(OH_n)$ 
 for any $n\ge 1$. 
We first claim that the identity from $\overline{OH_n} \otimes OH_n$ to itself
induces a mapping $V_n:\ \overline{OH_n} \otimes_{h} OH_n
\to \overline{OH_n} \otimes_{\min}OH_n$
such that $\|V_n\|_{cb} \le 1$ and $\|V^{-1}_n\|_{cb} \le n^{1/2}$.
By the minimality of the minimal tensor product the first assertion is obvious.
To check the second one,
recall the  identity map on $n$-dimensional Hilbert space
defines an isomorphism $u_n:\ R_n\to OH_n$ such that
$\|u_n\|_{cb}=\|u^{-1}_n\|_{cb}=n^{1/4}$ (see e.g. \cite[p. 219]{P6}).
Therefore, we have a factorization of $V^{-1}_n$ as follows
$$\overline{OH_n} \otimes_{\min}OH_n {\buildrel {Id\otimes u^{-1}_n}\over \to} 
\overline{OH_n} \otimes_{\min}R _n=\overline{OH_n} \otimes_{h}R _n
{\buildrel {Id\otimes u_n}\over \to}  \overline{OH_n} \otimes_{h} OH_n,$$
where we used the identity $E\otimes_{\min} R_n=E\otimes_{h} R_n$
for which we refer e.g. to \cite[p. 95]{P6}.
From this follows  $\|V^{-1}_n\|_{cb} \le \|u^{-1}_n\|_{cb}\|u_n\|_{cb}   =n^{1/2}$.\\
More explicitly, recall that for any pair of Hilbert spaces $H,K$ we have (see \cite[Cor. 2.12]{P3})
$$OH \otimes_h OK= O(H\otimes_2 K);$$
in particular,
$\Lambda_2(M_n,tr)$ is the same as $L_2(M_n,tr)=\overline{OH_n} \otimes_{h}OH _n$. Therefore, for any $g\in B\otimes \Lambda_2(M_n,tr)$ we have
 \begin{equation}\label{g3}  \|g\|_{B\otimes_{\min} CB(OH_n)}  \le  \|g\|_{B\otimes_{\min}\Lambda_2(M_n,tr)}\le n^{1/2}\|g\|_{B\otimes_{\min} CB(OH_n)}.
\end{equation}
Consider an even integer $p$ and $f\in B\otimes M_n$. Let
$g=f^*\dot\otimes f \dot\otimes  \cdots$ (here there are  $p/2$ factors equal to $f^*\dot\otimes f$).
We have by \eqref{g}
$$\|g\|^2_{\bar B\otimes_{\min}   \dots \otimes_{\min}CB(OH_n)}=\| g\dot\otimes g^* \|_{\bar B\otimes_{\min}   \dots \otimes_{\min}CB(OH_n)}.$$
We have also $g\dot\otimes g^*=f^*\dot\otimes f \dot\otimes \cdots$
(here there are  $p$ such factors) and hence,
assuming $p=2^m$ for some $m$ and using \eqref{gbis},
we find
 \begin{equation}\label{g4} \|g\|^2_{\bar B\otimes_{\min}   \dots \otimes_{\min}CB(OH_n)}=\|f\|^p_{B\otimes_{\min} CB(OH_n)}.\end{equation}
By definition of   $\Lambda_p$  we have 
$$\|f\|^p_{B\otimes_{\min} \Lambda_p(M_n,tr)}=\| g\|^2_{\Lambda_2(M_n,tr)} $$
and hence by \eqref{g3}  
$$   \|g\|^2_{B\otimes_{\min} CB(OH_n)}
  \le \|f\|^p_{B\otimes_{\min} \Lambda_p(M_n,tr)}\le  n\|g\|^2_{B\otimes_{\min} CB(OH_n)}.
$$
Then by \eqref{g4} we obtain
$$   \|f\|^p_{B\otimes_{\min} CB(OH_n)}
  \le \|f\|^p_{B\otimes_{\min} \Lambda_p(M_n,tr)}\le  n \|f\|^p_{B\otimes_{\min} CB(OH_n)},
$$
and, if we  take the $p$-th root and let  $p\to \infty$ this yields
$$ \|f\|_{B\otimes_{\min} CB(OH_n)}
  = \|f\|_{B\otimes_{\min} \Lambda_\infty(M_n,\tau_n)}.$$
We now consider the general case.
 Let $f\in B\otimes \cl M$. For any $p\in 2 \N$ we have $\|f\|_{(p)}^p=\hat \tau(f^*\dot\otimes f\cdots)$.
 As a linear form on $\cl M$, $\tau$ has norm 1, and hence c.b. norm equal to 1 on $\underline{ \cl M} $. Therefore
 $$\|f\|_{(p)}^p=\|\hat \tau(f^*\dot\otimes f\cdots)\| \le \|f^*\dot\otimes f\cdots\|_{\bar B\otimes\cdots B\otimes \underline{ \cl M} }=\|f^*\dot\otimes f\cdots\|_{\bar B\otimes\cdots B\otimes { \cl B} }$$
 and by \eqref{gbis} (assuming $p=2^m$) 
 $$ \|f^*\dot\otimes f\cdots\|_{\bar B\otimes\cdots B\otimes { \cl B} }=\|f\|^p_{B\otimes_{\min} \cl B}=\|f\|^p_{B\otimes_{\min} \underline{ \cl M}}.$$ 
 Thus we obtain $\|f\|_{(p)}\le \|f\|_{B\otimes_{\min} \underline{ \cl M}}$, and taking the supremum 
 over $p$ yields $$\|f\|_{B\otimes_{\min} \Lambda_\infty(\cl M,\tau)}\le \|f\|_{B\otimes_{\min} \underline{ \cl M}}.$$
 It remains to prove the converse inequality.
 
 Consider $f\in B\otimes \cl M$. Let $F:\ OH \to OH\otimes_{\min} B$ be the associated
 c.b. map (as in the proof of Lemma \ref{mn0}).
By Corollary  \ref{mn01} we may assume  that $H=L_2(\tau)$ and that the inclusion $\cl M\subset B(L_2(\tau))$
 is the usual realization of $ \cl M$  acting on $ L_2(\tau)$ by left multiplication.
 \\
 Let $B'$ be another copy of $B$. 
 Note that for any $\xi \in B'\otimes  {OH}$ we have
 $$\|\xi\|_{B'\otimes_{\min}  {OH}}=\|\xi\|_{(2)}.$$
 Moreover, if $\xi\in B'\otimes   {\cl M}\subset B'\otimes  {OH}$,
 then up to permutation of factors $(Id_{B'}\otimes F)(\xi) \approx f\dot\otimes \xi$.
Since $\| F\|_{cb}=\|Id_{B'}\otimes F\|_{cb}$,
   the definition of the o.s.s of $CB(OH)$ (see \eqref{cb+} above) shows that
 $$\|f\|_{B\otimes_{\min}  \underline{\cl M}  }=\| F\|_{cb}=
 \sup\{    \| f\dot\otimes \xi\|_{(2) }   \mid \xi \in B'\otimes  {\cl M},\     \|\xi\|_{(2)}\le 1\}.$$
 Fix $\xi\in B'\otimes   {\cl M}$ with $ \|\xi\|_{(2)}\le 1  $. To complete the proof
 it suffices to show
 that     
 $$ \| f\dot\otimes \xi\|_{(2) }  \le \sup_{p\in 2\N}\|f\|_{(p)}=\|f\|_{B\otimes_{\min} \Lambda_\infty(\cl M,\tau)}.$$
 To verify this, we claim that for any $p$ of the form $p=2^m$ we have
 \begin{equation}\label{g6}  \| f\dot\otimes \xi\|_{(2) }  \le \| \hat\tau ( (f^*\dot\otimes f)^{\dot\otimes p/2 } \dot\otimes \xi  \dot\otimes\xi^*)\|^{1/p}.\end{equation}
  This is easy to check by induction on $m$. Indeed, by  \eqref{eq7.1} (for $p=2$), equality holds
  in the case $m=1$  and if we assume our claim proved for a given value of $m$
  then the Haagerup-Cauchy-Schwarz inequality \eqref{eq8.0}
  shows that it holds also for $m+1$, because we may write (recall \eqref{eq8.1})
   $$\| \hat\tau ( (f^*\dot\otimes f)^{\dot\otimes p/2 } \dot\otimes \xi  \dot\otimes\xi^*)\|
   \le   \|  (f^*\dot\otimes f)^{\dot\otimes p/2 } \dot\otimes \xi \|_{(2) } \|\xi^*\|_{(2) } =
      \| \hat\tau ( (f^*\dot\otimes f)^{\dot\otimes p } \dot\otimes \xi  \dot\otimes\xi^*)\|^{1/2} \|\xi^*\|_{(2) },
     $$
     and by \eqref{eq7.3} $\|\xi^*\|_{(2) }\le 1$, so we obtain \eqref{g6} with  $2p$ in place of $p$. \\
  We now use the claim to conclude: By  \eqref{eq8.0} again (or by \eqref{eq7.2}) we have
 $$ \| \hat\tau ( (f^*\dot\otimes f)^{\dot\otimes p/2 } \dot\otimes \xi  \dot\otimes\xi^*)\|^{1/p}\le
 \|f\|_{(2p)} \|\xi  \dot\otimes\xi^*\|^{1/p}_{(2)} .$$
 Now $\xi\in B'\otimes   {\cl M}$ implies $\xi\dot\otimes\xi^*\in B'\otimes   \bar{B'}\otimes  {\cl M}$ and,  since 
 $\tau$ is finite, we have $\|\xi  \dot\otimes\xi^*\|_{(2)} <\infty$, therefore
 $\|\xi  \dot\otimes\xi^*\|^{1/p}_{(2)} \to 1$ when $p\to \infty$ and
 we deduce from \eqref{g6}
 $$  \| f\dot\otimes \xi\|_{(2) } \le  \limsup_{p\to \infty} \| \hat\tau ( (f^*\dot\otimes f)^{\dot\otimes p/2 } \dot\otimes \xi  \dot\otimes\xi^*)\|^{1/p}\le
 \limsup_{p\to \infty}   \|f\|_{(2p)}=  \sup_{p\in 2\N}\|f\|_{(p)},$$
 which completes the proof.
 \end{proof}
 
 \begin{rem}\label{rere} By the minimality of the min tensor product, we know
 that we have a completely contractive inclusion $OH^*\otimes_h OH \to OH^*\otimes_{\min} OH\subset CB(OH)$. Therefore, for any pair of sets $I,J$, in analogy with the inclusion of the Hilbert-Schmidt class into the bounded operators,  we have  a completely contractive inclusion
 $$OH(I\times J)\to CB(OH(I),OH(J)).$$
  \end{rem}
\section{Non-commutative Khintchine inequalities}\label{sec9}

We start by a fairly simple statement mimicking a classical commutative fact:

\begin{pro}\label{probuch} Let $p=2n$. Let  $\{x_k\}$ be   a sequence  in $L_p(\tau)$,
such that, for some constant $C$,
for any finite sum $f=\sum b_k \otimes x_k$ with coefficients $b_k$ in $B(H)$, we have
\begin{equation}\label{eq-buch}
\|f\|_{(p)}\le C  \left\|\sum b_k\otimes \bar b_k\right\|^{1/2}.
\end{equation}
Assume moreover that $\{x_k\}$ is orthonormal in $L_2(\tau)$ and $\tau(1)=1$.
Then    the closed span of $(x_k)$ in $\Lambda_p(\tau)$
is completely isomorphic to $OH$ and  completely complemented  in $\Lambda_p(\tau)$.
More precisely the orthogonal projection $P$ onto
this span satisfies $\|P\colon\ \Lambda_p \to \Lambda_p\|_{cb}\le C$.
\end{pro}
\begin{proof} Let $P$ be the orthogonal projection 
on $\Lambda_2$ onto the span under consideration.
For any $f\in B\otimes \Lambda_p$,
let $h=(Id\otimes P)(f)$. 
By a well known fact (see \cite[p. 19]{P3}), $P$ is completely contractive on $\Lambda_2$, so that
$\|h\|_{(2)}\le \|f\|_{(2)}$.
By Corollary \ref{cor2.4nc}, we have $ \|f\|_{(2)}\le \|f\|_{(p)}$
and by our assumption 
$\|h\|_{(p)}\le C\|h\|_{(2)}$. Therefore   $\|h\|_{(p)}\le C\|f\|_{(p)}$.   Thus, the c.b. norm of $P$ acting
from $\Lambda_p$ to itself is automatically $\le C$. Moreover,   for any $h\in B\otimes \overline{\rm span}[x_k]$, we have 
 $ \|h\|_{(2)}\le \|h\|_{(p)}\le C\|h\|_{(2)}$, which shows that the span is 
completely isomorphic to $OH$.
\end{proof}
With the ``natural'' o.s.s.\ introduced in \cite{P4} the Khintchine inequalities for $1<p<\infty$ are due to F.~Lust-Piquard \cite{LP} . For $p$ an even integer, A.~Buchholz \cite{Buch2}  found a beautiful proof that yields optimal constants. His proof is valid for a much more general class of variables instead of the Rademacher functions. We will now follow his ideas to investigate the analogous question in the space $\Lambda_p$.

Let $P_2(2n)$ denote the set of all partitions of $[1,\ldots, 2n]$ onto subsets each with exactly 2 elements. So an element $\nu$ in $P_2$ can be described as a collection of disjoint pairs $\{k_i,j_i\}$ $(1\le i\le n)$ with $k_i\ne j_i$ such that $\{1,\ldots, 2n\} = \{k_1,\ldots, k_n, j_1,\ldots, j_n\}$. 

We call such a partition into pairs a 2-partition.
Let $p=2n$ be an even integer $\ge 2$. Following \cite{Buch} we say that a sequence $\{x_k\}$ in $L_p(\tau)$, has $p$-th moments defined by pairings if there is a function ${\psi}\colon \ P_2(2n)\to {\bb C}$ defined on the set of 2-partitions of $[2n]  = \{1,\ldots, 2n\}$ such that for any $k_1,\ldots, k_{2n}$ we have
\[
 \tau(x_{k_1}x^*_{k_2}x_{k_3}\ldots x_{k_{2n-1}}x^*_{k_{2n}}) = \sum_{\nu\sim (k_1,\ldots, k_{2n})} {\psi}(\nu)
\]
where the notation $\nu\sim (k_1,\ldots, k_{2n})$ means that $k_i=k_j$ whenever the pair $\{i,j\}$ is a block of the partition $\nu$.\\
Note that, for each $k$, taking  the $k_j$'s all equal to $k$, this implies 

\begin{equation}\label{eq9.1b} \tau( |x_k|^p)=\sum\nolimits_{\nu\in P_2(2n) }  {\psi}(\nu).
\end{equation}

Now let $E = \text{span}[x_j]$ and $B=B(H)$. Consider $f\in B\otimes E$ of the form
\[
 f = \sum b_j\otimes x_j.
\]
We have
\[
\widehat \tau((f\dot{\otimes}\, f^*)^{\dot{\otimes}\, n}) = \sum_{k_1,\ldots, k_{2n}} \sum_{\nu\sim (k_1,\ldots, k_{2n})} {\psi} (\nu) b_{k_1} \otimes \bar b_{k_2} \otimes\cdots \otimes b_{k_{2n-1}} \otimes \bar b_{k_{2n}}.
\]
Therefore
\begin{align*}
 \|f\|^{2n}_{(2n)} &= \left\|\sum_{\nu\in P_2(2n)} {\psi}(\nu) \sum_{(k_1,\ldots, k_{2n})\sim\nu} b_{k_1} \otimes \bar b_{k_2} \otimes\cdots\otimes \bar b_{k_{2n}}\right\|\\
&\le \sum_{\nu\in P_2(2n)} |{\psi}(\nu)| \left\|\sum_{(k_1,\ldots, k_{2n})\sim \nu} b_{k_1} \otimes \bar b_{k_2} \otimes\cdots\otimes \bar b_{k_{2n}}\right\|.
\end{align*}
But now let
\[
 \Phi(\nu) = \sum_{(k_1,\ldots, k_{2n})\sim\nu} b_{k_1}\otimes \bar b_{k_2}\otimes\cdots \otimes b_{k_{2n-1}} \otimes \bar b_{k_{2n}}.
\]
Then up to permutation $\Phi(\nu)$ is equal to a product of $n$ terms of the form either $\sum b_k\otimes b_k$, $\sum \bar b_k\otimes\bar b_k$ or $\sum b_k\otimes \bar b_k$. 
Let $T_1,\cdots T_n$ be an enumeration of the latter terms.
Since the permutation leaves the norm invariant, we have $\|\Phi(\nu)\| =\prod_1^n \|T_j\|$.
By \eqref{eq2.1}   $\|T_j\|\le \|\sum b_k\otimes \bar b_k \|$ for each $j$ (actually there is equality  for terms the third kind), and hence
\[
 \|\Phi(\nu)\| \le  \left\|\sum b_k\otimes \bar b_k\right\|^n
\]
and   we conclude that
\begin{equation}\label{eq9.1}
 \|f\|_{(2n)} \le \left(\sum_{\nu\in P_2(2n)} |{\psi}(\nu)|\right)^{1/{2n}} \left\|\sum b_k\otimes \bar b_k\right\|^{1/2}.
\end{equation}
Moreover by \eqref{eq9.1b} we know that if ${\psi}(\nu)\ge 0$ for all $\nu$, then the constant
$\sum_{\nu\in P_2(2n)} |{\psi}(\nu)|$ is optimal.
Recapitulating, we have proved:
\begin{thm}\label{thm9.1} Let $p=2n$. Let  $\{x_k\}$ be as above a sequence  in $L_p(\tau)$, with $p$-th moments defined by pairings via a function $\psi\colon \ P_2(2n)\to {\bb C}$.
Then for any finite sum $f=\sum b_k \otimes x_k$ ($b_k\in B(H)$), we have
\begin{equation}\label{buch}
\|f\|_{(p)}\le C_{\psi,p} \left\|\sum b_k\otimes \bar b_k\right\|^{1/2}, 
\end{equation}
where $C_{\psi,p} = \left(\sum_{\nu\in P_2(2n)} |{\psi}(\nu)|\right)^{1/{2n}}$.
Moreover this constant is optimal if ${\psi}(\nu)\ge 0$ for all $\nu$.
\end{thm}

Buchholz applied the preceding statement  to  a $q$-Gaussian family
with $q\in [-1,1]$. The latter have moments defined by pairings. When $q\in [0,1]$, the 
function $\psi$ is non-negative, so the constant $C_{\psi,p}$ is optimal and,  by \eqref{eq9.1b}, we know $C_{\psi,p}=\|x_1\|_p   $.
In particular, we have:
\begin{cor}\label{cor9.1b} Let $(x_k)$ be a sequence of independent Gaussian normal random variables on
a probability space $(\Omega,\bb P)$. Then the span of $(x_k)$ 
is completely isomorphic to $OH$ and is completely complemented  in $\Lambda_p(\Omega,\bb P)$
for every even integer $p$. Moreover, \eqref{buch} holds with a constant $C_{\psi,p} =\|x_1\|_p   $ that is
  $O(\sqrt{p})$ when $p\to \infty$. \end{cor}
  \begin{rem}\label{rem9.1b}
The preceding Corollary also holds
  when $(x_k)$ is a sequence $(\vp_k)$ of independent symmetric
  $\pm1$ valued variables (or equivalently for  the Rademacher functions).
We   show this in Corollary \ref{cor9.1}  below, but here is a quick proof with a slightly worse constant.
 Let $(x_k)$ be independent Gaussian normal random variables
and assume that $(\vp_k)$ is independent from $(x_k)$. It is well known that
$(x_k)$  has the same distribution as $(\vp_k |x_k|)$. Let 
  $\delta={\bb E}(|x_k|)= 2/\sqrt{\pi}$. The  conditional expectation
 $\cl E$ with respect to  $(\vp_k)$ satisfies ${\cl E}(\vp_k |x_k|)=\delta\vp_k $.
 Therefore $\delta\sum \vp_k b_k={\cl E}(\sum \vp_k|x_k| b_k)$,
 and by \eqref{eq10.2}, this implies
 $$\delta\|\sum \vp_k b_k\|_{(p)}\le \|\sum \vp_k|x_k| b_k\|_{(p)}=\|\sum x_k b_k\|_{(p)}.$$
 So we obtain the Rademacher case with a constant $\le \delta^{-1} C_{\psi,p}$
 since  Proposition \ref{probuch} ensures the complete complementation.
\end{rem}

 The preceding result applies to $q$-Gaussian and in particular
free semi-circular (or circular) elements, see \cite{Buch2} for details.
We have then  $ C_{\psi,p} \le 2/\sqrt{1-|q|}$ for all even $p$. \\
In either the semi-circular ($q=0$) or the circular  case, we have  $C_{\psi,p} \le 2$ for all even $p$, and hence:

\begin{cor} For any even integer $p$, the closed span of a free semi-circular (or circular) family,
 is completely isomorphic to $OH$ and completely complemented
(by the orthogonal projection)
in the space $\Lambda_p$ for the associated trace (on the free group factor). Moreover, the 
corresponding constants are bounded by $2$ uniformly over $p$.
\end{cor}

\begin{cor} \label{cor9.2}  Let $\cl M$ be the von Neumann algebra of the free group $\F_\infty$ with infinitely many generators $(g_k)$.
For any $p=2n$ and any finite sum $\tilde f=\sum b_k \otimes \lambda(g_k) $ ($b_k\in B(H)$), we have
\begin{equation}\label{eqbuch1}\left\|\sum b_k\otimes \bar b_k\right\|^{1/2}\le    \|\tilde   f\|_{(p)}\le 2 \left\|\sum b_k\otimes \bar b_k\right\|^{1/2}.\end{equation}
More generally, let $W_d\subset \F_\infty$ denote the subset formed of the reduced words of length $d$.
Then for any finitely supported   function $b:\ W_d\to B$ we have
\begin{equation}\label{buchd}\|\sum\nolimits_{t\in W_d} b(t) \otimes \lambda(t)\|_{B\otimes \underline{\cl M}}\le (d+1)  \left\|\sum\nolimits_{t\in W_d}  b(t) \otimes \bar b(t) \right\|^{1/2}.\end{equation}

\end{cor}
\begin{proof} The left hand side of \eqref{eqbuch1} follows from Corollary \ref{cor2.4nc} with $q=2$
and the orthonormality of  $(\lambda(g_k))$ in $L_2(\tau)$. 
By \cite[Th. 2.8]{Buch0}  the operator space spanned by $\{ \lambda(t)\}_{t\in W_d} $
is completely isomorphic to the intersection $X$ of a family of $d+1$ operator spaces
$X_i,\  0\le i\le d$, with associated constant equal to $d+1$. On one hand, the space $X_0$ (resp. $X_d$) is completely isometric to
$R$ (resp. $C$), the underlying respective Hilbert space
being $\ell_2(W_d)$. On the other hand, when $0<j<d$ the space  $X_j$ is completely isometric to
the  subspace of 
$B(\ell_2(W_{d-j}), \ell_2(W_{j}))$ associated to matrices of the form $[a(st)]$ 
when $a$ is supported on $W_d$.
Identifying each $W_i$ simply with $\N$ we see that 
$\underline{X_0}$ (resp. $\underline{X_d}$) is completely isometric to $OH( \N) $,
while $\underline{X_i}$ is completely isometric to the associated subspace of  $CB(OH(\N ) )$. By Remark  \ref{rere}
we have a completely isometric inclusion $\underline{X_0} \to \underline{X_i}   $ for any $ 0\le i\le d$,
therefore the intersection of the family $\underline{X_i}\  0\le i\le d$ is completely isometric to $OH$
with $H=\ell_2(W_d)$. Since by Corollary \ref{cor00}  we know that 
 $\underline{X}=\cap_{0\le i\le d}\underline{X_i}$, \eqref{buchd} follows. 
\end{proof}

\begin{rem} A comparison with known results (see \cite{Buch2} for detailed references)
shows that the limit of $ \|\tilde f\|_{(p)}$ when $p\to \infty$ is {\it not} equivalent to  
 $ \|\tilde f\|_{B(H)\otimes_{\min} {\cl M}}$ (here ${\cl M}$ is  the 
von Neumann algebra of the free group with infinitely many generators), in sharp contrast with 
\eqref{eq3.9} above.
\end{rem}

More generally, let $L_p({\cl N}, \varphi)$, or briefly $L_p({\cl N}, \varphi)$, be another non-commutative (semi-finite) $L_p$-space. Consider $f_k\in B\otimes L_p(\varphi)$ and let
\[
 F = \sum f_k\otimes x_k\in B \otimes L_p(\varphi\times \tau) 
\]
where $\{x_k\}$ is as in Theorem~\ref{thm9.1}. We have then:

\begin{thm}\label{thm9.2}
Let $p=2n$ and let $C=C_{\psi,p} $ be the constant appearing in \eqref{buch}. Then for any $F$ as above we have
\begin{equation}\label{eq9.2}
 \|F\|_{(p)} \le C\max\left\{\left\|\sum f_k\dot{\otimes} f^*_k\right\|_{(p/2)}^{1/2}, \left\|\sum f^*_k\dot{\otimes} f_k\right\|_{(p/2)}^{1/2}\right\}.
\end{equation}
\end{thm}

\begin{proof}
 Repeating the steps of the proof of Theorem~\ref{thm9.1}, all we need to do is majorize
\[
 \left\|\widehat\varphi \left(\sum_{(k_1,\ldots, k_{2n})\sim\nu} f_{k_1} \dot{\otimes} f^*_{k_2} \dot{\otimes} \cdots \dot{\otimes} f^*_{k_{2n}}\right)\right\|
\]
by   the right side of \eqref{eq9.2}. This is established in Lemma~\ref{lem9.4} below that is a rather easy adaptation to our $\Lambda_p$-setting of \cite[Lemma 2]{Buch}.
\end{proof}

By the same argument as in  Remark \ref{rem9.1b},  the case of the free generators of the free group
 can be deduced    from the  ``free-Gaussian" one.  Indeed, let $(c_k)$ be a free circular family
(sometimes called ``complex free-Gaussian"). The polar decomposition $c_k=u_k |c_k|$, is such that
the $*$-distribution of $(u_k)$ is identical to that of a free family of Haar unitaries in the sense of \cite{VDN},
or equivalently $(u_k)$ has the same $*$-distribution as that of the free generators $\lambda(g_k)$ in the 
von Neumann algebra of the free group with infinitely many generators. Moreover, a simple calculation
relative to the circular distribution yields $\|c_k\|_1=8/3\pi$. These observations lead us to :

\begin{cor}\label{cor9.3}
With the same notation as in Corollary~\ref{cor9.2}, let $\widetilde F = \sum f_k\otimes \lambda(g_k) $. We have then 
\begin{equation}\label{eq9.4}
(3\pi/4)^{-1}   \|\widetilde F\|_{(p)} \le  \max\left\{\left\|\sum f_k\dot{\otimes} f^*_k\right\|_{(p/2)}^{1/2}, \left\|\sum f^*_k\dot{\otimes} f_k\right\|_{(p/2)}^{1/2}\right\}
 \le  \|\widetilde F\|_{(p)} 
\end{equation}
\end{cor}
\begin{proof} Let $\tilde \tau$ denote the normalized trace on the von Neumann algebra
of the free group with generators $(g_k)$. Let ${\cl E}$ denote the conditional expectation equal to the orthogonal projection
from $L_2(\varphi \otimes \tilde \tau)$ onto $L_2(\varphi) \otimes 1$. Then
${\cl E}(\widetilde F  \dot{\otimes}   {\widetilde F }^*)=\sum f^*_k\dot{\otimes} f_k$.
Since $\|\widetilde F\|_{(p)}^2=\|  \widetilde F  \dot{\otimes}   {\widetilde F }^* \|_{p/2}$
and $ \|  \widetilde F  \dot{\otimes}   {\widetilde F }^* \|_{p/2}\ge \|{\cl E}  (\widetilde F  \dot{\otimes}   {\widetilde F }^*) \|_{p/2}$ by \eqref{eq10.2},
the right hand side follows. To prove the left hand side,
consider   $F=\sum f_k \otimes c_k=\sum f_k \otimes u_k |c_k|$ with $(c_k)$ free circular as above and note
that by the preceding observations (this is similar to Remark \ref{rem9.1b}) we have $   \sum f_k \otimes u_k= (3\pi/8)  (Id \otimes {\cl E}_1 )(\sum f_k \otimes u_k |c_k|)$ where ${\cl E}_1 $ denotes the conditional expectation from the von Neumann
algebra generated by $\{c_k\}$ onto the one generated by $\{u_k\}$.
Since $\{u_k\}$ and $\{\lambda(g_k) \}$ have identical $*$-moments, we find
$$\|\widetilde F\|_{(p)} =\| \sum f_k \otimes u_k\|_{(p)}\le (3\pi/8) \|\sum f_k \otimes u_k |c_k|\|_{(p)}= (3\pi/8)\| \sum f_k \otimes c_k \|_{(p)}$$ and hence the left hand side
of \eqref{eq9.4} follows from \eqref{eq9.2}, recalling that $C\le 2$ when $(x_k)$ is  a free circular sequence.
\end{proof}
\begin{rem} A more careful estimate probably yields the preceding Corollary with the constant
$2$ in place of $3\pi/4$.
\end{rem}
\begin{lem}\label{lem9.4}
 With the preceding notation let
\[
 S(\nu) = \sum_{(k_1,\ldots, k_{2n})\sim\nu} f_{k_1} \dot{\otimes} f^*_{k_2} \dot{\otimes} \cdots \dot{\otimes} f^*_{k_{2n}},
\]
and let $\nu'_0$ (resp. $\nu''_0$) denote the partition  of $[1,\ldots, 2n]$ into consecutive pairs of the form $\{1,2\}$, $\{3,4\},\ldots$ (resp.\ $\{2n,1\}$, $\{2,3\}$, $\{4,5\},\ldots$). We have then
\[
 \|\widehat\varphi(S(\nu))\| \le \max \{\|\widehat\varphi(S(\nu'_0))\|, \|\widehat\varphi(S(\nu''_0))\|\}.
\]
\end{lem}

\begin{proof}[Sketch of Proof]
We set
\[
 {\cl C} = \sup\|\widehat\varphi(S(\nu))\|
\]
where the sup runs over all pair partitions $\nu$ in $P_2(2n)$. By the cyclicity of the trace (see \eqref{eq8.1}) we may assume that $k_1,\ldots, k_{2n}$ is such that for some $j$ with $n<j\le 2n$, the pair $\{k_n,k_j\}$ is a block of our partition $\nu$. Let 
\[
 F(\nu) = \sum_{k_1,\ldots, k_{2n}\sim \nu} f_{k_1} \dot{\otimes} f^*_{k_2} \dot{\otimes} \cdots \dot{\otimes} f^*_{k_{2n}}.
\]
We may rewrite $F(\nu)$ as
\begin{equation}\label{eq9.3}
F(\nu) = \sum\nolimits_\alpha \sum\nolimits_\beta a_{\alpha,\beta} \dot{\otimes} b_{\alpha,\beta}
\end{equation}
where $\alpha$ represents the set of indices $k_j$ such that the pair containing $j$ is split by the partition $[1,\ldots, n] [n+1,\ldots, 2n]$, and $\beta$ represents the remaining indices, and the sum is restricted to $(k_1,\ldots, k_{2n})\sim \nu$. Since the indices in $\beta$  correspond to pairs of indices $\{k_i,k_j\}$ with $\{i,j\}$ included either in $[1,\ldots, n]$ or in $[n+1,\ldots, 2n]$, we can rewrite the sum \eqref{eq9.3} as 
\[
 F(\nu) = \sum\nolimits_\alpha \sum\nolimits_{\beta',\beta''} \alpha_{\alpha,\beta'} \dot{\otimes} b_{\alpha,\beta''}.
\]
Then 
\[
 S(\nu) = \sum\nolimits_\alpha x_\alpha\dot{\otimes} y_\alpha
\]
with $x_\alpha = \sum_{\beta'} a_{\alpha,\beta'}$ and $y_\alpha = \sum_{\beta''} b_{\alpha,\beta''}$. By \eqref{eq8.0} we find
\[
 \|\widehat\varphi(S(\nu))\|\le \left\|\widehat\varphi \left(\sum\nolimits_\alpha x_\alpha \dot{\otimes} x^*_\alpha\right)\right\|^{1/2} \left\|\widehat\varphi\left(\sum\nolimits_\alpha y^*_\alpha \dot{\otimes} y_\alpha\right)\right\|^{1/2}.
\]
But now $\sum_\alpha x_\alpha \otimes x^*_\alpha$ is a sum of the kind $S(\nu')$ for some $\nu'$ but for which we know (by our initial choice relative to the pair $\{n,j\}$) that the pair $\{n,n+1\}$ appears in $\nu'$. If we then iterate the argument in the style of \cite{Buch} we end up with a number $0<\theta<1$ such that we have either 
\[
 \|\widehat\varphi(S(\nu))\| \le ({\cl C}')^\theta {\cl C}^{1-\theta}
\]
or 
\[
 \|\widehat\varphi(S(\nu))\| \le ({\cl C}'')^\theta {\cl C}^{1-\theta}
\]
where ${\cl C}' = \|\widehat\varphi(S(\nu'_0))\|$ and ${\cl C}'' = \|\widehat\varphi(S(\nu''_0))\|$. Thus we conclude that
\[
 {\cl C} \le (\max({\cl C}',{\cl C}''))^\theta {\cl C}^{1-\theta}
\]
and hence ${\cl C}\le \max({\cl C}',{\cl C}'')$. Since $S(\nu'_0) = (\sum f_k \dot{\otimes} f^*_k)^{\dot{\otimes} n}$ and $S(\nu''_0) = (\sum f^*_k \dot{\otimes} f_k)^{\dot{\otimes}n}$, this completes the proof.
\end{proof}

By a spin system we mean a system of anticommuting self-adjoint unitaries
assumed realized over a non-commutative probability space $(M,\tau)$.
In the $q$-Gaussian case with $q=-1$, Theorem \ref{thm9.1} describes the closed span of a spin system
in $\Lambda_p$, and exactly for the same reason as   in \cite{Buch2}
we obtain optimal constants for those.
\begin{cor}\label{cor9.1} If $(x_k)$ is    a spin system,
then   \eqref{buch} holds with 
  the same optimal constant $C_{\psi,p}$ as in the Gaussian case. In particular, this constant
grows like $\sqrt{p}$ when $p\to \infty$.
Moreover, the same result holds for the (Rademacher) sequence $(\vp_k)$,
and again the Gaussian constant is optimal.
\end{cor}
\begin{proof} Let $\psi(q)$ denote the function $\psi$ for a $q$-Gaussian system.
By Bo\.zejko and Speicher's results (see \cite{Buch2}) we have
$\psi(q)(\nu)=q^{i(\nu)}$ where $i(\nu)$ is the crossing number of the the partition $\nu$.
This implies $|\psi(q)(\nu)|=\psi(|q|)(\nu)$ and hence also $C_{\psi(q),p}=C_{\psi(|q|),p}$.
In particular, any spin system
 $(x_k)$   satisfies \eqref{buch} with the constant  $C_{\psi(1),p}$, i.e. the same constant as in the Gaussian
 case. We now address the Rademacher case.
Just as in \cite{Buch2} we use the fact that the sequences $(x_k \otimes x_k)$ and $(\vp_k)$ have the same distribution.
We then apply \eqref{eq9.2} to $\sum f_k \otimes x_k$ with $f_k=b_k \otimes x_k$. Recalling that the $x_k$'s are unitary, we find $\sum f^*_k \dot{\otimes} f_k=\sum \bar b_k \otimes b_k \otimes 1$ and
$\sum f_k \dot{\otimes} f^*_k=\sum   b_k \otimes \bar b_k \otimes 1$.
This gives us $\| \sum b_k \otimes \vp_k\|_{(p)}\le C \| \sum b_k \otimes \bar b_k\|^{1/2}$
where $C=C_{\psi(1),p}$ is the Gaussian constant. By the central limit theorem, the latter is optimal.
\end{proof}



In the rest of this  section we turn to the span of an i.i.d.\  sequence of Gaussian random matrices of size $N\times N$ in $\Lambda_p$. We will use ideas from \cite{HT2} and \cite{Buch}. We analyze the dependence in $N$ using a concentration of measure argument. Let $\{g_{ij}\mid i,j\ge 1\}$ be a doubly indexed family of complex valued Gaussian random variables such that $\EE g_{ij}=0$ and $\EE|g_{ij}|^2 = 1$. Let $Y^{(N)}$ be the random $N\times N$ matrix defined by
\[
Y^{(N)}({i,j}) = N^{-1/2} g_{ij}.
\]
Let $Y^{(N)}_1,Y^{(N)}_2,\ldots$ be an i.i.d.\ sequence of copies of $Y^{(N)}$ on some probability space $(\Omega,{\cl A},{\bb P})$. We will view $(Y^{(N)}_j)_{j\ge 1}$ as a sequence in $L_p({\bb P}\times \tau_N)$ where $\tau_N$ denotes the normalized trace on $M_N$.

By the Appendix \S \ref{sec13}, we know that, for any even $p\ge 2$, $(Y^{(N)}_j)_{j\ge 1}$ has $p$-th moments defined by pairings via the function
\[
Y^{(N)}(\nu) = \EE \tau_N(Y^{(N)\nu})
\]
where $Y^{(N)\nu} = Y^{(N)}_{k_1} Y^{(N)^*}_{k_2}\ldots Y^{(N)}_{k_{p-1}} Y^{(N)^*}_{k_p}$ for $k=(k_j)$ such that $k_i=k_j$ if and only if $(i,j)$ belong to the same block of  $\nu$. It is easy to see that the distribution of $Y^{(N)\nu}$ does not depend on the choice of such a $k$. Moreover, $\psi^{(N)}(\nu)\ge 0$ for any $\nu$ since $\psi^{(N)}(\nu)$ is
 a sum of terms of the form
  $$ \EE\left(Y^{(N)}_{k_1} (i_1,j_1) \ovl{Y^{(N)}_{k_2}(i_2,j_2)}\ldots \ovl{Y^{(N)}_{k_{p}}(i_p,j_p)}  \right)$$ 
  and, when $k \sim \nu$,  these are either $=0$ or $=N^{-p/2} (\EE|g_{11}|^2)^{p/2}=N^{-p/2}$. Therefore we again have
\[
\sum|\psi^{(N)}(\nu)| = \sum\psi^{(N)}(\nu) = \EE\tau_N(|Y^{(N)}|^p).
\]
By   \eqref{buch} we have:

\begin{cor}\label{cor9.13}
Let $p=2n$ and let $(b_k)$ be any finite sequence in $B=B(H)$. Let $f\in \sum b_k\otimes Y^{(N)}_k\in B\otimes L_p({\bb P}\times \tau_N)$. We have
\[
\|f\|_{(p)} \le (\EE\tau_N(|Y^{(N)}|^p))^{1/p} \left\|\sum b_k\otimes \bar b_k\right\|^{1/2}
\]
and this constant is optimal.
\end{cor}

\section{Non-commutative martingale inequalities}\label{sec10}

In this section, we assume given a filtration ${\cl M}_0 \subset {\cl M}_1\subset\cdots$ of von~Neumann subalgebras of $\cl M$. We assume for simplicity that $\cl M$ coincides with the von~Neumann algebra generated by $\cup {\cl M}_n$. We will denote again by ${\bb E}_n$ the conditional expectation with respect to ${\cl M}_n$. Then to any $f$ in $L_p(\tau)$ $(1\le p<\infty)$ we can associate  a martingale $(f_n)$ (defined by $f_n = {\bb E}_n(f)$) that converges in $L_p(\tau)$ to $f$. We will continue to denote $d_0 = f_0$ and $d_n=f_n-f_{n-1}$.

It is natural to expect that Theorems  \ref{thm3.1} and \ref{thm3.1+} will extend to the non-commutative 
case. However, at the time of this writing, we have   completed this task only for $p=4$. We also proved
below (see Theorem \ref{thm12.1}) a one sided version of  \eqref{eq3.1+} using the notion of $p$-orthogonal sums. 

Let $H_1,H_2$ be two Hilbert spaces. To lighten the notation in the rest of this section we set $B_1=  B(H_1)$ and $B_2=B(H_2)$. It is useful to observe that for any $f_1\in B_1\otimes L_4(\tau)$ and $f_2\in B_2\otimes L_4(\tau)$ we have 
\begin{equation}\label{eq10.7+}
 \widehat\tau(f^*_1 \dot{\otimes} f_1 \dot{\otimes} f^*_2\dot{\otimes} f_2) \approx \widehat\tau(f_1 \dot{\otimes} f^*_2 \dot{\otimes} f_2 \dot{\otimes} f^*_1)\succ 0.
\end{equation}
Indeed, the first sign $\approx$ is by the trace property while sign $\succ 0$ holds because $f_1 \dot{\otimes} f^*_2 \dot{\otimes} f_2 \dot{\otimes} f^*_1 \approx F\dot{\otimes} F^*$ with $F = f_1 \dot{\otimes} f^*_2\in B_1\otimes \ovl B_2 \otimes L_2(\tau)$. In the next lemma, we extend this observation to $\widehat\tau(f^*_1\dot{\otimes} f_1 \dot{\otimes} T(f^*_2\dot{\otimes} f_2))$ where $T\colon \ L_2(\tau)\to L_2(\tau)$ is a completely positive map (e.g.\ a conditional expectation). The reader can convince himself easily that the simplest case of maps of the form $T(x) = \sum a^*_kxa_k$ $(a_k\in {\cl M})$, follows immediately from \eqref{eq10.7+}.

\begin{lem}\label{lem10.7}
With the preceding notation, let ${\cl B} = B_1\otimes \ovl B_2 = B(H_1\otimes_2\ovl H_2)$. For any completely positive map $T\colon \ L_2(\tau)\to L_2(\tau)$, we have
for any $f_1\in B_1\otimes L_4(\tau)$ and $f_2\in B_2\otimes L_4(\tau)$
\[
 \widehat\tau(f^*_1\dot{\otimes} f_1 \dot{\otimes} T(f^*_2\dot{\otimes} f_2))\succ 0
\]
where the latter element is identified with an element of ${\cl B}\otimes \ovl{\cl B}$, via the permutation $\begin{pmatrix} 1&2&3&4\end{pmatrix} \to \begin{pmatrix} 2&3&1&4\end{pmatrix}$ of the tensorial factors that takes $\ovl B_1\otimes B_1 \otimes \ovl B_2 \otimes B_2$ to ${\cl B} \otimes \ovl{\cl B}$.
\end{lem}

\begin{rem}
 Let $E = L_4(\tau) \otimes \ovl{L_4(\tau)}$. Let $T\colon \ L_2(\tau)\to L_2(\tau)$ be a completely positive map, so that for any finite sequence $a_1,\ldots, a_n$ in $L_4(\tau)$ the matrix $[T(a^*_ia_j)]$ is in $L_2(M_n({\cl M}))_+$. Let $\Phi\colon \ E\otimes \ovl E\to {\bb C}$ be the bilinear form defined by
\[
 \Phi(a_1\otimes \bar b_1 \otimes \ovl{a_2\otimes \bar b_2}) = \tau(a^*_2a_1T(b^*_1b_2)).
\]
We claim that $\Phi$ is positive definite on $E\otimes\ovl E$, i.e.\ $\Phi(e\otimes\bar e)\ge 0$ for any $e$ in $E$. Indeed, if $e = \sum a_i\otimes\bar b_i\in E$ we have
\[
 \Phi(e\otimes\bar e) = \sum\nolimits_{ij} \tau(a^*_ja_iT(b^*_ib_j)),
\]
so that, if $\tau_n$ denote the trace on $M_n({\cl M})$, we have $\Phi(e\otimes\bar e) = \tau_n(\alpha\beta) = \tau_n(a^{1/2}\beta \alpha^{1/2})\ge 0$ where $\beta_{ij} = T(b^*_ib_j)$ $\alpha_{ij} = a^*_ia_j$ and of course $\alpha\ge 0$ and $\beta\ge 0$. This proves our claim.
\end{rem}

\begin{proof}[Proof of Lemma \ref{lem10.7}]
Let $L_p=L_p(\tau)$.
Consider
$f_1\otimes \bar f_2\in B_1\otimes L_4 \otimes \ovl B_2 \otimes \bar L_4$.
  Let $g \in {\cl B} \otimes E
  $   be the element obtained by the natural permutation of factors from  $f_1\otimes \bar f_2$.
Then an easy verification shows that
\[
 \widehat\tau(f^*_1\dot{\otimes} f_1\dot{\otimes} T(f^*_2 \dot{\otimes} f_2)) \approx   (I\otimes \Phi) (g\otimes \bar g)
\]
or more precisely with the notation indicated in Lemma~\ref{lem8.3}
\[
 \widehat\tau(f^*_1\dot{\otimes} f_1\dot{\otimes} T(f^*_2 \dot{\otimes} f_2))  \approx  (\Phi)_{24} (g\otimes \bar g)
\]
so that Lemma~\ref{lem10.7} follows from Lemma~\ref{lem8.3} and the preceding Remark.
\end{proof}

\begin{pro}\label{pro8.5}
Let $(f_n)_{n\ge 0}$ be a martingale in $B\otimes L_4(\tau)$. Assume for simplicity $f=f_N$ for some $N\ge 0$. Let $g = f^*\dot{\otimes} f - \sum d^*_n \dot{\otimes} d_n$. We have then
\begin{equation}\label{eq10.9}
 \|g\|_{(2)}\le \|f\|_{(4)} (\|\sigma_r\|^{1/2}_{(2)} + \|\sigma_c\|^{1/2}_{(2)}).  
\end{equation}
where
\[
 \sigma_r = \sum {\bb E}_{n-1}(d_n\dot{\otimes} d^*_n)\quad \text{and}\quad \sigma_c = \sum {\bb E}_{n-1}(d^*_n\dot{\otimes} d_n).
\]
\end{pro}

\begin{proof}
As usual we start by $g = x+y$ with $x = \sum d^*_n \dot{\otimes} f_{n-1}$ and $y = \sum f^*_{n-1} \dot{\otimes} d_n$, so that $\|g\|_{(2)} \le \|x\|_{(2)} + \|y\|_{(2)}$. Then 
\[
 \|x\|^2_{(2)} = \|\widehat \tau(x\dot{\otimes} x^*)\| = \left\|\widehat \tau\left(\sum d^*_n \dot{\otimes} f_{n-1} \dot{\otimes} f^*_{n-1} \dot{\otimes} d_n\right)\right\|.
\]
Let $\delta_n = f-f_{n-1}$. Note that since ${\bb E}_{n-1}(\delta_n) = 0$ 
\[
 {\bb E}_{n-1}(f\dot{\otimes} f^*) = f_{n-1}\dot{\otimes} f^*_{n-1} + {\bb E}_{n-1}(\delta_n \dot{\otimes} \delta^*_n)
\]
and hence
$$\widehat \tau\left(\sum d^*_n \dot{\otimes} f_{n-1} \dot{\otimes} f^*_{n-1} \dot{\otimes} d_n\right)  = \widehat \tau \left(\sum d^*_n \dot{\otimes} {\bb E}_{n-1} (f\dot{\otimes} f^*)\dot{\otimes} d_n\right)
 - \widehat \tau \left(\sum d^*_n \dot{\otimes} {\bb E}_{n-1}(\delta_n \dot{\otimes} \delta^*_n)\dot{\otimes} d_n\right),$$
By the trace property and by Lemma \ref{lem10.7}, these last three terms can all be viewed as $\succ 0$ in
a suitable permutation of the factors. This shows by \eqref{eq1.1-}
$$\left\|\widehat \tau\left(\sum d^*_n \dot{\otimes} f_{n-1} \dot{\otimes} f^*_{n-1} \dot{\otimes} d_n\right)\right\|\le \left\|\widehat \tau\left(\sum d^*_n \dot{\otimes} {\bb E}_{n-1} (f\dot{\otimes} f^*) \dot{\otimes} d_n\right)\right\|.$$
 Since
 \[
 \widehat \tau\left(\sum d^*_n\dot{\otimes} {\bb E}_{n-1}(f\dot{\otimes} f^*)\dot{\otimes} d_n\right)\approx  \widehat \tau\left(   {\bb E}_{n-1}(f\dot{\otimes} f^*)\dot{\otimes} \sum d_n\dot{\otimes} d^*_n\right),
\]
and since ${\bb E}_{n-1}$ is  self-adjoint, we have
 \[
 \left\|\widehat \tau\left(\sum d^*_n\dot{\otimes} {\bb E}_{n-1}(f\dot{\otimes} f^*)\dot{\otimes} d_n\right)\right\|=\left\| \widehat \tau\left(f\dot{\otimes} f^*\dot{\otimes} \sum {\bb E}_{n-1}(d_n\dot{\otimes} d^*_n)\right)
\right\|,
\]
thus we find
\[
 \|x\|^2_{(2)} \le \|\widehat \tau(f\dot{\otimes} f^*\dot{\otimes} \sigma_r)\|\le \|f\|^2_{(4)} \|\sigma_r\|_{(2)}.
\]
A similar reasoning leads to
\[
 \|y\|^2_{(2)}\le \|f\|^2_{(4)} \|\sigma_c\|_{(2)},
\]
so we conclude
\[
 \|g\|_{(2)} \le \|f\|_{(4)} (\|\sigma_r\|^{1/2}_{(2)} + \|\sigma_c\|^{1/2}_{(2)}).\eqno\qed
\]
\renewcommand{\qed}{}\end{proof}

To complete the case $p=4$, we need to check the non-commutative extension of Lemma~\ref{lem3.2} as follows:

\begin{lem}\label{lem10.5}
Let $\theta_n$ be any finite sequence in $B\otimes L_4(\tau)$, let $\beta_n = \theta^*_n \dot{\otimes} \theta_n$ and $\alpha_n = {\bb E}_n(\beta_n)$. Then
\[
 \left\|\sum\alpha_n\right\|_{(2)} \le 2\left\|\sum \beta_n\right\|_{(2)}.
\]
\end{lem}

\begin{proof}
The proof is essentially the same as for Lemma~\ref{lem3.2}. We just need to observe that if $n\le k$ we have
\[
 \widehat\tau(\alpha_n \dot{\otimes} \alpha_k) = \widehat\tau(\alpha_n \dot{\otimes} \beta_k),
\]
but also by Lemma \ref{lem10.7} (and the trace property) 
\begin{equation}
 \widehat\tau(\alpha_n \dot{\otimes} \beta_k) \approx \widehat\tau(\beta_k \dot{\otimes} \alpha_n) \succ 0\tag*{$\forall n,k$},\end{equation}
so that again we  have 
\[
 \left\|\sum\nolimits_{n\le k} \widehat\tau(\alpha_n \dot{\otimes} \beta_k)\right\|\le \left\|\sum\nolimits_{n,k} \widehat\tau(\alpha_n \dot{\otimes} \beta_k)\right\| = \|\widehat\tau( \alpha\dot{\otimes} \beta)\| \le \|\alpha\|_{(2)} \|\beta\|_{(2)}
\]
and similarly for $\|\sum_{n>k}\|$.
\end{proof}
Let $S_r=\sum d_j\dot{\otimes} d^*_j   $
abd $S_c=\sum d^*_j \dot{\otimes} d_j $. Applying 
this Lemma to   Proposition~\ref{pro8.5}, we find
$$\|g\|_{(2)} \le 2\sqrt{2}\|f\|_{(4)} \max\{\|S_r\|^{1/2}_{(2)} , \|S_c\|^{1/2}_{(2)} \}.$$
We then obtain by the same reasoning as for the commutative case:
\begin{cor}\label{cor10.6}
There is a constant $C$ such that for any finite martingale $f_0,\ldots, f_N$ in $L_4(\tau)$ we have
\[
 C^{-1}\max\left\{\left\|\sum d_j\dot{\otimes} d^*_j\right\|^{1/2}_{(2)}, \left\|\sum d^*_j \dot{\otimes} d_j\right\|^{1/2}_{(2)}\right\} \le \|f\|_{(4)} \le C\max\left\{\left\|\sum d_j\dot{\otimes} d^*_j\right\|^{1/2}_{(2)}, \left\|\sum d^*_j \dot{\otimes} d_j\right\|^{1/2}_{(2)}\right\}.
\]
\end{cor}
\begin{rem} We leave as an open problem whether the extension of the left hand side
of Corollary \ref{cor10.6} is valid for any even integer $p>4$. Note however that the right hand side
is proved below as a consequence of Theorem \ref{thm12.1}.
\end{rem}
We will now  extend Theorem \ref{thm3.1+} to the non-commutative case for $p=4$.\\
Given $f\in B\otimes L_4(\tau)$ let us denote
$$\|f\|_{[4]} =\max \{ \|\hat \tau(\sum d_n\dot\otimes   d^*_n \dot\otimes d_n \dot\otimes d^*_n)\|^{1/4}  ,\|\sigma_r\|^{1/2}_{(2)} , \|\sigma_c\|^{1/2}_{(2)}\}.$$
\begin{cor}\label{cor10.6bis}
For any finite martingale $f_0,\ldots, f_N$ in $L_4(\tau)$ we have
\begin{equation}\label{eq10.10}
 2C^{-1} \|f\|_{[4]} \le \|f\|_{(4)} \le 2C\|f\|_{[4]},
\end{equation} where $C$ is as in the preceding statement.
\end{cor}
\begin{proof} By the preceding Corollary and by Lemma \ref{lem10.5} we
have $\max\{\|\sigma_r\|^{1/2}_{(2)}  , \|\sigma_c\|^{1/2}_{(2)}\}\le 2C  \|f\|_{(4)} $.
Moreover, by \eqref{eq10.7+}
we have $\hat\tau(\sum d_n\dot\otimes  d^*_n \dot\otimes d_n \dot\otimes   d^*_n)\prec
 \hat\tau(\sum d_n\dot\otimes  d^*_n) \dot\otimes  (\sum d_j\dot\otimes  d^*_j)^*)$, and hence \\
 $\|\hat \tau(\sum d_n\dot\otimes  d^*_n \dot\otimes d_n \dot\otimes   d^*_n)\|^{1/4}\le \left\|\sum d_j\dot{\otimes} d^*_j\right\|^{1/2}_{(2)}$.
Therefore, the left hand side of \eqref{eq10.10} follows.\\
To prove the right hand side, we will use the preceding Corollary. \\
Let $x_n=d_n^*\dot\otimes d_n$ and $y_n=d_n\dot\otimes d_n^*$.
We have
$$\sum x_n =\sigma_c+ \sum \delta_n$$
where $\delta_n= x_n- \EE_{n-1} x_n$. 
Then, by the triangle inequality
$\|\sum x_n\|_{(2)}\le  \|\sigma_c\|_{(2)} + \|\sum \delta_n\|_{(2)} $
and by the orthogonality of the martingale differences $(\delta_n)$ we have
$ \|\sum \delta_n\|^2_{(2)} = \|\sum \hat\tau (\delta^*_n \dot\otimes  \delta_n)\|$.
But by Remark \ref{rem8.9} we have
$ \hat\tau (\delta^*_n \dot\otimes  \delta_n)\prec  \hat\tau (x^*_n \dot\otimes  x_n)$
and hence also  $ \sum\hat\tau (\delta^*_n \dot\otimes  \delta_n)\prec  \sum\hat\tau (x^*_n \dot\otimes  x_n)$
from which follows by Lemma \ref{lem1.1} that
$ \|\sum \hat\tau (\delta^*_n \dot\otimes  \delta_n)\|\le  \|\sum \hat\tau (x^*_n \dot\otimes  x_n)\|$.\\
Recapitulating, we find $\|\sum x_n\|_{(2)}\le  \|\sigma_c\|_{(2)} + \|\sum \hat\tau (x^*_n \dot\otimes  x_n)\|^{1/2},$
and a fortiori
$$\|\sum x_n\|^{1/2}_{(2)}\le  \|\sigma_c\|^{1/2}_{(2)} + \|\sum \hat\tau (x^*_n \dot\otimes  x_n)\|^{1/4}\le 2 \| f\|_{[4]}.$$
Since a similar argument applies to majorize $\|\sum y_n\|_{(2)}$, by Corollary \ref{cor10.6} we obtain
$$C^{-1} \| f\|_{(4)}\le \max\{\|\sum x_n\|^{1/2}_{(2)},\|\sum y_n\|^{1/2}_{(2)}  \} \le 2 \| f\|_{[4]}.$$
\end{proof}

\section{$p$-orthogonal sums}\label{sec11}

Let $L_p(\tau)$ be 
as before the ``non-commutative'' $L_p$-space associated to a von Neumann algebra equipped with a standard (= faithful, 
normal) semi-finite trace. (Of course, if $M$ is commutative, we 
recover the classical $L_p$ associated to a measure space.) Let $p\ge 2$ 
be an even integer. A family $d = (d_i)_{i\in I}$ is called $p$-orthogonal if, 
for any {\it  injective} function $g\colon \ [1,2,\ldots, p]\to I$ we have
$$\tau(d^*_{g(1)} d_{g(2)} d^*_{g(3)} d_{g(4)}\ldots d^*_{g(p-1)} d_{g(p)}) = 
0.$$
Clearly, any 
martingale difference sequence is $p$-orthogonal, but 
the class of $p$-orthogonal sums is more general.
In the commutative case, \ie for classical random variables, this notion
is very close to that of  ``multiplicative sequence" already considered in
the literature, see the references in \cite{P5}, on which this section is modeled.

By a natural extension, we will say that a sequence $(d_j)_{j\in I}$ in $B \otimes L_p(\tau) $
is $p$-orthogonal if for any {\it  injective} function $g\colon \ [1,2,\ldots, p]\to I$
as before we have
$$\widehat \tau(d^*_{g(1)} \dot\otimes d_{g(2)} \dot\otimes  \ldots \dot\otimes  d^*_{g(p-1)} \dot\otimes d_{g(p)}) = 
0.$$

The method used in \cite{P5}, that is based on a combinatorial formula involving the ``M\"obius function",
is particularly easy to adapt to our setting where $\Lambda_p$ takes the place of $L_p$.

We will use crucially some well known ideas from the combinatorial theory of 
partitions, which can be found, for instance, in  
the book \cite{A}. We denote by $P_n$ the lattice of all partitions of $[1,\ldots, 
n]$, equipped with the following order:\ we write $\sigma\le \pi$ (or 
equivalently $\pi\ge\sigma$) when every ``block'' of the partition $\sigma$ is 
contained in some block of $\pi$. Let $\dot 0$ and $\dot 1$ be respectively 
the minimal and maximal elements in $P_n$, so that $\dot 0$ is the partition 
into $n$ singletons and $\dot 1$ the partition formed of the single set 
$\{1,\ldots, n\}$. We denote by $\nu(\pi)$ the number of blocks of $\pi$ (so
that $\nu(\dot 0) = n$ and $\nu(\dot 1) = 1$).

For any $\pi$ in $P_n$ and any $i=1,2,\ldots, n$, we denote by $r_i(\pi)$ the 
number of blocks (possibly $=0$) of $\pi$ of cardinality $i$. In particular, 
we have $\sum^n_1 ir_i(\pi) = n$ and $\sum^n_1 r_i(\pi) = \nu(\pi)$.

Given two partitions $\sigma,\pi$ in $P_n$ with $\sigma\le \pi$ we denote by 
$\mu(\sigma,\pi)$ the M\"obius function, which has the following fundamental 
property:

 Let $V$ be a vector space. Consider two functions 
$\Phi\colon \ P_n\to V$ and $\Psi\colon \ P_n\to V$.
 If 
$\Psi(\sigma) = \sum\limits_{\pi\le \sigma} \Phi(\pi)$,
 then
$\Phi(\sigma) = \sum_{\pi\le\sigma} \mu(\pi,\sigma)\Psi(\pi).$

 Essentially equivalently, if
$  \Psi(\sigma)  = \sum_{\pi\ge\sigma} \Phi(\pi)$, 
 then
$\Phi(\sigma)  = \sum_{\pi\ge \sigma} \mu(\sigma,\pi) \Psi(\pi).$

In particular we have:
\begin{equation}\label{eq-moe2}\forall~\sigma\ne \dot 
0\quad \sum_{\dot 0\le \pi \le \sigma} \mu(\pi,\sigma) = 0.  \end{equation}

 The last assertion follows from the above  fundamental property applied with $\Phi$ equal to the delta 
function at $\dot 0$ (i.e.\ $\Phi(\pi) = 0$ $\forall~\pi\ne \dot 0$ and 
$\Phi(\dot 0) = 1$) and $\Psi\equiv 1$.

We also recall Sch\"utzenberger's theorem (see \cite{A}):

 For any $\pi$ we have
$$\mu(\dot 0, \pi) = \prod^n_{i=1} [(-1)^{i-1} (i-1)!]^{r_i(\pi)},$$
and consequently
$$\sum_{\pi\in P_n} |\mu(\dot 0,\pi)| = n!.$$

We now apply these results to set the stage for the questions of interest to 
us. Let $E_1,\ldots, E_n, V$ be vector spaces equipped with a multilinear
form  (= a ``product'')  
$$\varphi\colon \ E_1\times\cdots\times E_n\to V.$$
 Let $I$ 
be a finite set. For each $k=1,2,\ldots, n$ and $i\in I$, we give ourselves 
elements $d_i(k) \in E_k$, and we form the sum
$$F_k = \sum_{i\in I} d_i(k).$$
Then we are interested in ``computing'' the quantity
$\varphi(F_1,\ldots, F_n).$
We have obviously
$$\varphi(F_1,\ldots, F_n) = \sum_g \varphi(d_{g(1)}(1),\ldots, d_{g(n)}(n))$$
where the sum runs over all functions $g\colon \ [1,2,\ldots, n] \to I$. Let 
$\pi(g)$ be the partition associated to $g$, namely the partition obtained 
from $\bigcup\limits_{i\in I} g^{-1}(\{i\})$ after deletion of all the empty 
blocks.
We can write
$$\varphi(F_1,\ldots, F_n) = \sum_{\sigma\in P_n} \Phi(\sigma)$$
where
$\Phi(\sigma) = \sum_{g\colon\ \pi(g) =\sigma} \varphi(d_{g(1)}(1),\ldots, 
d_{g(n)}(n))$.
Let $\Psi(\sigma) = \sum\limits_{\pi\ge\sigma} 
\Phi(\pi)$. \\ Using \eqref{eq-moe2} (with $\sigma,\pi$ exchanged), we have then:  
 \begin{align}\varphi(F_1,\ldots, F_n) &= \Phi(\dot 0) + \sum_{\dot 0<\sigma} 
\Phi(\sigma)
= \Phi(\dot 0) + \sum_{\dot 0<\sigma} \sum_{\pi\ge \sigma} \mu(\sigma,\pi) 
\Psi(\pi)\\
&= \Phi(\dot 0) + \sum_{\dot 0<\pi} \Psi(\pi)\cdot \sum_{\dot 0<\sigma\le\pi} 
\mu(\sigma,\pi)
= \Phi(\dot 0) - \sum_{\dot 0<\pi} \Psi(\pi) \mu(\dot 0,\pi).
\end{align}
Recapitulating, we found:
\begin{equation}\label{eq-moe}\varphi(F_1,\ldots, F_n) = \Phi(\dot 0) - \sum_{\dot 0<\pi} \Psi(\pi) \mu(\dot 0, 
\pi)\end{equation}
where
$$\Phi(\dot 0) = \sum_{g~{\rm injective}} \varphi(d_{g(1)}(1),\ldots, 
d_{g(n)}(n)) \ 
{\rm and} \  
\Psi(\pi) = \sum_{g\colon \ \pi(g)\ge \pi} \varphi(d_{g(1)}(1),\ldots, 
d_{g(n)}(n)).$$

\begin{thm}\label{thm12.1}
Let $p=2n$ be an even integer $> 2$. Then for any $p$-orthogonal finite sequence $(d_j)_{j\in I}$ in $B\otimes L_p(\tau)$ we have
\begin{equation}\label{eqporth}
 \left\|\sum d_j\right\|_{(p)} \le (3\pi/2)p\max\left\{\left\|\sum d_j\dot{\otimes} d^*_j\right\|^{1/2}_{(p/2)}, \left\|\sum d^*_j \dot{\otimes} d_j\right\|^{1/2}_{(p/2)}\right\}.
\end{equation}
\end{thm}

\begin{proof}
This proof is modeled on that in \cite{P5} so we will be deliberately sketchy. Let $f = \sum d_j$. We can write
\[
 \widehat\tau [(f^*\dot{\otimes} f)^{\dot{\otimes}n}] = -\sum_{\dot 0<\pi} \mu(\dot{0},\pi)\Psi(\pi)
\]
where $\Phi$ and $\Psi$ are defined by
\[
 \Phi(\sigma) = \sum_{g\colon \ \pi(g)=\sigma} \widehat\tau(d^*_{g(1)} \dot{\otimes} d_{g(2)} \ldots \dot{\otimes} d^*_{g(p-1)} \dot{\otimes} d_{g(p)})
\]
and $\Psi = \sum\limits_{\sigma\ge \pi} \Phi(\sigma)$, or equivalently,
\[
\Psi(\pi) = \sum\nolimits_{g\sim\sigma} \widehat\tau(d^*_{g(1)} \dot{\otimes} \cdots \dot{\otimes} d_{g(p)})
\]
where  (as  in \S\ref{sec9}) $g\sim\sigma$ means that $g(i) = g(j)$ whenever $i,j$ are in the same block of  $\sigma$.
Here the functions $\Phi$ and $\Psi$ take values in $\ovl B \otimes B \otimes\cdots\otimes \ovl B \otimes B$ where $\ovl B \otimes B$ is repeated $n$-times .

Let $\alpha = 3\pi/4$ as in \cite{P5}. Arguing as in \cite[p.~912]{P5} we see that it suffices to prove that
\begin{equation}\label{eq12.1}
 \|\Psi(\pi)\| \le (\alpha\Delta)^{p-r_1(\pi)} \|f\|^{r_1(\pi)}_{(p)}
\end{equation}
where $\Delta = \max\left\{\left\|\sum d_j\dot{\otimes} d^*_j\right\|^{1/2}_{(p/2)}, \left\|\sum d^*_j \dot{\otimes} d_j\right\|^{1/2}_{(p/2)}\right\}$, and we recall that $r_1(\pi)$ is the number of singletons in $\pi$. Let ${\bb F}_I$ be the free group with generators $(g_j)_{j\in I}$, and let $\varphi$ be the normalized trace on the von~Neumann algebra of ${\bb F}_I$.

Let $f_k = \sum_{i\in I} d_i(k)$ be a finite sum in $B \otimes L_p(\tau)$, $k=1,\ldots, p$. We denote by
\[
 \tilde f_k = \sum\nolimits_{i\in I} \lambda(g_i) \otimes d_i(k)
\]
the corresponding sum in $L_p(\varphi\times \tau)\otimes B$. Note that by \eqref{eq9.4} we know that
\begin{equation}\label{eq12.2}
 \|\tilde f_k\|_{(p)} \le (3\pi/4) \max\left\{\left\|\sum d_j\dot{\otimes} d^*_j\right\|^{1/2}_{(p/2)}, \left\|\sum d^*_j \dot{\otimes} d_j\right\|^{1/2}_{(p/2)}\right\}.
\end{equation}
Let $\pi$ be a partition of $[1,\ldots, p]$. Let $B_1$ be the union of all the singletons in $\pi$ and let $B_2$ be the complement of $B_1$ in $[1,\ldots, p]$. By the construction in the proof of \cite[Sublemma~3.3]{P5} for a suitable discrete group $G$ there are $F_1,\ldots, F_p$ in $L_p(\tau_G\times\tau)\otimes B$ such that $\|F_k\|_{(p)} = \|\tilde f_k\|_{(p)}$ for all $k$ in $B_2$, $\|F_k\|_{(p)} = \|f_k\|_{(p)}$ for all $k$ in $B_1$, and also
\[
 \widehat\tau\left(\sum_{\pi(g)\ge \pi} d_{g(1)}(1) \dot{\otimes}\cdots \dot{\otimes} d_{g(p)}(p) \right) = \widehat{(\tau_G\times\tau)} [F_1 \dot{\otimes}\cdots\dot{\otimes} F_p].
\]
Then if we apply this to $d_j(k) = d^*_j$ if $k$ is odd and $d_j(k) = d_j$ if $k$ is even we find by \eqref{eq7.2}
\begin{align*}
 \|\Psi(\pi)\| &= \|\widehat{(\tau_G\times\tau)} (F_1 \dot{\otimes} \cdots\dot{\otimes} F_p)\|\\
&\le \prod^p_{k=1} \|F_k\|_{(p)}\\
&\le \|f\|^{|B_1|}_{(p)} \cdot \prod_{k\in B_2} \|\tilde f_k\|_{(p)}
\end{align*}
and by \eqref{eq12.2} we obtain \eqref{eq12.1}.
\end{proof}

\begin{cor}\label{corrad} Let $p$ be an even integer.  
Assume $\tau(1)=1$.
Let $(f_j)$ be a $p$-orthogonal  sequence in $L_p(\tau)$ that is orthonormal in
$L_2(\tau)$. Consider a finite sequence $(b_j)$ with   $b_j\in B$. 
   We have then 
\[ \|\sum b_j  \otimes \bar b_j\|^{ 1/2}\le 
 \left\|\sum b_j\otimes f_j \right\|_{(p)} \le (3\pi/2)p \|\sum b_j  \otimes \bar b_j\|^{ 1/2}
.
\]
Let $E_p$ denote   the
closed span of  $(f_j)$ in $\Lambda_p(\tau)$. 
Then $E_p$ is  completely isomorphic to $OH$, and
moreover,  the orthogonal projection $P$ 
from $L_2(\tau)$ onto the span of $(f_j)$  is c.b. on $\Lambda_p$ with c.b. norm
at most $ (3\pi/2)p $.
\end{cor}
\begin{proof} 
The right hand side of 
the inequality follows from \eqref{eqporth}
 since
$(d_j)=(b_j\otimes f_j)$ is clearly $p$-orthogonal.
 By Corollary \ref{cor2.4nc}, the inclusion $\Lambda_p(\tau)\to \Lambda_2(\tau)=L_2(\tau)$
has c.b norm 1. Using this the left hand side follows. This shows that   $E_p\simeq OH$.
The projection $P$ can then be factorized as $\Lambda_p\to \Lambda_2 \to E_2\to E_p$,
which implies    $\|P\colon\ \Lambda_p\to E_p\|_{cb} \le  (3\pi/2)p $, since the first two arrows
are completely contractive.
\end{proof}
\section{Lacunary Fourier series in $\Lambda_p$}\label{sec12}

\indent

In this section, we review the results of \cite{Har} and \cite{P5} with $\Lambda_p$ in place of $L_p$, and again we find the space $OH$ appearing in place of $R_p\cap C_p$. To save space, it will be convenient to adopt the general viewpoint in \cite[\S4]{P5}, although this may seem obscure to a reader unfamiliar with \cite{Har}.\medskip 

\n {\bf Notation:}\ Let $1 = \sum_{k\in J} P_k$ be an orthogonal decomposition of the identity of $L_2(\tau)$ on a semi-finite ``non-commutative'' measure space $(M,\tau)$. Let $p =2n$ be an even integer $>2$. Let $(d_j)_{j\in I}$ be a finite family in $B \otimes L_p(\tau)$. We set $x^\omega = x^*$ if $n$ is odd and $x^\omega=x$ in $n$ is even.

\begin{thm}\label{thm11.1}
Let $F$ be the set of all injective functions $g\colon \ [1,2,\ldots, n]\to I$. For any $g$ in $F$, we let $x_g =   d^*_{g(1)} \dot{\otimes} d_{g(2)} \dot{\otimes}   d^*_{g(3)} \ldots d^\omega_{g(n)}$. We define
\[
 N(d)=\sup_{k\in J} \text{\rm card}\{g\in F\mid (P_k\otimes I) (x_g)\ne 0\}.
\]
Then
\[
 \bigg\|\sum_{j\in I} d_j\bigg\|_{(p)} \le \left[(4N(d))^{1/p} + p \frac{9\pi}8\right] \max\left\{\left\|\sum d_j\dot{\otimes} d^*_j\right\|^{1/2}_{(p/2)}, \left\|\sum d^*_j \dot{\otimes} d_j\right\|^{1/2}_{(p/2)}\right\}.
\]
\end{thm}

\begin{proof}
Since we follow closely the ideas in \cite{Har} and \cite{P5} we will merely sketch the proofs. We have
\[
 \|f\|^n_{(p)} = \|f^* \dot{\otimes} f\cdots\dot{\otimes} f^\omega\|_{(2)}.
\]
Arguing as in \cite[p.~919]{P5} we find
\[
 f^* \dot{\otimes} f\cdots \dot{\otimes} f^\omega = \Phi(\dot{0}) - \sum_{\dot{0}<\pi\in P_n} \mu(\dot{0},\pi)\Psi(\pi)
\]
with $\Phi(\sigma)= \sum_{\pi(g)=\sigma} x_g$ and
\[
 \Psi(\pi) = \sum\nolimits_{\sigma\le\pi} \Phi(\sigma).
\]
Note that $\Phi(\dot{0}) = \sum_{g\in F} x_g$. Using a suitable adaptation of \cite[Sublemma 3.3]{P5} and replacing \cite[(3.5)]{P5} by   Corollary~\ref{cor9.3} above, we find:
\[
 \|\Psi(\pi)\|_{(2)} \le \|f\|^{r_1(\pi)}_{(p)} (\alpha\Delta)^{n-r_1(\pi)}
\]
where
\[
 \Delta = \max\left\{\left\|\sum d_j\dot{\otimes} d^*_j\right\|^{1/2}_{(p/2)}, \left\|\sum d^*_j \dot{\otimes} d_j\right\|^{1/2}_{(p/2)}\right\}.
\]
Let $F_k  = \{g\in F\mid (id\otimes P_k)x_g\ne 0\}$ and $\Phi_k = (id\otimes P_k) \Phi(\dot{0})$, so that
 $$\Phi(\dot{0}) = \sum\Phi_k \quad{\rm and}\quad \Phi_k = \sum_{g\in F_k} x_g(k)$$ where $x_g(k) = (id\otimes P_k)(x_g)$. By \eqref{eq2.03} we have by ``orthogonality'' of $\Phi_k$
\[
 \|\Phi(\dot{0})\|^2_{(2)} = \left\|\widehat\tau\left(\sum \Phi_k\dot{\otimes} \Phi^*_k\right)\right\|.
\]
Since $\text{card}(F_k) \le N(d)$, by \eqref{eq1.3} and Lemma \ref{lem8.4} we have
\[
 \widehat\tau\left( \sum \Phi_k \dot{\otimes} \Phi^*_k\right) \prec N(d)\widehat\tau \left(\sum\nolimits_{k}\sum\nolimits_{g\in F_k} x_g(k) \dot{\otimes} x_g(k)^*\right) = N(d) \widehat \tau \left(\sum\nolimits_{g\in F} x_g\dot{\otimes} x^*_g\right).
\]
Therefore, we find
\[
 \frac1{N(d)} \|\Phi(\dot{0})\|^2_{(2)} \le \left\|\widehat\tau \left(\sum\nolimits_{g\in F} x_g \dot{\otimes} x^*_g\right)\right\| \le \left\|\widehat\tau \left(\sum_{g(1),\ldots, g(n)} d^*_{g(1)} \dot{\otimes}\cdots \dot{\otimes} d^\omega_{g(n)} \dot{\otimes} d^{\omega*}_{g(n)} \ldots \dot{\otimes} d_{g(1)}\right)\right\|
\]
and hence
\[
 \|\Phi(\dot{0})\|^2_{(2)} \le N(d) \left\|\widehat\tau \left(\left(\sum d_j \dot{\otimes} d^*_j\right)^{\dot{\otimes}n}\right)\right\| = N(d) \left\|\sum d_j\dot{\otimes} d^*_j\right\|^{p/2}_{(p/2)}.
\]
Thus we may conclude by the same reasoning as in \cite[p.~920]{P5}.
\end{proof}

We can now reformulate the main result of \cite{Har} with $\Lambda_p$ in place of $L_p$:

\begin{cor}\label{cor11.2}
Fix an even integer $p=2n>2$. Let $E \subset \Gamma$ be a subset of a discrete group $\Gamma$ with unit $e$. For any $\gamma$ in $\Gamma$ let $Z_p(\gamma,E)$ be the cardinality of the set of injective functions $g\colon \ [1,\ldots, n]\to E$ such that
\[
 \gamma = g(1) g(2)^{-1} g(3)\ldots g(n)^w
\]
where $g^w = g^{-1}$ if $n$ is even and $g^w=g$ if $n$ is odd. We set
\[
 Z(E) = \sup\{Z_p(\gamma,E)\mid \gamma\in \Gamma\}.
\]
Then for any finitely supported family $(b(t))_{t\in E}$ in $B=B(H)$ we have
\begin{equation}\label{eq11.1}
 \left\|\sum\nolimits_{t\in E} \lambda(t) \otimes b(t)\right\|_{(p)} \le ((4Z(E))^{1/p} + (9\pi/8)p) \left\|\sum b(t)\otimes \ovl{b(t)}\right\|^{1/2}.
\end{equation}
\end{cor}

\begin{proof}
Here $L_2(\tau)$ is the $L_2$-space associated to the usual trace on the von~Neumann algebra associated to $\Gamma$. Any element in $L_2(\tau)$ has an orthonormal expansion in a ``Fourier series'' $x = \sum_{t\in \Gamma} x(t)\lambda(t)$, so we can apply Theorem~\ref{thm11.1} to this orthogonal decomposition (with $J=\Gamma$). Note that if $(t_j)$ are distinct elements in $E$ and if $d_j = \lambda(t_j)\otimes b(t_j)$ we have $N(d)\le Z(E)$. Lastly, we note that in the present situation, since $L_p(\varphi) = {\bb C}$, the term previously denoted by $\Delta$ coincides with
$$
 \left\|\sum b(t)\otimes \ovl{b(t)}\right\|^{1/2}.\eqno\qed
$$
\renewcommand{\qed}{}\end{proof}

We have also a (one sided) version of  the Littlewood--Paley inequality for $\Lambda_p$:

\begin{cor}\label{cor11.3}
Consider a Fourier series of the form
\[
 f = \sum\nolimits_{n>0} \hat f(n)e^{int}
\]
where $n\to\hat f(n)$ is a finitely supported $B(H)$-valued function. Let 
\begin{align*}
\Delta_n &= \sum\nolimits_{2^n\le k <2^{n+1}} \hat f(k)e^{ikt}
\end{align*}
 and let
\[
 {\cl S}(f) =   \sum \Delta_n \dot{\otimes} \ovl{\Delta_n} .
\]
There is an absolute constant $C$ such that for any even integer $p\ge 2$
\[
 \|f\|_{(p)} \le C_p\|{\cl S}(f)\|^{1/2}_{(p/2)}.
\]
\end{cor}

\begin{proof}
 It suffices to prove the inequality separately for the cases $f = \sum_m \Delta_{2m}$ and $f = \sum_m \Delta_{2m+1}$. But then each of these cases follows from Theorem~\ref{thm11.1} and elementary arithmetic involving lacunary sequences.
\end{proof}

It may be worthwhile to point out that in the commutative case, the following variant of Theorem~\ref{thm11.1} holds:

\begin{thm}\label{thm11.4}
Consider the same situation as in Theorem~\ref{thm11.1} but with $M$ commutative so that $L_2(\tau)$ can be identified with $L_2(\Omega,\mu)$. Let $y_g = d_{g(1)} \dot{\otimes} d_{g(2)} \dot{\otimes} \cdots \dot{\otimes} d_{g(n)}$ and let
\[
 N_+(d) = \sup_{k\in J} \text{\rm card}\{g\in F\mid (P_k\otimes I) (y_g) \ne 0\}.
\]
Then
\[
 \left\|\sum d_j\right\|_{(p)} \le [(4 N_+(d))^{1/p} + 9\pi/8] \left\|\sum d_j\dot{\otimes} \bar d_j\right\|^{1/2}_{(p/2)}.
\]
\end{thm}

\begin{proof}
We argue exactly as for Theorem~\ref{thm11.1} except that we start instead from
\[
\|f\|^n_{(p)} = \|f\dot{\otimes} \cdots \dot{\otimes} f\|_{(2)}.\eqno\qed
\]
\renewcommand{\qed}{}\end{proof}

\begin{rem}
In particular, if $E\subset \Gamma$ is a subset of a \emph{commutative} group, let $Z_{p+}(\gamma,E)$ be the cardinality of the set of injective $g\colon \ [1,\ldots, n]\to E$ such that $\gamma = g(1)g(2)\ldots g(n)$ and let $Z_+(E) = \sup\{Z_{p+}(\gamma,E)\mid \gamma\in\Gamma\} < \infty$. Then for any finitely supported family $(b(t))_{t\in E}$ in $B$ we have \eqref{eq11.1}  with $Z_+(E)$ in place of $Z(E)$. Thus we obtain $OH$ also for the span of certain $\Lambda(p)$-sets originally considered by Rudin \cite{Ru}, which are \emph{not} $\Lambda(p)_{cb}$-sets in the sense of \cite{Har}.
\end{rem}

\section{Appendix}\label{sec13}

\indent 

The goal of this appendix is to clarify the relation between ``moments defined by pairings'' used in \S 9
(following \cite{Buch}) and the well known Wick formula. The latter
 (probably going back independently to Ito and Wick)   was used by Ito in connection with multiple Wiener integrals and Wiener chaos.  Although we reformulate them using tensor products,
 the results below are all well known.
  \\
We first consider the Gaussian case in a very general framework. Let $B$ be a real vector space. Let $X$ be a $B$-valued Gaussian variable. This means that for any ${\bb R}$-linear form $\xi\in B^*$, the real valued variable $\xi(X)$ is Gaussian with mean zero and variance equal to ${\bb E}\xi(X)^2$. If $B$ is a complex space, (e.g.\ if $B={\bb C}$) we may view it a fortiori as a real one and the previous notion still makes sense.

Let $X = (X_1,\ldots, X_n)$ be a Gaussian variable with values in $B^n$. Then $X_1 \otimes\cdots\otimes X_n$  is a random variable with values in $B^{\otimes n}$. When $n$ is odd its mean vanishes. Let us assume that $X = (X_1,\ldots, X_n)$ is defined on $(\Omega,{\cl A}, {\bb P})$ and that $n$ is even. \\
Let $\pi$ be a partition of $[1,\ldots, n]$ into $K$ blocks. We will define a $B^{\otimes n}$-valued random variable $X^{\otimes\pi}$ on $(\Omega,{\cl A}, {\bb P})^{\otimes K}$ as follows: \ Assume that the blocks of $\pi$ have been enumerated as $\alpha_1,\ldots,\alpha_{K}$. We define $\widehat\omega_j$ for $j=1,\ldots, n$ by setting $\widehat\omega_j = \omega_k$ if $j\in \alpha_k$. We then define
\[
X^{\otimes\pi} (\omega_1,\ldots,\omega_{K}) = X_1(\widehat\omega_1) \otimes\cdots\otimes X_{n}(\widehat\omega_{n}).
\]
Note that the distribution (and hence all the moments) of $X^{\otimes\pi}$ do not depend on the particular enumeration $(\alpha_1,\alpha_2,\ldots)$ chosen to define it. In particular, ${\bb E}(X^{\otimes\pi})$ depends only on $\pi$. We now may state 

\begin{pro}\label{pro10}
For any even integer $n$ 
\begin{equation}\label{lemeq1}
{\bb E}(X_1\otimes\cdots\otimes X_n) = \sum\nolimits_{\nu\in P_2(n)} {\bb E}(X^{\otimes\nu}).
\end{equation}
\end{pro}

\begin{proof} We will use the same trick as in  \cite[Prop. 1.5]{HT2}) to deduce the formula
 from the rotational invariance of Gaussian distribution.
Let $X(s) = (X_1(s),\ldots, X_n(s)$ be an i.i.d.\ sequence indexed by $s\in {\bb N}$ of copies of $X$. By the invariance of Gaussian distributions, the variable $\widehat X(s) = s^{-1/2}(X(1) +\cdots+ X(s))$ has the same distribution as $X$. Therefore for any $s$, we have
\[
{\bb E}(X_1 \otimes\cdots\otimes X_n) = {\bb E}(\widehat X_1(s) \otimes\cdots\otimes \widehat X_n(s))
\]
and hence
\begin{equation}\label{prfeq0}
{\bb E}(X_1 \otimes\cdots\otimes X_n)  = \lim_{s\to\infty} {\bb E}(\widehat X_1(s) \otimes\cdots\otimes \widehat X_n(s)).
\end{equation}
Let $E(s) = {\bb E}(\widehat X_1(s) \otimes\cdots\otimes \widehat X_n(s))$. We have
\[
E(s) = s^{-n/2} \sum\nolimits_g {\bb E}(X_1(g(1)) \otimes\cdots\otimes X_n(g(n))),
\]
where the sum runs over all functions $g\colon \ [1,\ldots, n]\to [1,\ldots, s]$. We claim that after elimination of all the (vanishing) odd terms and all the (asymptotically vanishing) terms such as 
\[
\frac1{s^2} \sum\nolimits_{t\le s} X_1(t) \otimes\cdots\otimes X_4(t)
\]
we find
\begin{equation}\label{prfeq00}
\lim\nolimits_{s\to\infty} E(s) = \sum\nolimits_{\nu\in P_2(n)} {\bb E}(X^{\otimes\nu}).
\end{equation}
Indeed, if we let
\[
t(g) = {\bb E}(X_1(g(1)) \otimes\cdots\otimes X_n(g(n))
\]
and if we denote by $\pi(g)$ the partition of $[1,\ldots, n]$ defined by $\bigcup_{k\le s} g^{-1}(k)$ we have (eliminating vanishing terms)
\[
E(s) = s^{-n/2} \sum\nolimits_{g\in A_s} t(g)
\]
where $A_s$ is the set of $g\colon \ [1,\ldots, n]\to [1,\ldots,s]$ such that $\pi(g)$ is a partition of $[1,\ldots, n]$ into blocks of even cardinality. For any such $\pi$, let
\[
G_s(\pi) = \{g\in A_s\mid \pi(g) = \pi\}.
\]
Let $B_s\subset A_s$ be the set of all $g$'s such that $\pi(g)$ is in $P_2(n)$ (i.e.\ is a partition into pairs) so that 
\[
B_s = \bigcup\nolimits_{\nu\in P_2(n)} G_s(\nu).
\]
Note that for any $g$ in $G_s(\pi)$ we have
\[
t(g) = {\bb E}(X^{\otimes\pi}).
\]
Let $P'_2(n)$ denote the set of partitions $\pi$ of $[1,\ldots, n]$ into blocks of even cardinality. We have
\[
E(s) =s^{-n/2} \sum\nolimits_{\pi \in P'_2(n)} |G_s(\pi)| {\bb E}(X^{\otimes\pi}).
\]
Note that $P_2(n) \subset P'_2(n)$.
Therefore
\begin{equation}\label{prfeq2}
E(s) = \sum\nolimits_{\nu\in P_2(n)} s^{-n/2} |G_s(\nu)| {\bb E}(X^{\otimes\nu}) + \sum\nolimits_{\pi\in P'_2(n)\backslash P_2(n)} s^{-n/2}|G_s(\pi)| {\bb E}(X^{\otimes\pi}).
\end{equation}
But now a simple counting argument shows that $|G_s(\nu)| = s(s-1) \ldots \left(s-\frac{n}2+1\right)\simeq s^{n/2}$ and hence $s^{-n/2} |G_s(\nu)|\to 1$, while for any $\pi$ in $P'_2(n)\backslash P_2(n)$ we have $s^{-n/2}|G_s(\pi)|\to 0$. Thus taking the limit when $s\to\infty$ in \eqref{prfeq2} yields our claim \eqref{prfeq00}. By \eqref{prfeq0} this completes the proof.
\end{proof}

To emphasize the connection with the classical Wick formula of which \eqref{lemeq1} is but an abstract form,
let us state:
\begin{cor} Let $X=(X_j)$ be a Gaussian sequence of real valued  random variables
(i.e. all their linear combinations are Gaussian).
Then
$$\EE(X_1\cdots X_n)=\sum_\nu \prod \langle X_{k_j} X_{\ell_j} \rangle $$
where the
sum runs over all partitions $\nu$  of $[1,\ldots, n]$ into pairs, the product runs
over all pairs $\{k_j,\ell_j\}$ ($j=1\cdots n/2$) of $\nu$, and the 
  scalar products are meant in $L_2$.
\end{cor}

\begin{cor}\label{acor1}
For any even integer $p$ and any $N\ge 1$, any sequence $X=(X_j)$ of i.i.d.\ Gaussian variables with values in the space $M_N$ of $N\times N$ (complex) matrices has its $p$-th moments defined by pairings.
(Here the moments are meant with respect to the functional $x\to {\bb E} \text{ \rm tr}(x)$.)
\end{cor}

\begin{proof}
Consider the ${\bb R}$-linear map $\varphi\colon \ M_N\otimes\cdots\otimes M_N\to {\bb C}$ defined by
\[
\varphi(x_1\otimes\cdots\otimes x_p) = \text{tr}(x_1x^*_2 \ldots x_{p-1}x^*_p).
\]
Let $k=(k_1,\ldots, k_p)$ where $k_1,k_2,\ldots, k_p$ are positive integers. Let $Y_k = (X_{k_1},\ldots, X_{k_p})$. Applying \eqref{prfeq0} to $X_{k_1} \otimes\cdots\otimes X_{k_p}$ we find
\[
{\bb E}\varphi(X_{k_1} \otimes\cdots\otimes X_{k_p}) = \sum\nolimits_{\nu\in P_2(p)} {\bb E}\varphi(Y^{\otimes\nu}_k).
\]
Since the variables $X_1,X_2,...$ are assumed independent, we have ${\bb E}\varphi(Y^{\otimes\nu}_k)=0$ except possibly when $k\sim \nu$ (indeed, if $\{i,j\}$ is a block of $\nu$ and $k_i\ne k_j$ then the entries of the ${k_i}$ factor of $Y^{\otimes\nu}_k$ are orthogonal to those of the ${k_j}$ factor and independent of all the other factors  of $Y^{\otimes\nu}_k$). Moreover, since $X_1,X_2,...$ all  have the same distribution, it is easy to check that the distribution of $Y_k^{\otimes\nu}$ depends only on $\nu$ and not on $k$. Thus we obtain
\[
{\bb E}\varphi(X_{k_1} \otimes\cdots\otimes X_{k_p}) = \sum\nolimits_{\nu\sim k} \psi(\nu)
\]
with $\psi(\nu) = {\bb E}\varphi(Y^{\otimes\nu})$, with $Y=(X_1,\ldots, X_n)$.
\end{proof}
\begin{rem} The preceding result is used in \cite{HT2} for the complex Gaussian random matrices $(Y^{(N)}_j)_{j\ge 1}$ appearing in Corollary \ref{cor9.13}. In that case, since $Y^{(N)}\otimes Y^{(N)}$ has mean zero, there is an extra cancellation: $\psi(\nu)=0$ for any
    partition  $\nu\in P_2(n)$  admitting a block with two indices of the same parity.
\end{rem}
More generally,   the same proof shows

\begin{cor}\label{acor2}
The preceding corollary is valid for any even integer $p$ for any i.i.d.\ Gaussian sequence with values in $L_p(M,\tau)$ for any non-commutative measure space $(M,\tau)$. (Here the moments are meant with respect to $x\mapsto {\bb E}\tau(x)$.)
\end{cor}

 Using exactly the same method  but replacing stochastic independence  by freeness in the sense of \cite{VDN} and Gaussian by semi-circular (or ``free-Gaussian''), it is easy to extend the preceding Corollary to the free  case. More generally, we can use the $q$-Fock space ($-1\le q\le 1$) and the associated $q$-Gaussian variables described in \cite{BS1,BS2,BKS}.

Fix $q$ with $-1\le q<1$. Given a complex Hilbert space ${\cl H}$, we denote by ${\cl F}_q({\cl H})$ the $q$-Fock space associated to ${\cl H}$. Let us assume that ${\cl H}$ is the complexification of a real Hilbert space $H$ so that  ${\cl H} = H+iH$. For simplicity we assume $H=\ell_2({\bb R})$. To any real Hilbert subspace $K\subset H$ we can associate (following \cite{BKS}) a von~Neumann algebra $\Gamma_q(K)$, so that we have natural embeddings $\Gamma_q(K_1)\subset \Gamma_q(K_2)$ when $K_1\subset K_2$. Moreover $\Gamma_q(H)$ is equipped with a normalized trace (faithful and normal) denoted by $\tau_q$, that we may restrict to $\Gamma_q(K)$ to view the latter as a non-commutative probability space.

For any $h\in H$ we denote by $a^*(h)$ (resp. $a(h)$) the $q$-creation (resp. $q$-annihilation) operator on ${\cl F}_q({\cl H})$ and we  let $g_q(h) = a(h)+ a^*(h)$. By definition, the von~Neumann algebras $\Gamma_q(K)$ is generated by $\{g_q(h)\mid h\in K\}$. We  will say that a family $X_1,\ldots X_n$ in $\Gamma_q(H)$ is $q$-independent if there are mutually orthogonal real subspaces $K_j\subset H$ such that $X_j\in \Gamma_q(K_j)$.

Let $B=B(\ell_2)$. More generally, consider $x_1,\ldots, x_n$ in $B\otimes \Gamma_q(H)$. We will say that $x_1,\ldots, x_n$ are $q$-independent if there are $K_j$ as above such that $x_j\in B\otimes \Gamma_q(K_j)$ for all $j=1,\ldots, n$. The elements of $g_q(H) = \{g_q(h) \mid h\in H\}$ will be called $q$-Gaussian.

More generally, an element $x\in B\otimes \Gamma_q(H)$ will be called $q$-Gaussian if $x\in B\otimes g_q(H)$. The $q$-Gaussian elements satisfy an analogue (called ``second quantization'') of the rotational invariance of Gaussian distributions: For any $\bb R$-isometry $T\colon \ K\to H$   the families $\{g_q(t)\mid t\in K\}$ and $\{g_q(Tt)\mid t\in K\}$ have identical distributions. Here ``same distribution'' means equality of  the  moments of all non-commutative monomials. We will denote by $\widetilde T\colon \ B\otimes g_q(K)\to B\otimes g_q(H)$ the linear map taking $b\otimes g_q(t)$ to $b\otimes g_q(Tt)$ $(b\in B, t\in K)$. In particular, if $x\in g_q(K)$ and if $K_1=K_2=\cdots= K_s=K$ we may use the isometry $u_s\colon \ K\to K_1\oplus\cdots\oplus K_s\subset H$ defined by $u_s(x)=s^{-1/2}(x\oplus\cdots\oplus x)$ in order to define elements $x_j$ in $g_q(K_j)$ each with the same distribution as $x$ such that $x \overset{\text{dist}}{=} s^{-1/2}(x_1+\cdots+ x_s)$. 

Let $x_1,\ldots, x_n$ be any sequence in $B\otimes g_q(H)$ and let $\pi$ be a partition. For any block $\alpha$ of $\pi$ we give ourselves an isometry $u_\alpha\colon \ H\to H_\alpha\subset H$ where $H_\alpha$ are mutually orthogonal (real) Hilbert subspaces.
Then we define a sequence $(y_1,\ldots, y_n)$ in $B\otimes g_q(H)$ by setting
\begin{equation}
y_j = \widetilde u_\alpha x_j.\tag*{$\forall j\in\alpha$}
\end{equation}
It is not hard to check that $\widehat\tau(y_1\dot\otimes\cdots\dot\otimes y_n)$ depends only on $x = (x_1,\ldots, x_n)$ and $\pi$ (and not on the $u_\alpha$'s). Therefore we may set
\[
\widehat\tau(x^\pi) \overset{\text{def}}{=} \widehat\tau(y_1 \dot\otimes\cdots\dot\otimes y_n).
\]
As before, by symmetry $\widehat\tau_q(x_1 \dot\otimes\cdots\dot\otimes x_n) = 0$ for all odd $n$.\\
The $q$-Gaussian analogue of \eqref{lemeq1} is as follows:
\begin{pro} For any even integer $n$ and  any $n$-tuple
  $x_1\ldots, x_n$    in $B\otimes g_q(H)$ ($-1\le q<1$), we have   
\[
\widehat\tau_q(x_1 \dot\otimes\cdots\dot\otimes x_n) = \sum\nolimits_{\nu\in P_2(n)} \widehat\tau_q(x^\nu).
\]
\end{pro}
\begin{proof}
With the preceding ingredients, this can be proved exactly as Proposition~\ref{pro10} above.
\end{proof}
In particular, replacing $B$ by $\bb C$, we find
\begin{cor}\label{acor3}
Any $q$-Gaussian sequence $(x_j)$ in the sense of \cite[Def.~3.3]{BKS} with covariance equal to the identity matrix has $p$-th moment, defined by parings for any even integer $p$.
\end{cor}

\n\textbf{Acknowledgment.}  I am extremely grateful to Quanhua Xu for a careful reading that led to numerous corrections and improvements.


\begin{thebibliography}{99}

\bibitem{A} 
 G. Andrews, \emph{The theory of partitions},
Cambridge Univ. Press, 1984.

 \bibitem{BKS} M. Bo\.zejko,   B. KŸmmerer,   and  R. Speicher, Roland . $q$-Gaussian processes: non-commutative and classical aspects.
 \emph{ Comm. Math. Phys.}  185  (1997),   129-154.



  \bibitem{BS1}  M. Bo\.zejko,   R. Speicher,   An example of a generalized Brownian motion.
 \emph{ Comm. Math. Phys.}  137  (1991),    519-531.

  \bibitem{BS2}  M. Bo\.zejko,   R. Speicher,   Interpolations between bosonic and fermionic relations given by
 generalized Brownian motions.
  \emph{Math. Z.}  222  (1996),    135-159.

\bibitem{Buch0} A. Buchholz, Norm of convolution by operator-valued functions on free groups. Proc. Amer. Math. Soc.  {  127} (1999), no. 6, 1671Ð1682. 
\bibitem{Buch} A. Buchholz,   Operator Khintchine inequality in non-commutative probability.  \emph{Math. Ann.}  { 319}  (2001),    1-16.

\bibitem{Buch2} A. Buchholz, 
Optimal Constants in Khintchine Type Inequalities for Fermions, Rademachers and q-Gaussian Operators.
    \emph{Bull. Polish Acad. Sci. Math.} {  53} (2005), 315-321.
    
    
 \bibitem{Bur}   D. Burkholder, Martingales and singular integrals in Banach spaces.
 Handbook of the geometry of Banach spaces, Vol. I, 
 233--269, North-Holland, Amsterdam,  2001. 


    \bibitem{CoMa} B. Collins and C. Male, The strong asymptotic freeness of Haar and deterministic
matrices, To appear.
    
\bibitem{Co} M. Cotlar,  A unified theory of Hilbert transforms and ergodic theorems.  \emph{Rev. Mat. Cuyana} {  1}  (1955), 105-167 (1956). 

\bibitem{ER}  E. G. Effros and  Z-J.Ruan,  \emph{Operator spaces},  Oxford University Press, New York, 2000. xvi+363 pp.

\bibitem{HT2} 
U. Haagerup and S. Thorbj{\o}rnsen, Random matrices and $K$-theory for exact $C^*$-algebras,  \emph{Doc. Math.} { 4} (1999), 341-450 (electronic).

\bibitem{Har} A.  Harcharras, Fourier analysis, Schur multipliers on $S\sp p$ and non-commutative $\Lambda(p)$-sets.  \emph{Studia Math.}  { 137}  (1999),  no. 3, 203-260.

  \bibitem{Har2} A. Harcharras,  S. Neuwirth  and  K.  Oleszkiewicz Lacunary matrices.
 \emph{Indiana Univ. Math. J.}  {  50}  (2001),  no. 4, 1675-1689.
 
 \bibitem{Ju}  M. Junge,  Doob's inequality for non-commutative martingales.
  \emph{J. Reine Angew. Math.}  549  (2002), 149-190.

\bibitem{JPX}  M. Junge, J. Parcet and Q. Xu, Rosenthal type inequalities for free chaos.
  \emph{Ann. Probab.}  35  (2007),  no. 4, 1374-1437.
 
\bibitem{JX0}   M. Junge and Q. Xu, Noncommutative Burkholder/Rosenthal inequalities.
  \emph{Ann. Probab.}  31  (2003),  no. 2, 948-995.
 \bibitem{JX1}  M. Junge and Q. Xu, On the best constants in some non-commutative martingale
 inequalities.
 \emph{ Bull. London Math. Soc.}  37  (2005),  no. 2, 243-253.

 \bibitem{JX2}   M. Junge and Q. Xu, Noncommutative Burkholder/Rosenthal inequalities. II.
 Applications.
 \emph{ Israel J. Math. } 167  (2008), 227-282.

 
\bibitem{LP}
F. Lust-Piquard, In\'egalit\'es de Khintchine
daus $C_p$ $(1<p<\infty)$, \emph{C.R. Acad. Sci. Paris} { 303} (1986),
289-292.
\bibitem{PRa}  J. Parcet and N. Randrianantoanina, 
 Gundy's decomposition for non-commutative martingales and
 applications.
  \emph{Proc. London Math. Soc.}    93  (2006),  no. 1, 227--252.


\bibitem{P3}
  G. Pisier,  The operator Hilbert space ${\rm OH}$, complex interpolation and tensor norms, \emph{Mem. Amer. Math. Soc.} { 122} (1996), no. 585. 


\bibitem{P4}
  G. Pisier, Noncommutative
 vector valued $L_p$-spaces and completely
$p$-summing maps,  \emph{Soc. Math. France. 
Ast\'erisque} { 237} (1998), vi+131 pp.. 
\bibitem{P5}
  G. Pisier, An inequality for $p$-orthogonal
  sums in non-commutative ${L_p}$, \emph{Illinois
J. Math.} { 44} (2000), 901-923.
\bibitem{P6}
  G. Pisier, \emph{Introduction to operator space theory},  Cambridge University Press, Cambridge, 2003.   
  \bibitem{PX1}  
G. Pisier and Q. Xu, Non-commutative martingale inequalities, \emph{Comm. Math. Phys.} { 189} (1997),  no.~3, 667-698. 
\bibitem{PX3} 
G. Pisier and Q. Xu,  Non-commutative $L_p$-spaces, \emph{Handbook of the Geometry of Banach Spaces, Vol. II}, North-Holland,
Amsterdam, 2003. 

\bibitem{Ra1}  N. Randrianantoanina, Non-commutative martingale transforms.
  \emph{J. Funct. Anal.}  194  (2002),  no. 1, 181--212.

\bibitem{Ra2}  N. Randrianantoanina, Square function inequalities for non-commutative martingales.
 \emph{ Israel J. Math.}  140  (2004), 333--365.

\bibitem{Ra3}  N. Randrianantoanina, A weak type inequality for non-commutative martingales and
 applications.
  \emph{Proc. London Math. Soc.}    91  (2005),  no. 2, 509--542.
\bibitem{Ra4}  N. Randrianantoanina, Conditioned square functions for noncommutative martingales.
 \emph{ Ann. Probab. } 35  (2007),  no. 3, 1039--1070.

\bibitem{RX} \'E. Ricard and Q. Xu,  Khintchine type inequalities for reduced free products and applications.  \emph{J. Reine Angew. Math. } 599 (2006), 27-59.

\bibitem{Ru} W. Rudin, Trigonometric series with gaps.  \emph{J. Math. Mech.} {  9} (1960) 203-228.

\bibitem{SS} A. Sinclair and R.R. Smith, \emph{Hochschild cohomology of von Neumann algebras},  Cambridge University Press, Cambridge, 1995. 
\bibitem{VDN} D. Voiculescu, K. Dykema and A.  Nica, \emph{Free random variables},  American Mathematical Society, Providence, RI, 1992.


\bibitem{X2}  Q. Xu,  Embedding of Cq and Rq into noncommutative $L_p$-spaces, $1\le p<q\le 2$.
  \emph{Math. Ann. } 335  (2006) 109--131.
\bibitem{X3}  Q. Xu,  Operator-space Grothendieck inequalities for noncommutative $L_p$-spaces.
  \emph{Duke Math. J.}  131  (2006),    525--574.


\end{thebibliography}
\end{document}